\documentclass[12pt, twoside]{article}
\usepackage{amsmath,amsthm,amssymb}
\usepackage{times}
\usepackage{enumerate}
\usepackage{amssymb}
\usepackage{graphicx}

\def\titlerunning#1{\gdef\titrun{#1}}
\makeatletter
\def\author#1{\gdef\autrun{\def\and{\unskip, }#1}\gdef\@author{#1}}

\makeatother

\def\keywords#1{\par\medskip
\noindent\textbf{Keywords.} #1}
\def\subjclass#1{\par\smallskip
\noindent\textbf{MSC (2010):} #1}

\newtheorem{thm}{Theorem}[section]

\newtheorem{lem}[thm]{Lemma}

\newtheorem{prop}[thm]{Proposition}

\theoremstyle{definition}

\newtheorem{rem}[thm]{Remark}
\newtheorem{exa}[thm]{Example}

\numberwithin{equation}{section}

\newtheorem*{notations}{Notations}

\DeclareMathOperator*{\esssup}{ess\,sup}
\DeclareMathOperator*{\essinf}{ess\,inf}

\begin{document}

\baselineskip=17pt

\titlerunning{Compactness and existence in weighted Sobolev spaces, I}

\title{Compactness and existence results in weighted Sobolev spaces of radial functions, Part I: Compactness}

\author{
Marino Badiale\thanks{Partially supported by the PRIN2009 grant ``Critical Point Theory and
Perturbative Methods for Nonlinear Differential Equations''},
\ \ Michela Guida,
\ \ Sergio Rolando\footnotemark[1]}

\date{
\begin{footnotesize}
\emph{
Dipartimento di Matematica ``Giuseppe Peano'' \smallskip\\
Universit\`{a} degli Studi di Torino, Via Carlo Alberto 10, 10123 Torino, Italy \\
e-mail:} marino.badiale@unito.it, michela.guida@unito.it, sergio.rolando@unito.it
\end{footnotesize}
}
\maketitle

\begin{abstract}
Given two measurable functions $V\left(r \right)\geq 0$ and $K\left(r\right)> 0$, $r>0$,
we define the weighted spaces 
\[
H_{V}^{1}=\left\{ u\in D^{1,2}(\mathbb{R}^{N}):\int_{\mathbb{R}^{N}}V\left( \left|
x\right| \right) u^{2}dx<\infty \right\} ,\quad L_{K}^{q}=L^{q}(\mathbb{R}%
^{N},K\left( \left| x\right| \right) dx)
\]
and study the compact embeddings of the radial subspace of $H_{V}^{1}$ into $%
L_{K}^{q_{1}}+L_{K}^{q_{2}}$, and thus into $L_{K}^{q}$ ($%
=L_{K}^{q}+L_{K}^{q}$) as a particular case. 
Both super- and sub-quadratic exponents $q_{1}$, $q_{2}$ and $q$ are considered.
Our results do not require any compatibility between how the potentials $V$ and $K$ behave at the origin and at infinity, and 
essentially rely on power type estimates of their relative growth, not of the potentials separately.
Applications to existence results for nonlinear elliptic problems like 
\[
-\triangle u+V\left( \left| x\right| \right) u=f\left( \left| x\right|
,u\right) \quad \text{in }\mathbb{R}^{N},\quad u\in H_{V}^{1},
\]
will be given in a forthcoming paper.

\end{abstract}

\section{Introduction}

Consider the nonlinear elliptic radial equation 
\begin{equation}
-\triangle u+V\left( \left| x\right| \right) u=K\left( \left| x\right|
\right) f\left( u\right) \quad \text{in }\mathbb{R}^{N},  \label{EQ}
\end{equation}
where $N\geq 3$, $f:\mathbb{R}\rightarrow \mathbb{R}$ is a continuous nonlinearity
satisfying $f\left( 0\right) =0$ and $V\geq 0,K>0$ are given potentials. The
motivation of this paper is concerned with the problem of the existence of
non-zero non-negative solutions to equation (\ref{EQ}) in the following weak
sense: we say that $u\in H_{V}^{1}$ is a \textit{weak solution}\emph{\ }to (%
\ref{EQ}) if 
\begin{equation}
\int_{\mathbb{R}^{N}}\nabla u\cdot \nabla h\,dx+\int_{\mathbb{R}^{N}}V\left(
\left| x\right| \right) uh\,dx=\int_{\mathbb{R}^{N}}K\left( \left| x\right|
\right) f\left( u\right) h\,dx\quad \text{for all }h\in H_{V}^{1},
\label{weak solution}
\end{equation}
where 
\begin{equation}
H_{V}^{1}:=H_{V}^{1}\left( \mathbb{R}^{N}\right) :=\left\{ u\in D^{1,2}\left( 
\mathbb{R}^{N}\right) :\int_{\mathbb{R}^{N}}V\left( \left| x\right| \right)
u^{2}dx<\infty \right\}  \label{H:=}
\end{equation}
is the energy space associated to the linear part of the equation, which is
a Hilbert space with respect to the following inner product and related
norm: 
\[
\left( u\mid v\right) :=\int_{\mathbb{R}^{N}}\left( \nabla u\cdot \nabla
v+V\left( \left| x\right| \right) uv\right) dx,\quad
\left\| u\right\|^2:=\int_{\mathbb{R}^{N}}\left(\left| \nabla u\right| ^{2}
+V\left( \left| x\right| \right) u^{2}\right)dx.
\]
By well known arguments, such solutions lead to special solutions (\textit{%
solitary waves} and \textit{solitons}) for several nonlinear classical field
theories, such as Schr\"{o}dinger and Klein-Gordon equations (see e.g. \cite
{Be09,YangY, BBR1}). In this respect, since the early studies of \cite
{Beres-Lions,Strauss,Floer-Wein,Rabi92}, equation (\ref{EQ}) has been
massively addressed in the mathematical literature, recently focusing on the
case of $V$ possibly vanishing at infinity, i.e., 
$\liminf_{\left|x\right| \rightarrow \infty }V\left( \left| x\right| \right) =0$ 
(some first results on such a case can be found in \cite
{Ambr-Fel-Malch,BR,Be-Gr-Mic,Be-Gr-Mic 2}; for more recent references, see
e.g. the bibliography of \cite{Alves-Souto-13,SuTian12}).

The natural approach in studying equation (\ref{EQ}) is variational, since
its weak solutions are (at least formally) critical points of the Euler
functional 
\begin{equation}
I\left( u\right) :=\frac{1}{2}\left\| u\right\| ^{2}-\int_{\mathbb{R}%
^{N}}K\left( \left| x\right| \right) F\left( u\right) dx,  \label{I:=}
\end{equation}
where $F\left( t\right) :=\int_{0}^{t}f\left( s\right) ds$. Then the problem
of existence is easily solved if $V$ does not vanish at infinity and $K$ is
bounded. Indeed, in this case, introducing the weighted Lebesgue space 
\begin{equation}
L_{K}^{q}:=L_{K}^{q}\left( \mathbb{R}^{N}\right) :=L^{q}\left( \mathbb{R}%
^{N},K\left( \left| x\right| \right) dx\right) ,  \label{L_K:=}
\end{equation}
one easily sees that the embeddings $H_{V}^{1}\hookrightarrow H^{1}(\mathbb{R}%
^{N})$ and $L^{q}(\mathbb{R}^{N})\hookrightarrow L_{K}^{q}$ are continuous, so
that, by the well known theory of $H^{1}(\mathbb{R}^{N})$, the embedding 
\[
H_{V}^{1}\left( \mathbb{R}^{N}\right) \hookrightarrow L_{K}^{q}\left( \mathbb{R}%
^{N}\right) ,\qquad 2<q<2^{*}:=\frac{2N}{N-2}, 
\]
is continuous and becomes compact if restricted to the radial subspace of $%
H_{V}^{1}$, namely 
\begin{equation}
H_{V,\mathrm{r}}^{1}:=H_{V,\mathrm{r}}^{1}\left( \mathbb{R}^{N}\right) :=%
\overline{\left\{ u\in C_{\mathrm{c,rad}}^{\infty }\left( \mathbb{R}^{N}\right)
:\int_{\mathbb{R}^{N}}V\left( \left| x\right| \right) u^{2}dx<\infty \right\} }%
^{\ H_{V}^{1}(\mathbb{R}^{N})}.  \label{Hr:=}
\end{equation}
Therefore, assuming that $f$ grows as a super-linear and subcritical power,
the functional $I$ is of class $C^{1}$ on $H_{V}^{1}$ (because its
non-quadratic part is of class $C^{1}$ on $L_{K}^{q}$) and, under some
additional conditions on the nonlinearity, by-now almost standard in the
literature, the restriction of $I$ to $H_{V,\mathrm{r}}^{1}$ has a
mountain-pass geometry and satifies the Palais-Smale condition, so that it
has a non-zero critical point by the Mountain-Pass Theorem \cite{Ambr-Rab}.
Such a critical point is actually a critical point of $I$, i.e., a weak
solution of (\ref{EQ}) in the sense of (\ref{weak solution}), thanks to the
Palais' Principle of Symmetric Criticality \cite{Palais}, which applies
since $I$ is rotationally invariant and of class $C^{1}$ on $H_{V}^{1}$.

If $V$ vanishes at infinity, the space $H_{V}^{1}$ is no more necessarily
contained in $L^{2}(\mathbb{R}^{N})$ and thus the embedding properties of $%
H^{1}(\mathbb{R}^{N})$ become useless, so that the above scheme fails in
essentially two points: the compactness properties of the radial subspace of 
$H^{1}(\mathbb{R}^{N})$ are no more avaliable and a growth condition of the
form $\left| F\left( u\right) \right| \leq \left( \mathrm{const.}\right)
\left| u\right| ^{q}$, $q\neq 2^{*}$, does not ensure the differentiability,
not even the finiteness, of $I$ on $H_{V}^{1}$. A possible way of saving the
scheme, then, is to face and solve the following problems:

\begin{itemize}
\item[(i)]  find a compact embedding of $H_{V,\mathrm{r}}^{1}$ into a space $%
X$ where the non-quadratic part of $I$ is of class $C^{1}$;

\item[(ii)]  prove a simmetric criticality type result, ensuring that the
critical points of $I_{\mid H_{V,\mathrm{r}}^{1}}$ are weak solutions of
equation (\ref{EQ}) in the sense of definition (\ref{weak solution}).
\end{itemize}

Problem (i) has been largely investigated in the literature for $X=L_{K}^{q}$
and, as far as we know, the more general results in this direction are the
ones recently obtained by Su, Wang and Willem \cite{Su-Wang-Will p}, Su and
Tian \cite{SuTian12}, and Bonheure and Mercuri \cite{BonMerc11} (for older
results, see the references in \cite{Su-Wang-Will p,SuTian12}; for related
results without symmetry assumptions, see \cite
{Alves-Souto-13,Bon-VanSchaft-10} and the references therein). In
particular, \cite{Su-Wang-Will p}, \cite{SuTian12} and \cite{BonMerc11}
respectively concern the cases $q\geq 2$, $1<q<2$ and $q>2$. We also observe
that \cite{Su-Wang-Will p,SuTian12} actually deal with the more general case
of Banach energy spaces $W_{V,\mathrm{r}}^{1,p}$, $1<p<N$, which we do not
study here. The spirit of the results of \cite{Su-Wang-Will p,SuTian12} (for 
$p=2$) is essentially the following: assuming that $V,K$ are continuous and
satisfy power type estimates of the form: 
\begin{equation}
\liminf_{r\rightarrow 0^{+}}\frac{V\left( r\right) }{r^{a_{0}}}>0,\quad
\liminf_{r\rightarrow +\infty }\frac{V\left( r\right) }{r^{a}}>0,\quad
\limsup_{r\rightarrow 0^{+}}\frac{K\left( r\right) }{r^{b_{0}}}<\infty
,\quad \limsup_{r\rightarrow +\infty }\frac{K\left( r\right) }{r^{b}}<\infty
,  \label{pow-estim}
\end{equation}
the authors find two limit exponents $\underline{q}=\underline{q}\left(
a,b\right) $ and $\overline{q}=\overline{q}\left( a_{0},b_{0}\right) $ such
that the embedding $H_{V,\mathrm{r}}^{1}\hookrightarrow L_{K}^{q}$ is
compact if $\underline{q}<q<\overline{q}$. The exponent $\underline{q}$ is
always defined, while $\overline{q}$ exists provided that suitable
compatibility conditions between $a_{0}$ and $b_{0}$ occur. Moreover, the
condition $\underline{q}<q<\overline{q}$ also asks for $\underline{q}<%
\overline{q}$, which is a further assumption: a compatibility is also
required between the behaviours of the potentials at zero and at infinity.
These results are extended in \cite{BonMerc11} by replacing (\ref{pow-estim}%
) with estimates on $V$ and $K$ in terms of a much wider class of comparison
functions than the powers of $r$ (the so-called \textit{Hardy-Dieudonn\'{e}
comparison class}). A price to pay for such a generality is that the compact
embedding $H_{V,\mathrm{r}}^{1}\hookrightarrow L_{K}^{q}$ is no more ensured
for a range of exponents $q$, but under assumptions which join $q$ and the
comparison functions together.

Motivated by a wide recent use of the sum of weighted Lebesgue spaces in
dealing with nonlinear problems (see the references in \cite{BPR}), the case 
$X=L_{K}^{q_{1}}+L_{K}^{q_{2}}$ has been considered in \cite{BPR},
generalizing a result of \cite{Be-F.2} and studying the compactness of the
embedding of $H_{V,\mathrm{r}}^{1}$ (and of $W_{V,\mathrm{r}}^{1,p}$, $1<p<N$%
) into $L_{K}^{q_{1}}+L_{K}^{q_{2}}$ for $V=0$ and $K$ satisfying power type
estimates from above at zero and infinity (see Example \ref{EX: BPR} below).

Here we investigate problem (i) for $X=L_{K}^{q_{1}}+L_{K}^{q_{2}}$ with $%
q_{1}$ and $q_{2}$ not necessarily different, and thus for $X=L_{K}^{q}$ ($%
=L_{K}^{q}+L_{K}^{q}$) as a particular case. Some properties of the $%
L_{K}^{q_{1}}+L_{K}^{q_{2}}$ spaces will be recalled in Section \ref{SEC:1}.
We consider $q_{1},q_{2}\in \left( 1,+\infty \right) $, so that both super-
and sub-quadratic cases are covered. Our embedding results are given by the
combination of Theorem \ref{THM(cpt)} with Theorems \ref{THM0}, \ref{THM1}, 
\ref{THM2} and \ref{THM3}, which give sufficient conditions in order to
apply Theorem \ref{THM(cpt)}. The spirit of the results is essentially the
following: assuming that the relative growth of the potentials satisfies
power type estimates of the form 
\[
\esssup_{r\ll 1}\frac{K\left( r\right) }{r^{\alpha _{0}}V\left(r\right) ^{\beta _{0}}}<+\infty ,\quad 
\esssup_{r\gg 1}\frac{K\left( r\right) }{r^{\alpha _{\infty }}V\left( r\right) ^{\beta _{\infty }}}%
<+\infty ,\quad \beta _{0},\beta _{\infty }\in \left[ 0,1\right] , 
\]
we find two open intervals, say $\mathcal{I}_{1}=\mathcal{I}_{1}\left(
\alpha _{0},\beta _{0}\right) $ and $\mathcal{I}_{2}=\mathcal{I}_{2}\left(
\alpha _{\infty },\beta _{\infty }\right) $, the one depending on the
behaviour of $V$ and $K$ at the origin, the other on their behaviour at
infinity, such that the embedding of $H_{V,\mathrm{r}}^{1}$ into $%
L_{K}^{q_{1}}+L_{K}^{q_{2}}$ is compact for $q_{1}\in \mathcal{I}_{1}$ and $%
q_{2}\in \mathcal{I}_{2}$. These intervals are independent from one another
and may not intersect; if they do, a compact embedding of $H_{V,\mathrm{r}%
}^{1}$ into $L_{K}^{q}$ ensues, by taking $q_{1}=q_{2}=q$.

The compactness results we present here generalize the ones of \cite
{Su-Wang-Will p,SuTian12,BPR} (for $p=2$) and are complementary to the ones
of \cite{BonMerc11}. The main novelties concern the decaying rates allowed
for the potentials and the independence between their behaviours at the
origin and at infinity.

As to the first issue, we do not require \textit{separate} estimates on $V$
and $K$, but only on their relative growth, so that potentials which do not
exhibit a power like behaviour, as prescribed in \cite{Su-Wang-Will
p,SuTian12}, are permitted (see Examples \ref{EX(nnP1)} and \ref{EX(nnP2)}).
Moreover, unlike in \cite{Su-Wang-Will p,SuTian12,BonMerc11}, we also allow
that $V$ vanishes identically in a neighbourhood of zero, or of infinity, or
both (see Remark \ref{RMK: suff12}.\ref{RMK: suff12-V^0} and Examples \ref
{EX: BPR} and \ref{EX(nnP2)}). On the other hand, many of the potentials
considered in \cite{BonMerc11} do not fall into the class studied here,
since they give rise to a ratio $K/V^{\beta }$ which behaves as a general
Hardy-Dieudonn\'{e} function and therefore cannot be estimated by a\emph{\ }%
power of $r$. However, we think that our arguments can be extended in order
to estimate $K/V^{\beta }$ by means of a wider class of functions than the
powers of $r$, such as, indeed, the Hardy-Dieudonn\'{e} comparison functions.

As far as the second issue is concerned, we avoid any compatibility
requirement between how the potentials behave at the origin and at infinity,
since the use of $L_{K}^{q_{1}}+L_{K}^{q_{2}}$ spaces, not simply of $%
L_{K}^{q}$, leads to the already mentioned independent intervals $\mathcal{I}%
_{1},\mathcal{I}_{2}$ for the exponents $q_{1},q_{2}$. This also provides
new compact embeddings for potentials which belong to the classes considered
in \cite{Su-Wang-Will p,SuTian12,BonMerc11} but escape their results, for
instance because $\underline{q}\geq \overline{q}$ (see Examples \ref{EX(ST)}%
, \ref{EX(SWW)} and \ref{EX(nnP1)}).

Besides these general considerations, it is worth observing that we also
improve the compact embeddings of \cite{SuTian12} (for $p=2$), i.e., $H_{V,%
\mathrm{r}}^{1}\hookrightarrow L_{K}^{q}$ for $q$ sub-quadratic, in three
further respects: for potentials satisfying the same assumptions of \cite{SuTian12},
we find that the embedding is compact for a wider interval of exponents than the
range $\underline{q}<q<\overline{q}$ obtained in \cite{SuTian12}, as well as
for cases in which $\underline{q}$ and $\overline{q}$ are not defined, or
are such that $\underline{q}\geq \overline{q}$. This is thoroughly
highlighted in Example \ref{EX(ST)}, even though in a particular case.

A precise comparison with the compactness result of \cite{BPR} (for $p=2$)
will be given in Example \ref{EX: BPR}.

As to problem (ii), it has been recently treated in \cite{BRpow} and \cite
{BPR} for particular potentials, respectively concerning nonlinearities
satisfying a single-power and a double-power growth condition (see also \cite
{BGR} for a related cylindrical case). We will study the problem in the
forthcoming paper \cite{BGRnext}, where we will prove a simmetric
criticality type result that contains the ones of \cite{BPR,BRpow} and, as
announced in \cite{BGRnonex}, it also completes the existence results of
other papers (e.g. \cite{SuTian12,Su-Wang-Will 2,Su-Wang-Will p}), where a
solution is found as a critical point of the restriction $I_{\mid H_{V,%
\mathrm{r}}^{1}}$ only.

The forthcoming paper \cite{BGRnext} will be also (and mainly) devoted to
applications of our embedding results to nonlinear elliptic problems like (%
\ref{EQ}). Further applications will be given in \cite{GR-boundedPS}, where
the problem of existence without the \textit{Ambrosetti-Rabinowitz growth
condition} is faced.

This paper is organized as follows. In Section \ref{SEC:MAIN} we state our
main results: a general result concerning the embedding properties of $H_{V,%
\mathrm{r}}^{1}$ into $L_{K}^{q_{1}}+L_{K}^{q_{2}}$ (Theorem \ref{THM(cpt)})
and some explicit conditions ensuring that the embedding is compact
(Theorems \ref{THM0}, \ref{THM1}, \ref{THM2} and \ref{THM3}). The general
result is proved in Section \ref{SEC:1}, the explicit conditions in Section 
\ref{SEC:2}. In Section \ref{SUB: es} we apply our results to some examples,
with a view to both illustrate how to use them in concrete cases and to
compare them with the literature mentioned above. The Appendix is devoted to
some detailed computations, displaced from Section \ref{SEC:2} for sake of
clarity.

\begin{notations}
We end this introductory section by collecting
some notations used in the paper.

\noindent $\bullet $ For every $R>0$, we set $B_{R}:=\left\{ x\in \mathbb{R}%
^{N}:\left| x\right| <r\right\} $.

\noindent $\bullet $ For any subset $A\subseteq \mathbb{R}^{N}$, we denote $%
A^{c}:=\mathbb{R}^{N}\setminus A$. If $A$ is Lebesgue measurable, $\left|
A\right| $ stands for its measure.

\noindent $\bullet $ $O\left( N\right) $ is the orthogonal group of $\mathbb{R}%
^{N}$.

\noindent $\bullet $ By $\rightarrow $ and $\rightharpoonup $ we
respectively mean \emph{strong} and \emph{weak }convergence.

\noindent $\bullet $ $\hookrightarrow $ denotes \emph{continuous} embeddings.

\noindent $\bullet $ $C_{\mathrm{c}}^{\infty }(\Omega )$ is the space of the
infinitely differentiable real functions with compact support in the open
set $\Omega \subseteq \mathbb{R}^{d}$; $C_{\mathrm{c,rad}}^{\infty }(\mathbb{R}%
^{N})$ is the radial subspace of $C_{\mathrm{c}}^{\infty }(\mathbb{R}^{N})$.

\noindent $\bullet $ If $1\leq p\leq \infty $ then $L^{p}(A)$ and $L_{%
\mathrm{loc}}^{p}(A)$ are the usual real Lebesgue spaces (for any measurable
set $A\subseteq \mathbb{R}^{d}$). If $\rho :A\rightarrow \left( 0,+\infty
\right) $ is a measurable function, then $L^{p}(A,\rho \left( z\right) dz)$
is the real Lebesgue space with respect to the measure $\rho \left( z\right)
dz$ ($dz$ stands for the Lebesgue measure on $\mathbb{R}^{d}$).

\noindent $\bullet $ $p^{\prime }:=p/(p-1)$ is the H\"{o}lder-conjugate
exponent of $p.$

\noindent $\bullet $ $D^{1,2}(\mathbb{R}^{N})=\{u\in L^{2^{*}}(\mathbb{R}%
^{N}):\nabla u\in L^{2}(\mathbb{R}^{N})\}$, $N\geq 3$, is the usual Sobolev
space, which identifies with the completion of $C_{\mathrm{c}}^{\infty }(%
\mathbb{R}^{N})$ with respect to the norm of the gradient; $D_{\mathrm{rad}%
}^{1,2}(\mathbb{R}^{N})$ is the radial subspace of $D^{1,2}(\mathbb{R}^{N})$; $%
D_{0}^{1,2}\left( B_{R}\right) $ is closure of $C_{\mathrm{c}}^{\infty
}\left( B_{R}\right) $ in $D^{1,2}(\mathbb{R}^{N})$.

\noindent $\bullet $ $2^{*}:=2N/\left( N-2\right) $ is the critical exponent
for the Sobolev embedding in dimension $N\geq 3$.

\end{notations}

\section{Main results \label{SEC:MAIN}}

Assume $N\geq 3$ and recall the definitions (\ref{H:=}), (\ref{Hr:=}) and (%
\ref{L_K:=}) of the spaces $H_{V}^{1}$, $H_{V,\mathrm{r}}^{1}$ and $%
L_{K}^{q} $. We will always require that the potentials $V$ and $K$ satisfy
the following basic assumptions:

\begin{itemize}
\item[$\left( \mathbf{V}\right) $]  $V:\left( 0,+\infty \right) \rightarrow
\left[ 0,+\infty \right] $ is a measurable function such that $V\in
L^{1}\left( \left( r_{1},r_{2}\right) \right) $ for some $r_{2}>r_{1}>0;$

\item[$\left( \mathbf{K}\right) $]  $K:\left( 0,+\infty \right) \rightarrow
\left( 0,+\infty \right) $ is a measurable function such that $K\in L_{%
\mathrm{loc}}^{s}\left( \left( 0,+\infty \right) \right) $ for some $s>\frac{%
2N}{N+2}.$
\end{itemize}

\noindent Assumption $\left( \mathbf{V}\right) $ implies that the spaces $%
H_{V}^{1}$ and $H_{V,\mathrm{r}}^{1}$ are nontrivial, while hypothesis $%
\left( \mathbf{K}\right) $ ensures that $H_{V,\mathrm{r}}^{1}$ is compactly
embedded into the weighted Lebesgue space $L_{K}^{q}(B_{R}\setminus B_{r})$
for every $1<q<\infty $ and $R>r>0$ (cf. Lemma \ref{Lem(corone)} below). In
what follows, the summability assumptions in $\left( \mathbf{V}\right) $ and 
$\left( \mathbf{K}\right) $ will not play any other role than this.

Given $V$ and $K$, we define the following functions of $R>0$ and $q>1$: 
\begin{eqnarray}
\mathcal{S}_{0}\left( q,R\right)&:=&
\sup_
{u\in H_{V,\mathrm{r}}^{1},\,
\left\| u\right\| =1  }
\int_{B_{R}}K\left( \left| x\right| \right)
\left| u\right| ^{q}dx,  \label{S_o :=}
\\
\mathcal{S}_{\infty }\left( q,R\right)&:=&
\sup_
{u\in H_{V,\mathrm{r}}^{1},\,
\left\| u\right\| =1  }
\int_{\mathbb{R}%
^{N}\setminus B_{R}}K\left( \left| x\right| \right) \left| u\right| ^{q}dx.
\label{S_i :=}
\end{eqnarray}
Clearly $\mathcal{S}_{0}\left( q,\cdot \right) $ is nondecreasing, $\mathcal{%
S}_{\infty }\left( q,\cdot \right) $ is nonincreasing and both of them can
be infinite at some $R$.

Our first result concerns the embedding properties of $H_{V,\mathrm{r}}^{1}$
into $L_{K}^{q_{1}}+L_{K}^{q_{2}}$ and relies on assumptions which are quite
general, sometimes also sharp (see claim (iii)), but not so easy to check.
More handy conditions ensuring these general assumptions will be provided by
the next results. Some recallings on the space $L_{K}^{q_{1}}+L_{K}^{q_{2}}$
will be given in Section \ref{SEC:1}.

\begin{thm}
\label{THM(cpt)}Let $N\geq 3$, let $V$, $K$ be as in $\left( \mathbf{V}%
\right) $, $\left( \mathbf{K}\right) $ and let $q_{1},q_{2}>1$.

\begin{itemize}
\item[(i)]  If 
\begin{equation}
\mathcal{S}_{0}\left( q_{1},R_{1}\right) <\infty \quad \text{and}\quad 
\mathcal{S}_{\infty }\left( q_{2},R_{2}\right) <\infty \quad \text{for some }%
R_{1},R_{2}>0,  
\tag*{$\left( {\cal S}_{q_{1},q_{2}}^{\prime }\right) $}
\end{equation}
then $H_{V,\mathrm{r}}^{1}(\mathbb{R}^{N})$ is continuously embedded into $%
L_{K}^{q_{1}}(\mathbb{R}^{N})+L_{K}^{q_{2}}(\mathbb{R}^{N})$.

\item[(ii)]  If 
\begin{equation}
\lim_{R\rightarrow 0^{+}}\mathcal{S}_{0}\left( q_{1},R\right)
=\lim_{R\rightarrow +\infty }\mathcal{S}_{\infty }\left( q_{2},R\right) =0, 
\tag*{$\left({\cal S}_{q_{1},q_{2}}^{\prime \prime }\right) $}
\end{equation}
then $H_{V,\mathrm{r}}^{1}(\mathbb{R}^{N})$ is compactly embedded into $%
L_{K}^{q_{1}}(\mathbb{R}^{N})+L_{K}^{q_{2}}(\mathbb{R}^{N})$.

\item[(iii)]  If $K\left( \left| \cdot \right| \right) \in L^{1}(B_{1})$ and 
$q_{1}\leq q_{2}$, then conditions $\left( \mathcal{S}_{q_{1},q_{2}}^{\prime
}\right) $ and $\left( \mathcal{S}_{q_{1},q_{2}}^{\prime \prime }\right) $
are also necessary to the above embeddings.
\end{itemize}
\end{thm}

Observe that, of course, $(\mathcal{S}_{q_{1},q_{2}}^{\prime \prime })$
implies $(\mathcal{S}_{q_{1},q_{2}}^{\prime })$. Moreover, these assumptions
can hold with $q_{1}=q_{2}=q$ and therefore Theorem \ref{THM(cpt)} also
concerns the embedding properties of $H_{V,\mathrm{r}}^{1}$ into $L_{K}^{q}$%
, $1<q<\infty $.\smallskip

We now look for explicit conditions on $V$ and $K$ implying $(\mathcal{S}%
_{q_{1},q_{2}}^{\prime \prime })$ for some $q_{1}$ and $q_{2}$. More
precisely, we will ensure $(\mathcal{S}_{q_{1},q_{2}}^{\prime \prime })$
through a more stringent condition involving the following functions of $R>0$
and $q>1$: 
\begin{eqnarray}
\mathcal{R}_{0}\left( q,R\right)&:=&
\sup_
{
u\in H_{V,\mathrm{r}}^{1},\,h\in H_{V}^{1},\,
\left\| u\right\| =\left\| h\right\| =1 
}%
\,\int_{B_{R}}K\left( \left| x\right| \right) \left| u\right| ^{q-1}\left|
h\right| dx,  \label{N_o} \\
\mathcal{R}_{\infty }\left( q,R\right)&:= &
\sup_
{
u\in H_{V,\mathrm{r}}^{1},\,h\in H_{V}^{1},\, \left\| u\right\| =\left\| h\right\| =1  
}
\,\int_{\mathbb{R}^{N}\setminus B_{R}}K\left( \left| x\right| \right) \left|
u\right| ^{q-1}\left| h\right| dx.  \label{N_i}
\end{eqnarray}
Note that $\mathcal{R}_{0}\left( q,\cdot \right) $ is nondecreasing, $%
\mathcal{R}_{\infty }\left( q,\cdot \right) $ is nonincreasing and both can
be infinite at some $R$. Moreover, for every $\left( q,R\right) $ one has $%
\mathcal{S}_{0}\left( q,R\right) \leq \mathcal{R}_{0}\left( q,R\right) $ and 
$\mathcal{S}_{\infty }\left( q,R\right) \leq \mathcal{R}_{\infty }\left(
q,R\right) $, so that $(\mathcal{S}_{q_{1},q_{2}}^{\prime \prime })$ is a
consequence of the following, stronger condition: 
\begin{equation}
\lim_{R\rightarrow 0^{+}}\mathcal{R}_{0}\left( q_{1},R\right)
=\lim_{R\rightarrow +\infty }\mathcal{R}_{\infty }\left( q_{2},R\right) =0. 
\tag*{$\left( {\cal R}_{q_{1},q_{2}}^{\prime \prime }\right) $}
\end{equation}
In Theorems \ref{THM0} and \ref{THM3} we will find ranges of exponents $%
q_{1} $ such that $\lim_{R\rightarrow 0^{+}}\mathcal{R}_{0}\left(q_{1},R\right)$ $=0$,
while in Theorems \ref{THM1} and \ref{THM2} we will do the same for exponents $q_{2}$ such that
$\lim_{R\rightarrow +\infty}\mathcal{R}_{\infty }\left( q_{2},R\right) =0$.
Condition $(\mathcal{R}_{q_{1},q_{2}}^{\prime \prime })$ 
then follows by joining Theorem \ref{THM0} or \ref{THM3} with Theorem \ref{THM1} or \ref{THM2}.

The reason why we introduce the functions $\mathcal{R}_{0}$ and $\mathcal{R}%
_{\infty }$ is just a matter of future convenience: some of the results of 
\cite{BGRnext} will need assumption $(\mathcal{R}_{q_{1},q_{2}}^{\prime
\prime })$ and therefore we provide, already at this stage, sufficient
conditions in order that $(\mathcal{R}_{q_{1},q_{2}}^{\prime \prime })$
holds. This does not affect our compactness results, since such sufficient
conditions are exactly the same under which our arguments ensure $(\mathcal{S%
}_{q_{1},q_{2}}^{\prime \prime })$.

The fact that Theorem \ref{THM(cpt)} allows the case of $V\left( r\right)=+\infty $
on a positive measure set is for future convenience as well,
since it will be used in \cite{BGRnext} to easily deduce existence results
on bounded or exterior radial domains. Such a generality for $V$ is not so
relevant in the next results, so we will assume hereafter that $V$ is finite
almost everywhere.\smallskip

For $\alpha \in \mathbb{R}$ and $\beta \in \left[ 0,1\right] $, we define two
functions $\alpha ^{*}\left( \beta \right) $ and $q^{*}\left( \alpha ,\beta
\right) $ by setting 
\[
\alpha ^{*}\left( \beta \right) :=\max \left\{ 2\beta -1-\frac{N}{2},-\left(
1-\beta \right) N\right\} =\left\{ 
\begin{array}{ll}
2\beta -1-\frac{N}{2}\quad \smallskip & \text{if }0\leq \beta \leq \frac{1}{2%
} \\ 
-\left( 1-\beta \right) N & \text{if }\frac{1}{2}\leq \beta \leq 1
\end{array}
\right. 
\]
and 
\[
q^{*}\left( \alpha ,\beta \right) :=2\frac{\alpha -2\beta +N}{N-2}. 
\]
Note that $\alpha ^{*}\left( \beta \right) \leq 0$ and $\alpha ^{*}\left(
\beta \right) =0$ if and only if $\beta =1$.

The following Theorems \ref{THM0} and \ref{THM1} only rely on a power type
estimate of the relative growth of the potentials and do not require any
other separate assumption on $V$ and $K$ than $\left( \mathbf{V}\right) $
and $\left( \mathbf{K}\right) $, including the case $V\left( r\right) \equiv
0$ (see Remark \ref{RMK: suff12}.\ref{RMK: suff12-V^0}).

\begin{thm}
\label{THM0}Let $N\geq 3$ and let $V$, $K$ be as in $\left( \mathbf{V}%
\right) $, $\left( \mathbf{K}\right) $ with $V\left( r\right) <+\infty $.
Assume that there exists $R_{1}>0$ such that 
\begin{equation}
\esssup_{r\in \left( 0,R_{1}\right) }\frac{K\left( r\right) }{%
r^{\alpha _{0}}V\left( r\right) ^{\beta _{0}}}<+\infty \quad \text{for some }%
0\leq \beta _{0}\leq 1\text{~and }\alpha _{0}>\alpha ^{*}\left( \beta
_{0}\right) .  \label{esssup in 0}
\end{equation}
Then $\displaystyle \lim_{R\rightarrow 0^{+}}\mathcal{R}_{0}\left(
q_{1},R\right) =0$ for every $q_{1}\in \mathbb{R}$ such that 
\begin{equation}
\max \left\{ 1,2\beta _{0}\right\} <q_{1}<q^{*}\left( \alpha _{0},\beta
_{0}\right) .  \label{th1}
\end{equation}
\end{thm}

\begin{thm}
\label{THM1}Let $N\geq 3$ and let $V$, $K$ be as in $\left( \mathbf{V}%
\right) $, $\left( \mathbf{K}\right) $ with $V\left( r\right) <+\infty $.
Assume that there exists $R_{2}>0$ such that 
\begin{equation}
\esssup_{r>R_{2}}\frac{K\left( r\right) }{r^{\alpha _{\infty
}}V\left( r\right) ^{\beta _{\infty }}}<+\infty \quad \text{for some }0\leq
\beta _{\infty }\leq 1\text{~and }\alpha _{\infty }\in \mathbb{R}.
\label{esssup all'inf}
\end{equation}
Then $\displaystyle \lim_{R\rightarrow +\infty }\mathcal{R}_{\infty }\left(
q_{2},R\right) =0$ for every $q_{2}\in \mathbb{R}$ such that 
\begin{equation}
q_{2}>\max \left\{ 1,2\beta _{\infty },q^{*}\left( \alpha _{\infty },\beta
_{\infty }\right) \right\} .  \label{th2}
\end{equation}
\end{thm}

We observe explicitly that for every $\left( \alpha ,\beta \right) \in \mathbb{R%
}\times \left[ 0,1\right] $ one has 
\[
\max \left\{ 1,2\beta ,q^{*}\left( \alpha ,\beta \right) \right\} =\left\{ 
\begin{array}{ll}
q^{*}\left( \alpha ,\beta \right) \quad & \text{if }\alpha \geq \alpha
^{*}\left( \beta \right) \smallskip \\ 
\max \left\{ 1,2\beta \right\} & \text{if }\alpha \leq \alpha ^{*}\left(
\beta \right)
\end{array}
\right. . 
\]

\begin{rem}
\label{RMK: suff12}\quad 

\begin{enumerate}
\item  \label{RMK: suff12-V^0}We mean $V\left( r\right) ^{0}=1$ for every $r$
(even if $V\left( r\right) =0$). In particular, if $V\left( r\right) =0$ for
almost every $r>R_{2}$, then Theorem \ref{THM1} can be applied with $\beta
_{\infty }=0$ and assumption (\ref{esssup all'inf}) means 
\[
\esssup_{r>R_{2}}\frac{K\left( r\right) }{r^{\alpha _{\infty }}}%
<+\infty \quad \text{for some }\alpha _{\infty }\in \mathbb{R}.
\]
Similarly for Theorem \ref{THM0} and assumption (\ref{esssup in 0}), if $%
V\left( r\right) =0$ for almost every $r\in \left( 0,R_{1}\right) $.

\item  \label{RMK: suff12-no hp}The inequality $\max \left\{ 1,2\beta
_{0}\right\} <q^{*}\left( \alpha _{0},\beta _{0}\right) $ is equivalent to $%
\alpha _{0}>\alpha ^{*}\left( \beta _{0}\right) $. Then, in (\ref{th1}),
such inequality is automatically true and does not ask for further
conditions on $\alpha _{0}$ and $\beta _{0}$.

\item  \label{RMK: suff12-Vbdd}The assumptions of Theorems \ref{THM0} and 
\ref{THM1} may hold for different pairs $\left( \alpha _{0},\beta_{0}\right)$,
$\left( \alpha _{\infty },\beta _{\infty }\right) $. In this
case, of course, one chooses them in order to get the ranges for $q_{1},q_{2}
$ as large as possible. For instance, if $V$ is not singular at the origin,
i.e., $V$ is essentially bounded in a neighbourhood of 0, and condition (\ref
{esssup in 0}) holds true for a pair $\left( \alpha _{0},\beta _{0}\right) $%
, then (\ref{esssup in 0}) also holds for all pairs $\left( \alpha
_{0}^{\prime },\beta _{0}^{\prime }\right) $ such that $\alpha _{0}^{\prime
}<\alpha _{0}$ and $\beta _{0}^{\prime }<\beta _{0}$. Therefore, since $\max
\left\{ 1,2\beta \right\} $ is increasing in $\beta $ and $q^{*}\left(
\alpha ,\beta \right) $ is increasing in $\alpha $ and decreasing in $\beta $%
, it is convenient to choose $\beta _{0}=0$ and the best interval where one
can take $q_{1}$ is $1<q_{1}<q^{*}\left( \overline{\alpha },0\right) $ with $%
\overline{\alpha }:=\sup \left\{ \alpha _{0}:\esssup_{r\in \left(
0,R_{1}\right) }\frac{K\left( r\right) }{r^{\alpha _{0}}}<+\infty \right\} $
(we mean $q^{*}\left( +\infty ,0\right) =+\infty $).
\end{enumerate}
\end{rem}

For any $\alpha \in \mathbb{R}$, $\beta \leq 1$ and $\gamma \in \mathbb{R}$,
define 
\begin{equation}
q_{*}\left( \alpha ,\beta ,\gamma \right) :=2\frac{\alpha -\gamma \beta +N}{%
N-\gamma }\quad \text{and}\quad q_{**}\left( \alpha ,\beta ,\gamma \right)
:=2\frac{2\alpha +\left( 1-2\beta \right) \gamma +2\left( N-1\right) }{%
2\left( N-1\right) -\gamma }.  \label{q** :=}
\end{equation}
Of course $q_{*}$ and $q_{**}$ are undefined if $\gamma =N$ and $\gamma
=2\left( N-1\right) $, respectively.

The next Theorems \ref{THM2} and \ref{THM3} improve the results of Theorems 
\ref{THM0} and \ref{THM1} by exploiting further informations on the growth
of $V$ (see Remarks \ref{RMK: Hardy 1}.\ref{RMK: Hardy 1-improve} and \ref
{RMK: Hardy 2}.\ref{RMK: Hardy 2-improve}).

\begin{thm}
\label{THM2}Let $N\geq 3$ and let $V$, $K$ be as in $\left( \mathbf{V}%
\right) $, $\left( \mathbf{K}\right) $ with $V\left( r\right) <+\infty $.
Assume that there exists $R_{2}>0$ such that 
\begin{equation}
\esssup_{r>R_{2}}\frac{K\left( r\right) }{r^{\alpha _{\infty
}}V\left( r\right) ^{\beta _{\infty }}}<+\infty \quad \text{for some }0\leq
\beta _{\infty }\leq 1\text{~and }\alpha _{\infty }\in \mathbb{R}
\label{hp all'inf}
\end{equation}
and 
\begin{equation}
\essinf_{r>R_{2}}r^{\gamma _{\infty }}V\left( r\right) >0\quad 
\text{for some }\gamma _{\infty }\leq 2.  \label{stima all'inf}
\end{equation}
Then $\displaystyle \lim_{R\rightarrow +\infty }\mathcal{R}_{\infty }\left(
q_{2},R\right) =0$ for every $q_{2}\in \mathbb{R}$ such that 
\begin{equation}
q_{2}>\max \left\{ 1,2\beta _{\infty },q_{*},q_{**}\right\} ,  \label{th3}
\end{equation}
where $q_{*}=q_{*}\left( \alpha _{\infty },\beta _{\infty },\gamma _{\infty
}\right) $ and $q_{**}=q_{**}\left( \alpha _{\infty },\beta _{\infty
},\gamma _{\infty }\right) .$
\end{thm}

For future convenience, we define three functions $\alpha _{1}:=\alpha
_{1}\left( \beta ,\gamma \right) $, $\alpha _{2}:=\alpha _{2}\left( \beta
\right) $ and $\alpha _{3}:=\alpha _{3}\left( \beta ,\gamma \right) $ by
setting 
\begin{equation}
\alpha _{1}:=-\left( 1-\beta \right) \gamma ,\quad \alpha _{2}:=-\left(
1-\beta \right) N,\quad \alpha _{3}:=-\frac{N+\left( 1-2\beta \right) \gamma 
}{2}.  \label{alpha_i :=}
\end{equation}
Then an explicit description of $\max \left\{ 1,2\beta ,q_{*},q_{**}\right\} 
$ is the following: for every $\left( \alpha ,\beta ,\gamma \right) \in \mathbb{%
R}\times \left( -\infty ,1\right] \times \left( -\infty ,2\right] $ we have 
\begin{equation}
\max \left\{ 1,2\beta ,q_{*},q_{**}\right\} =\left\{ 
\begin{array}{ll}
q_{**}\left( \alpha ,\beta ,\gamma \right) \quad & \text{if }\alpha \geq
\alpha _{1}\smallskip \\ 
q_{*}\left( \alpha ,\beta ,\gamma \right) & \text{if }\max \left\{ \alpha
_{2},\alpha _{3}\right\} \leq \alpha \leq \alpha _{1}\smallskip \\ 
\max \left\{ 1,2\beta \right\} & \text{if }\alpha \leq \max \left\{ \alpha
_{2},\alpha _{3}\right\}
\end{array}
\right. ,  \label{descrizioneThm2}
\end{equation}
where $\max \left\{ \alpha _{2},\alpha _{3}\right\} <\alpha _{1}$ for every $%
\beta <1$ and $\max \left\{ \alpha _{2},\alpha _{3}\right\} =\alpha _{1}=0$
if $\beta =1$.

\begin{rem}
\label{RMK: Hardy 1}\quad 

\begin{enumerate}
\item  \label{RMK: Hardy 1-B<0}The proof of Theorem \ref{THM2} does not
require $\beta _{\infty }\geq 0$, but this condition is not a restriction of
generality in stating the theorem. Indeed, under assumption (\ref{stima
all'inf}), if (\ref{hp all'inf}) holds with $\beta _{\infty }<0$, then it
also holds with $\alpha _{\infty }$ and $\beta _{\infty }$ replaced by $%
\alpha _{\infty }-\beta _{\infty }\gamma _{\infty }$ and $0$ respectively,
and this does not change the thesis (\ref{th3}), because $q_{*}\left( \alpha
_{\infty }-\beta _{\infty }\gamma _{\infty },0,\gamma _{\infty }\right)
=q_{*}\left( \alpha _{\infty },\beta _{\infty },\gamma _{\infty }\right) $
and $q_{**}\left( \alpha _{\infty }-\beta _{\infty }\gamma _{\infty
},0,\gamma _{\infty }\right) =q_{**}\left( \alpha _{\infty },\beta _{\infty
},\gamma _{\infty }\right) $.

\item  \label{RMK: Hardy 1-improve}Denote $q^{*}=q^{*}\left( \alpha _{\infty
},\beta _{\infty }\right) $ for brevity. If $\gamma _{\infty }<2$, then one
has 
\[
\max \left\{ 1,2\beta _{\infty },q^{*}\right\} =\left\{ 
\begin{array}{ll}
\max \left\{ 1,2\beta _{\infty }\right\} =\max \left\{ 1,2\beta _{\infty
},q_{*},q_{**}\right\} \quad \smallskip  & \text{if }\alpha _{\infty }\leq
\alpha ^{*}\left( \beta _{\infty }\right)  \\ 
q^{*}>\max \left\{ 1,2\beta _{\infty },q_{*},q_{**}\right\}  & \text{if }%
\alpha _{\infty }>\alpha ^{*}\left( \beta _{\infty }\right) 
\end{array}
\right. ,
\]
so that, under assumption (\ref{stima all'inf}), Theorem \ref{THM2} improves
Theorem \ref{THM1}. Otherwise, if $\gamma _{\infty }=2$, we have $%
q_{*}=q_{**}=q^{*}$ and Theorems \ref{THM2} and \ref{THM1} give the same
result. This is not surprising, since, by Hardy inequality, the space $%
H_{V}^{1}$ coincides with $D^{1,2}(\mathbb{R}^{N})$ if $V\left( r\right) =r^{-2}
$ and thus, for $\gamma _{\infty }=2$, we cannot expect a better result than
the one of Theorem \ref{THM1}, which covers the case of $V\left( r\right)
\equiv 0$, i.e., of $D^{1,2}(\mathbb{R}^{N})$.

\item  \label{RMK: Hardy 1-best gamma}Description (\ref{descrizioneThm2})
shows that $q_{*}$ and $q_{**}$ are relevant in inequality (\ref{th3}) only
for $\alpha _{\infty }>\alpha _{2}\left( \beta _{\infty }\right) $. In this
case, both $q_{*}$ and $q_{**}$ turn out to be increasing in $\gamma $ and
hence it is convenient to apply Theorem \ref{THM2} with the smallest $\gamma
_{\infty }$ for which (\ref{stima all'inf}) holds. This is consistent with
the fact that, if (\ref{stima all'inf}) holds with $\gamma _{\infty }$, then
it also holds with every $\gamma _{\infty }^{\prime }$ such that $\gamma
_{\infty }\leq \gamma _{\infty }^{\prime }\leq 2$.
\end{enumerate}
\end{rem}

In order to state our last result, we introduce, by the following
definitions, an open region $\mathcal{A}_{\beta ,\gamma }$ of the $\alpha q$%
-plane, depending on $\beta \leq 1$ and $\gamma \geq 2$. Recall the
definitions (\ref{q** :=}) of the functions $q_{*}=q_{*}\left( \alpha ,\beta
,\gamma \right) $ and $q_{**}=q_{**}\left( \alpha ,\beta ,\gamma \right) $.
We set 
\begin{equation}
\begin{array}{ll}
\mathcal{A}_{\beta ,\gamma }:=\left\{ \left( \alpha ,q\right) :\max \left\{
1,2\beta \right\} <q<\min \left\{ q_{*},q_{**}\right\} \right\} \quad
\smallskip & \text{if }2\leq \gamma <N, \\ 
\mathcal{A}_{\beta ,\gamma }:=\left\{ \left( \alpha ,q\right) :\max \left\{
1,2\beta \right\} <q<q_{**},\,\alpha >-\left( 1-\beta \right) N\right\}
\quad \smallskip & \text{if }\gamma =N, \\ 
\mathcal{A}_{\beta ,\gamma }:=\left\{ \left( \alpha ,q\right) :\max \left\{
1,2\beta ,q_{*}\right\} <q<q_{**}\right\} \smallskip & \text{if }N<\gamma
<2N-2, \\ 
\mathcal{A}_{\beta ,\gamma }:=\left\{ \left( \alpha ,q\right) :\max \left\{
1,2\beta ,q_{*}\right\} <q,\,\alpha >-\left( 1-\beta \right) \gamma \right\}
\smallskip & \text{if }\gamma =2N-2, \\ 
\mathcal{A}_{\beta ,\gamma }:=\left\{ \left( \alpha ,q\right) :\max \left\{
1,2\beta ,q_{*},q_{**}\right\} <q\right\} & \text{if }\gamma >2N-2.
\end{array}
\label{A:=}
\end{equation}
For more clarity, $\mathcal{A}_{\beta ,\gamma }$ is sketched in the
following five pictures, according to the five cases above. Recall the
definitions (\ref{alpha_i :=}) of the functions $\alpha _{1}=\alpha
_{1}\left( \beta ,\gamma \right) $, $\alpha _{2}=\alpha _{2}\left( \beta
\right) $ and $\alpha _{3}=\alpha _{3}\left( \beta ,\gamma \right) $.\bigskip

\noindent
\begin{tabular}[t]{l}
\begin{tabular}{l}
\textbf{Fig.1}:\ \ $\mathcal{A}_{\beta ,\gamma }$ for $\beta \leq 1$\\
and $2\leq \gamma <N$.\smallskip\\
$\bullet $ If $\gamma =2$, the two\\
straight lines above are\\
the same.\smallskip\\
$\bullet $ If $\beta <1$ we have\\
$\max \left\{ \alpha _{2},\alpha _{3}\right\} <\alpha _{1}<0;$\\
if $\beta =1$ we have\\
$\max \left\{ \alpha _{2},\alpha _{3}\right\} =\alpha _{1}=0$\\
and $\mathcal{A}_{1,\gamma }$ reduces to the\\
angle $2<q<q_{**}$.
\end{tabular}
\begin{tabular}{l}
\includegraphics[width=3.9in]{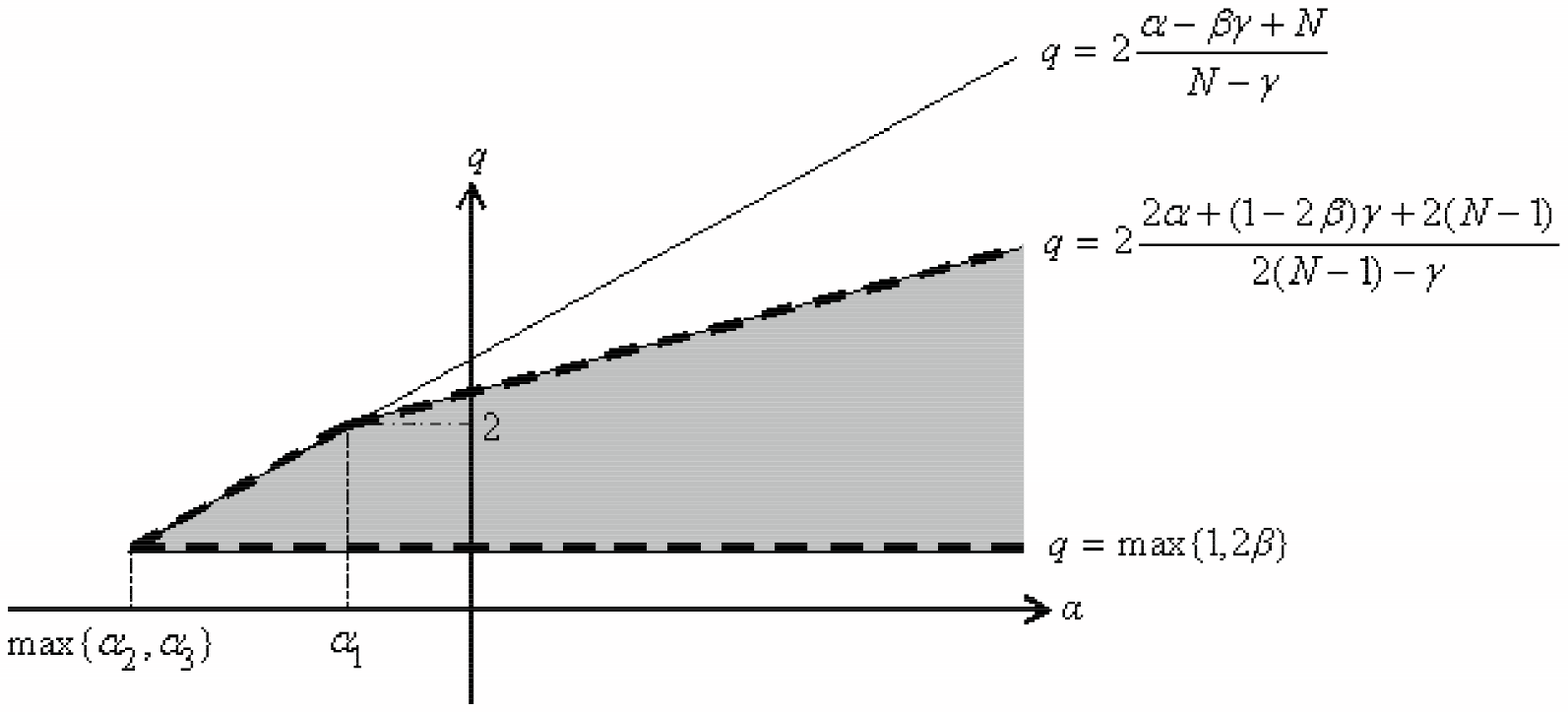}
\end{tabular}
\end{tabular}
\bigskip

\noindent
\begin{tabular}[t]{l}
\begin{tabular}{l}
\textbf{Fig.2}:\ \ $\mathcal{A}_{\beta ,\gamma }$ for $\beta \leq 1$\\
and $\gamma =N$.\smallskip\\
$\bullet $ If $\beta <1$ we have\\
$\alpha _{1}=\alpha _{2}=\alpha _{3}<0;$\\
if $\beta =1$ we have\\
$\alpha _{1}=\alpha _{2}=\alpha _{3}=0$\\
and $\mathcal{A}_{1,\gamma }$ reduces to the\\
angle $2<q<q_{**}$.
\end{tabular}
\begin{tabular}{l}
\includegraphics[width=3.9in]{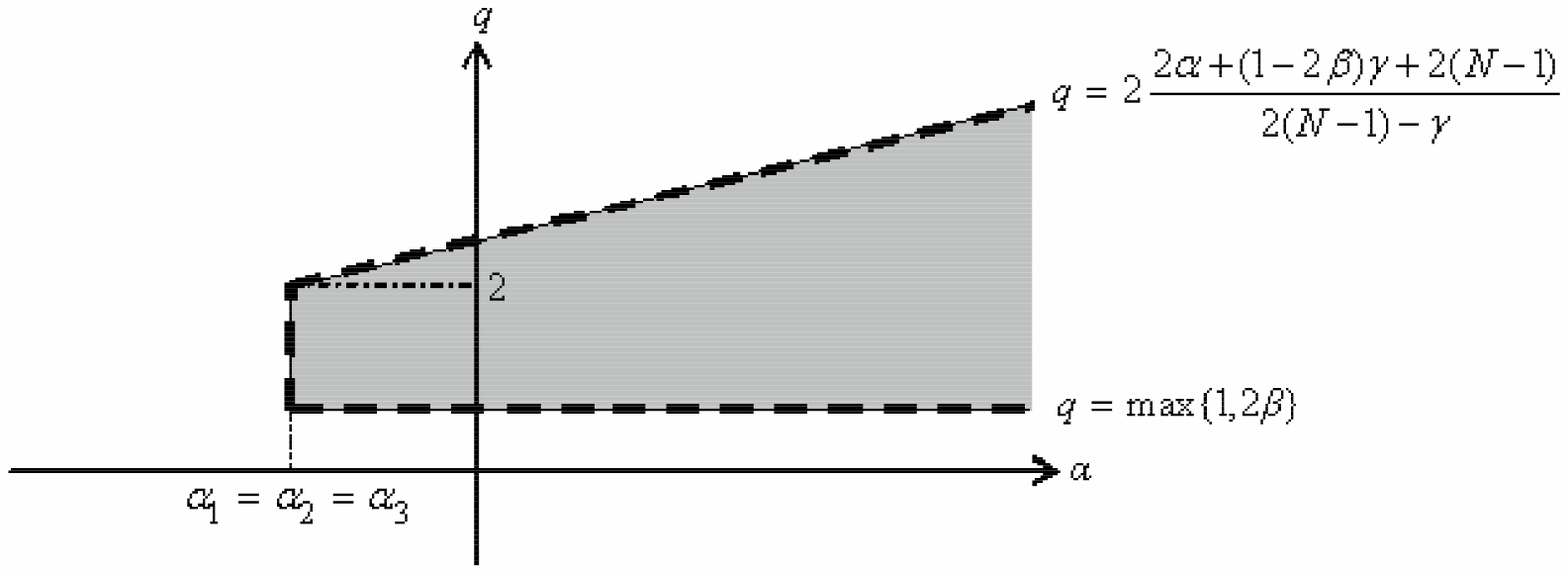}
\end{tabular}
\end{tabular}
\bigskip

\noindent
\begin{tabular}[t]{l}
\begin{tabular}{l}
\textbf{Fig.3}:\ \ $\mathcal{A}_{\beta ,\gamma }$ for $\beta \leq 1$\\
and $N<\gamma <2N-2$.\smallskip\\
$\bullet $ If $\beta <1$ we have\\
$\alpha _{1}<\min \left\{ \alpha _{2},\alpha _{3}\right\} <0;$\\
if $\beta =1$ we have\\
$\alpha _{1}=\min \left\{ \alpha _{2},\alpha _{3}\right\} =0$\\
and $\mathcal{A}_{1,\gamma }$ reduces to the\\
angle $2<q<q_{**}$.
\end{tabular}
\begin{tabular}{l}
\includegraphics[width=3.9in]{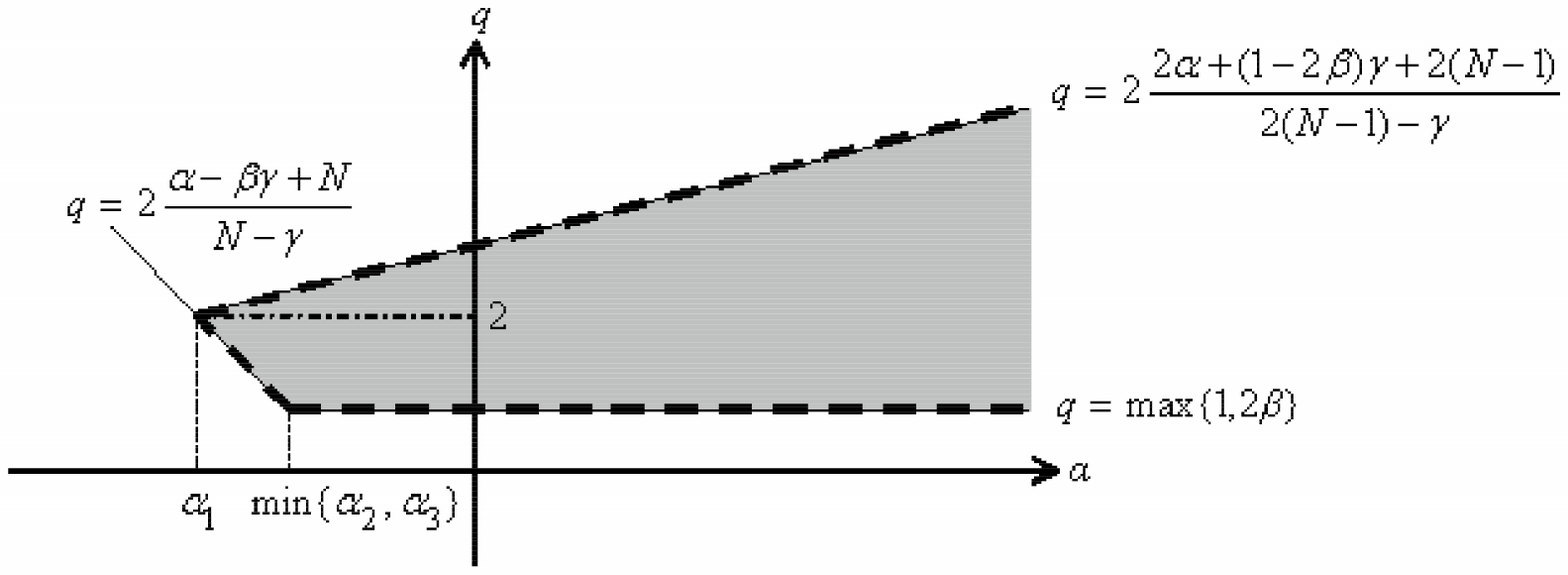}
\end{tabular}
\end{tabular}
\bigskip

\noindent
\begin{tabular}[t]{l}
\begin{tabular}{l}
\textbf{Fig.4}:\ \ $\mathcal{A}_{\beta ,\gamma }$ for $\beta \leq 1$ \\ 
and $\gamma =2N-2$.\smallskip\\ 
$\bullet $ If $\beta <1$ we have \\ 
$\alpha _{1}<\min \left\{ \alpha _{2},\alpha _{3}\right\} <0;$ \\ 
if $\beta =1$ we have \\ 
$\alpha _{1}=\min \left\{ \alpha _{2},\alpha _{3}\right\} =0$ \\ 
and $\mathcal{A}_{1,\gamma }$ reduces to the \\ 
angle $\alpha >0,\,q>2$.
\end{tabular}
\begin{tabular}{l}
\includegraphics[width=3.9in]{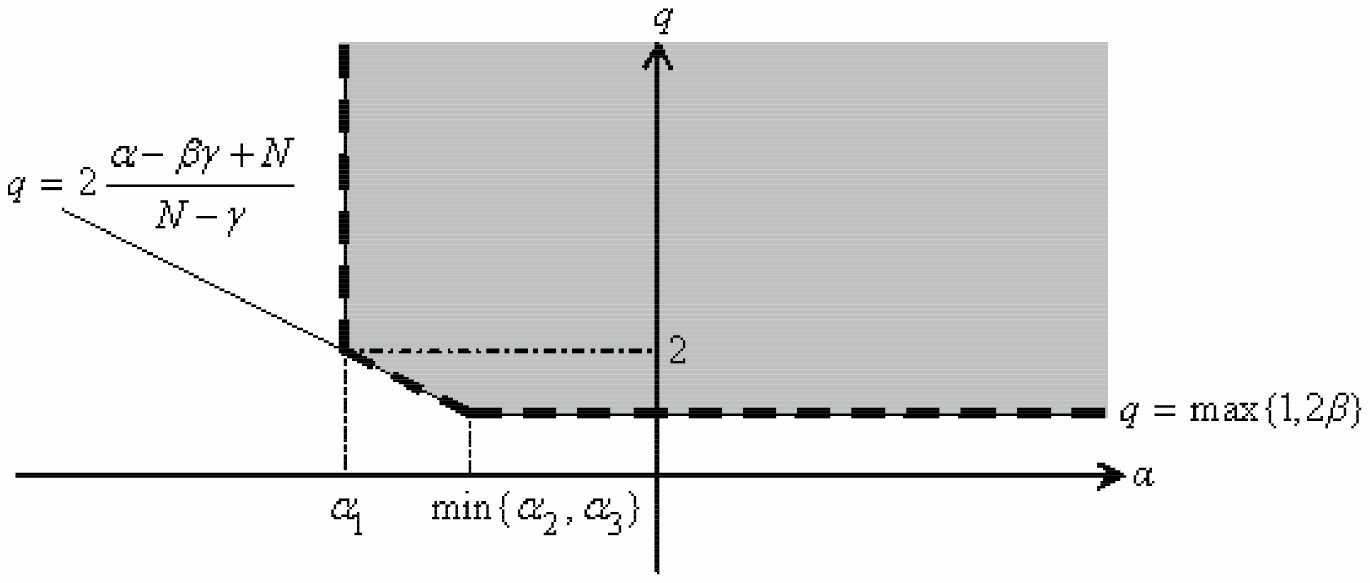}
\end{tabular}
\end{tabular}
\bigskip

\noindent
\begin{tabular}[t]{l}
\begin{tabular}{l}
\textbf{Fig.5}:\ \ $\mathcal{A}_{\beta ,\gamma }$ for $\beta \leq 1$\\
and $\gamma >2N-2$.\smallskip\\
$\bullet $ If $\beta <1$ we have\\
$\alpha _{1}<\min \left\{ \alpha _{2},\alpha _{3}\right\} <0;$\\
if $\beta =1$ we have\\
$\alpha _{1}=\min \left\{ \alpha _{2},\alpha _{3}\right\} =0$\\
and $\mathcal{A}_{1,\gamma }$ reduces to the\\
angle $q>\max \left\{ 2,q_{**}\right\} $.
\end{tabular}
\begin{tabular}{l}
\includegraphics[width=3.9in]{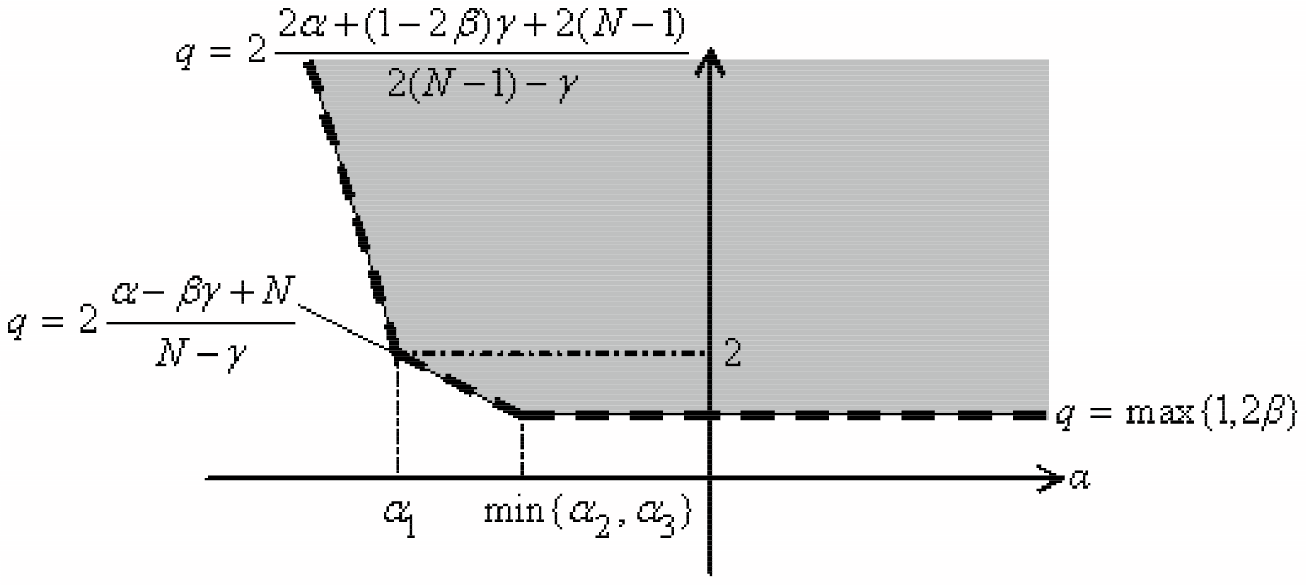}
\end{tabular}
\end{tabular}
\bigskip

\begin{thm}
\label{THM3}Let $N\geq 3$ and let $V$, $K$ be as in $\left( \mathbf{V}%
\right) $, $\left( \mathbf{K}\right) $ with $V\left( r\right) <+\infty $.
Assume that there exists $R_{1}>0$ such that 
\begin{equation}
\esssup_{r\in \left( 0,R_{1}\right) }\frac{K\left( r\right) }{%
r^{\alpha _{0}}V\left( r\right) ^{\beta _{0}}}<+\infty \quad \text{for some }%
0\leq \beta _{0}\leq 1\text{~and }\alpha _{0}\in \mathbb{R}  \label{hp in 0}
\end{equation}
and 
\begin{equation}
\essinf_{r\in \left( 0,R_{1}\right) }r^{\gamma _{0}}V\left(
r\right) >0\quad \text{for some }\gamma _{0}\geq 2.  \label{stima in 0}
\end{equation}
Then $\displaystyle \lim_{R\rightarrow 0^{+}}\mathcal{R}_{0}\left(
q_{1},R\right) =0$ for every $q_{1}\in \mathbb{R}$ such that 
\begin{equation}
\left( \alpha _{0},q_{1}\right) \in \mathcal{A}_{\beta _{0},\gamma _{0}}.
\label{th4}
\end{equation}
\end{thm}

\begin{rem}
\label{RMK: Hardy 2}\quad 

\begin{enumerate}
\item  Condition (\ref{th4}) also asks for a lower bound on $\alpha _{0}$,
except for the case $\gamma _{0}>2N-2$, as it is clear from Figures 1-5.

\item  The proof of Theorem \ref{THM3} does not require $\beta _{0}\geq 0$,
but this is not a restriction of generality in stating the theorem (cf.
Remark \ref{RMK: Hardy 1}.\ref{RMK: Hardy 1-B<0}). Indeed, under assumption (%
\ref{stima in 0}), if (\ref{hp in 0}) holds with $\beta _{0}<0$, then it
also holds with $\alpha _{0}$ and $\beta _{0}$ replaced by $\alpha
_{0}-\beta _{0}\gamma _{0}$ and $0$ respectively, and one has that $\left(
\alpha _{0},q_{1}\right) \in \mathcal{A}_{\beta _{0},\gamma _{0}}$ if and
only if $\left( \alpha _{0}-\beta _{0}\gamma _{0},q_{1}\right) \in \mathcal{A%
}_{0,\gamma _{0}}$.

\item  \label{RMK: Hardy 2-improve}If (\ref{stima in 0}) holds with $\gamma
_{0}>2$, then Theorem \ref{THM3} improves Theorem \ref{THM0}. Otherwise, if $%
\gamma _{0}=2$, then one has $\max \left\{ \alpha _{2},\alpha _{3}\right\}
=\alpha ^{*}\left( \beta _{0}\right) $ and $\left( \alpha _{0},q_{1}\right)
\in \mathcal{A}_{\beta _{0},\gamma _{0}}$ is equivalent to $\max \left\{
1,2\beta _{0}\right\} <q_{1}<q^{*}\left( \alpha _{0},\beta _{0}\right) $,
i.e., Theorems \ref{THM3} and \ref{THM0} give the same result, which is
consistent with Hardy inequality (cf. Remark \ref{RMK: Hardy 1}.\ref{RMK:
Hardy 1-improve}).

\item  \label{RMK: Hardy 2-best gamma}Given $\beta \leq 1$, one can check
that $\mathcal{A}_{\beta ,\gamma _{1}}\subseteq \mathcal{A}_{\beta ,\gamma
_{2}}$ for every $2\leq \gamma _{1}<\gamma _{2}$, so that, in applying
Theorem \ref{THM3}, it is convenient to choose the largest $\gamma _{0}$ for
which (\ref{stima in 0}) holds. This is consistent with the fact that, if (%
\ref{stima in 0}) holds with $\gamma _{0}$, then it also holds with every $%
\gamma _{0}^{\prime }$ such that $2\leq \gamma _{0}^{\prime }\leq \gamma _{0}
$.
\end{enumerate}
\end{rem}

\section{Examples \label{SUB: es}}

In this section we give some examples which might clarify how to use our
results and compare them with the ones of \cite{Su-Wang-Will
p,SuTian12,BonMerc11,BPR}, cited in the Introduction. It will be always
understood that $N\geq 3$.

\begin{exa}
\label{EX(SWW)}Consider the potentials 
\[
V\left( r\right) =\frac{1}{r^{a}},\quad K\left( r\right) =\frac{1}{r^{a-1}}%
,\quad a\leq 2.
\]
Since $V$ satisfies (\ref{stima all'inf}) with $\gamma _{\infty }=a$ (cf.
Remark \ref{RMK: Hardy 1}.\ref{RMK: Hardy 1-best gamma} for the best choice
of $\gamma _{\infty }$), we apply Theorem \ref{THM(cpt)} together with
Theorems \ref{THM0} and \ref{THM2}. Assumptions (\ref{esssup in 0}) and (\ref
{hp all'inf}) hold if and only if $\alpha _{0}\leq a\beta _{0}-a+1$ and $%
\alpha _{\infty }\geq a\beta _{\infty }-a+1$. According to (\ref{th1}) and (%
\ref{th3}), it is convenient to choose $\alpha _{0}$ as large as possible
and $\alpha _{\infty }$ as small as possible, so we take 
\[
\alpha _{0}=a\beta _{0}-a+1,\quad \alpha _{\infty }=a\beta _{\infty }-a+1.
\]
Then $q^{*}=q^{*}\left( \alpha _{0},\beta _{0}\right) $, $q_{*}=q_{*}\left(
\alpha _{\infty },\beta _{\infty },a\right) $ and $q_{**}=q_{**}\left(
\alpha _{\infty },\beta _{\infty },a\right) $ are given by 
\begin{equation}
q^{*}=2\frac{N-a+1-\left( 2-a\right) \beta _{0}}{N-2},\quad q_{*}=2\frac{%
N-a+1}{N-a}\quad \text{and}\quad q_{**}=2\frac{2N-a}{2N-a-2},  \label{qqq}
\end{equation}
where $a\leq 2$ implies $q_{*}\leq q_{**}$. Note that $\alpha _{0}>\alpha
^{*}\left( \beta _{0}\right) $ for every $\beta _{0}$. Since $q^{*}$ is
decreasing in $\beta _{0}$ and $q_{**}$ is independent of $\beta _{\infty }$%
, it is convient to choose $\beta _{0}=\beta _{\infty }=0$, so that Theorems 
\ref{THM0} and \ref{THM2} yield to exponents $q_{1},q_{2}$ such that 
\[
1<q_{1}<q^{*}=2\frac{N-a+1}{N-2},\quad q_{2}>q_{**}=2\frac{2N-a}{2N-a-2}.
\]
If $a<2$, then one has $q_{**}<q^{*}$ and therefore we get the compact
embedding 
\begin{equation}
H_{V,\mathrm{r}}^{1}\hookrightarrow L_{K}^{q}\qquad \text{for}\quad 2\frac{%
2N-a}{2N-a-2}<q<2\frac{N-a+1}{N-2}.  \label{ES(sww): p}
\end{equation}
If $a=2$, then $q_{**}=q^{*}=2\left( N-1\right) /\left( N-2\right) $ and we
have the compact embedding 
\[
H_{V,\mathrm{r}}^{1}\hookrightarrow L_{K}^{q_{1}}+L_{K}^{q_{2}}\qquad \text{%
for}\quad 1<q_{1}<2\frac{N-1}{N-2}<q_{2}.
\]
Since $V$ and $K$ are power potentials, one can also apply the results of 
\cite{Su-Wang-Will p} (or equivalently of \cite{BonMerc11}, which are the
same for power potentials), finding two limit exponents $\underline{q}$ and $%
\overline{q}$ such that the embedding $H_{V,\mathrm{r}}^{1}\hookrightarrow
L_{K}^{q}$ is compact if $\underline{q}<q<\overline{q}$. These exponents $%
\underline{q}$ and $\overline{q}$ are exactly exponents $q_{**}$ and $q^{*}$
of (\ref{qqq}) respectively, so that one obtains (\ref{ES(sww): p}) again,
provided that $a<2$ (which implies $\underline{q}<\overline{q}$). If $a=2$,
instead, one gets $\underline{q}=\overline{q}$ and no result is avaliable in 
\cite{Su-Wang-Will p} (nor in \cite{BonMerc11}). The results of \cite
{SuTian12} do not apply to $V$ and $K$, since the top and bottom exponents
of \cite{SuTian12} turn out to be equal to one another for every $a\leq 2$.
\end{exa}

\begin{exa}
\label{EX: BPR}Taking $V\left( r\right) \equiv 0$, $K$ as in $\left( \mathbf{%
K}\right) $ and $\beta _{0}=\beta _{\infty }=0$, from Theorems \ref{THM0}, 
\ref{THM1} and \ref{THM(cpt)} (see also Remark \ref{RMK: suff12}.\ref{RMK:
suff12-V^0}) we get the compact embedding $D_{\mathrm{rad}}^{1,2}(\mathbb{R}%
^{N})\hookrightarrow L_{K}^{q_{1}}+L_{K}^{q_{2}}$ for 
\begin{equation}
1<q_{1}<2\frac{\alpha _{0}+N}{N-2}\quad \text{and}\quad q_{2}>\max \left\{
1,2\frac{\alpha _{\infty }+N}{N-2}\right\} ,  \label{EX(BPR): th}
\end{equation}
provided that $\exists R_{1},R_{2}>0$ such that 
\begin{equation}
\esssup_{r>R_{2}}\frac{K\left( r\right) }{r^{\alpha _{\infty }}}%
<+\infty \quad \text{and}\quad \esssup_{r\in \left( 0,R_{1}\right) }%
\frac{K\left( r\right) }{r^{\alpha _{0}}}<+\infty \quad \text{with }\alpha
_{0}>-\frac{N+2}{2}.  \label{EX(BPR): hp}
\end{equation}
This result has already been proved in \cite[Theorem 4.1]{BPR}, assuming $%
K\in L_{\mathrm{loc}}^{\infty }(\left( 0,+\infty \right) )$. Of course,
according to (\ref{EX(BPR): th}), in (\ref{EX(BPR): hp}) it is convenient to
choose $\alpha _{0}$ as large as possible and $\alpha _{\infty }$ as small
as possible. For instance, if $K\left( r\right) =r^{d}$ with $d>-\left(
N+2\right) /2$, we choose $\alpha _{0}=\alpha _{\infty }=d$ and obtain the
compact embedding 
\[
D_{\mathrm{rad}}^{1,2}(\mathbb{R}^{N})\hookrightarrow
L_{K}^{q_{1}}+L_{K}^{q_{2}}\qquad \text{for}\quad 1<q_{1}<2\frac{d+N}{N-2}%
<q_{2}.
\]
Observe that $2(d+N)/(N-2)$ is the ``critical exponent'' found by Ni \cite
{Ni} for the Henon equation. We also observe that, if (\ref{EX(BPR): hp})
holds for some $\alpha _{0}>\alpha _{\infty }$, then we can take $q_{1}=q_{2}
$ in (\ref{EX(BPR): th}) and we get the compact embedding 
\[
D_{\mathrm{rad}}^{1,2}(\mathbb{R}^{N})\hookrightarrow L_{K}^{q}\qquad \text{for}%
\quad \max \left\{ 1,2\frac{\alpha _{\infty }+N}{N-2}\right\} <q<2\frac{%
\alpha _{0}+N}{N-2}
\]
(cf. \cite[Corollary 4.6]{BPR}).
\end{exa}

\begin{exa}
\label{EX(nnP1)}Essentially the same result of Example \ref{EX: BPR} holds
if $V$ is not singular at the origin and, roughly speaking, decays at
infinity much faster than $K$ (or is compactly supported). The result
becomes different (and better) if $K$ decays at infinity similarly to $V$,
or much faster. For example, consider the potentials 
\[
V\left( r\right) =e^{-ar},\quad K_{1}\left( r\right) =r^{d},\quad
K_{2}\left( r\right) =r^{d}e^{-br},\quad a,b>0,\quad d>-\frac{N+2}{2}.
\]
Since $V$ does not satisfy (\ref{stima in 0}) or (\ref{stima all'inf}), we
study the embedding properties of $H_{V,\mathrm{r}}^{1}$ by Theorems \ref
{THM0}, \ref{THM1} and \ref{THM(cpt)}. According to Remark \ref{RMK: suff12}.%
\ref{RMK: suff12-Vbdd}, both for $K=K_{1}$ and $K=K_{2}$, Theorem \ref{THM0}
leads to take 
\[
1<q_{1}<2\frac{d+N}{N-2}.
\]
If $K=K_{1}$, the ratio in (\ref{esssup all'inf}) is bounded only if $\beta
_{\infty }=0$ and the best $\alpha _{\infty }$ we can take is $\alpha
_{\infty }=d$, which yields $q_{2}>2\frac{d+N}{N-2}$. Then we get the
compact embedding 
\[
H_{V,\mathrm{r}}^{1}\hookrightarrow
L_{K_{1}}^{q_{1}}+L_{K_{1}}^{q_{2}}\qquad \text{for}\quad 1<q_{1}<2\frac{d+N%
}{N-2}<q_{2}.
\]
If $K=K_{2}$, instead, assumption (\ref{esssup all'inf}) holds with $\beta
_{\infty }=0$ and $\alpha _{\infty }\in \mathbb{R}$ arbitrary, so that we can
take $q_{2}>1$ arbitrary. Picking $q_{2}=q_{1}$, we get the compact
embedding 
\[
H_{V,\mathrm{r}}^{1}\hookrightarrow L_{K_{2}}^{q}\qquad \text{for}\quad 1<q<2%
\frac{d+N}{N-2}.
\]
A similar but weaker result follows from the theorems of \cite{BonMerc11},
which yield the compact embedding 
\[
H_{V,\mathrm{r}}^{1}\hookrightarrow L_{K_{2}}^{q}\qquad \text{for}\quad 2<q<2%
\frac{d+N}{N-2}.
\]
Note that this requires the restriction $d>-2$. Even though they belong to
the \textit{Hardy-Dieudonn\'{e} comparison class}, no result is avaliable in 
\cite{BonMerc11} for the potentials $V$ and $K_{1}$, due to a lack of
compatibility between their behaviours at zero and at infinity.
\end{exa}

\begin{exa}
\label{EX(nnP2)}Consider the potentials 
\[
V\left( r\right) =e^{\frac{1}{r}},\quad K\left( r\right) =e^{\frac{b}{r}%
},\quad 0<b\leq 1.
\]
Since $V$ satisfies (\ref{stima in 0}), we apply Theorem \ref{THM(cpt)}
together with Theorems \ref{THM1} and \ref{THM3}. Assumption (\ref{esssup
all'inf}) holds for $\alpha _{\infty }\geq 0$ and $0\leq \beta _{\infty
}\leq 1$, so that the best choice for $\alpha _{\infty }$, which is $\alpha
_{\infty }=0$, gives 
\[
\max \left\{ 1,2\beta _{\infty },q^{*}\left( 0,\beta _{\infty }\right)
\right\} =2\frac{N-2\beta _{\infty }}{N-2}.
\]
Now we then take $\beta _{\infty }=1$, so that Theorem \ref{THM1} gives $%
q_{2}>2$. Observe that the same result ensues by applying Theorem \ref{THM2}
with $\gamma _{\infty }=0$ in assumption (\ref{stima all'inf}). As to
Theorem \ref{THM3}, hypothesis (\ref{stima in 0}) holds with $\gamma
_{0}\geq 2$ arbitrary and therefore the most convenient choice is to assume $%
\gamma _{0}>2N-2$ (see Remark \ref{RMK: Hardy 2}.\ref{RMK: Hardy 2-best
gamma}). On the other hand, we have 
\[
\frac{K\left( r\right) }{r^{\alpha _{0}}V\left( r\right) ^{\beta _{0}}}=%
\frac{e^{\frac{b-\beta _{0}}{r}}}{r^{\alpha _{0}}}
\]
and thus hypothesis (\ref{hp in 0}) holds for some $\alpha _{0}\in \mathbb{R}$
if and only if $b\leq \beta _{0}\leq 1$. We now distinguish two cases. If $%
b<1$, we can take $\beta _{0}>b$ and thus (\ref{hp in 0}) holds for every $%
\alpha _{0}\in \mathbb{R}$, so that Theorem \ref{THM3} gives $q_{1}>\max
\left\{ 1,2\beta _{0}\right\} $ (see Fig.5), i.e., $q_{1}>\max \left\{
1,2b\right\} $. If $b=1$, then we need to take $\beta _{0}=1$ and thus (\ref
{hp in 0}) holds for $\alpha _{0}\leq 0$. Since $\gamma _{0}>2N-2$ implies 
\begin{eqnarray*}
\mathcal{A}_{1,\gamma _{0}}&=&
\left\{ \left( \alpha ,q\right) \in \mathbb{R}^{2}:q>\max \left\{ 2,q_{**}\left( \alpha ,1,\gamma _{0}\right) \right\}
\right\}\\
&=&\left\{ \left( \alpha ,q\right) \in \mathbb{R}^{2}:q>\max \left\{ 2,2-\frac{4\alpha }{\gamma _{0}-2\left( N-1\right) }\right\} \right\} ,
\end{eqnarray*}
the best choice for $\alpha _{0}\leq 0$ is $\alpha _{0}=0$ and we get that $%
\left( 0,q_{1}\right) \in \mathcal{A}_{1,\gamma _{0}}$ if and only if $%
q_{1}>2$. Hence Theorem \ref{THM3} gives $q_{1}>\max \left\{ 1,2b\right\} $
again. As a conclusion, recalling the condition $q_{2}>2$ and observing that 
$0<b\leq 1$ implies $\max \left\{ 1,2b\right\} \leq 2$, we obtain the
compact embedding 
\[
H_{V,\mathrm{r}}^{1}\hookrightarrow L_{K}^{q}\qquad \text{for}\quad q>2.
\]
If we now modify $V$ by taking a compactly supported potential $V_{1}$ such
that $V_{1}\left( r\right) \sim V\left( r\right) $ as $r\rightarrow 0^{+}$,
everything works as above in applying Theorem \ref{THM3}, but we now need to
take $\beta _{\infty }=0$ and $\alpha _{\infty }\geq 0$ in Theorem \ref{THM1}%
. This gives 
\[
\max \left\{ 1,2\beta _{\infty },q^{*}\left( \alpha _{\infty },\beta
_{\infty }\right) \right\} =2\frac{\alpha _{\infty }+N}{N-2}
\]
and thus, choosing $\alpha _{\infty }=0$, we get the compact embedding 
\[
H_{V_{1},\mathrm{r}}^{1}\hookrightarrow L_{K}^{q}\qquad \text{for}\quad
q>2^{*}.
\]
Similarly, if we modify $V$ by taking a potential $V_{2}$ such that $%
V_{2}\left( r\right) \sim V\left( r\right) $ as $r\rightarrow 0^{+}$ and $%
V_{2}\left( r\right) \sim r^{N}$ as $r\rightarrow +\infty $, Theorem \ref
{THM3} yields $q_{1}>\max \left\{ 1,2b\right\} $ as above and Theorem \ref
{THM1} gives $q_{2}>1$ (apply it for instance with $\alpha _{\infty }=-N/2$
and $\beta _{\infty }=1/2$), so that we get the compact embedding 
\[
H_{V_{2},\mathrm{r}}^{1}\hookrightarrow L_{K}^{q}\qquad \text{for}\quad
q>\max \left\{ 1,2b\right\} .
\]
The compact embeddings $H_{V,\mathrm{r}}^{1}\hookrightarrow L_{K}^{q}$ and $%
H_{V_{2},\mathrm{r}}^{1}\hookrightarrow L_{K}^{q}$ also follow from the
results of \cite{BonMerc11}, but for $q>2$ in both cases. No result is
avaliable in \cite{BonMerc11} for the embedding $H_{V_{1},\mathrm{r}%
}^{1}\hookrightarrow L_{K}^{q}$, since $V_{1}$ has compact support.
\end{exa}

\begin{exa}
\label{EX(ST)}Consider the potential 
\begin{equation}
V\left( r\right) =r^{a},\quad -2\left( N-1\right) <a<-N,  \label{EX1: V}
\end{equation}
and let $K$ be as in $\left( \mathbf{K}\right) $ and such that 
\begin{equation}
K\left( r\right) =O\left( r^{b_{0}}\right) _{r\rightarrow 0^{+}},\quad
K\left( r\right) =O\left( r^{b}\right) _{r\rightarrow +\infty },\quad
b_{0}>a,\quad b\in \mathbb{R}.  \label{EX1: K}
\end{equation}
Since $V$ satisfies (\ref{stima in 0}) with $\gamma _{0}=-a$ (cf. Remark \ref
{RMK: Hardy 2}.\ref{RMK: Hardy 2-best gamma} for the best choice of $\gamma
_{0}$), we apply Theorem \ref{THM(cpt)} together with Theorems \ref{THM1}
and \ref{THM3}, where assumptions (\ref{esssup all'inf}) and (\ref{hp in 0})
hold for $\alpha _{\infty }\geq b-a\beta _{\infty }$ and $\alpha _{0}\leq
b_{0}-a\beta _{0}$ with $0\leq \beta _{\infty }\leq 1$ and $\beta _{0}\leq 1$
arbitrary. Note that $N<\gamma _{0}<2N-2$. According to (\ref{th2}) and (\ref
{th4}) (see in particular Fig.3), it is convenient to choose $\alpha
_{\infty }$ as small as possible and $\alpha _{0}$ as large as possible, so
we take 
\begin{equation}
\alpha _{\infty }=b-a\beta _{\infty },\quad \alpha _{0}=b_{0}-a\beta _{0}.
\label{EX(SWW2):alpha}
\end{equation}
Then $q^{*}=q^{*}\left( \alpha _{\infty },\beta _{\infty }\right) $, $%
q_{*}=q_{*}\left( \alpha _{0},\beta _{0},-a\right) $ and $%
q_{**}=q_{**}\left( \alpha _{0},\beta _{0},-a\right) $ are given by 
\[
q^{*}=2\frac{N+b-\left( a+2\right) \beta _{\infty }}{N-2},\quad q_{*}=2\frac{%
N+b_{0}}{N+a}\quad \text{and}\quad q_{**}=2\frac{2N-2+2b_{0}-a}{2N-2+a}.
\]
Since $a+2<2-N<0$, the exponent $q^{*}$ is increasing in $\beta _{\infty }$
and thus, according to (\ref{th2}) again, the best choice for $\beta
_{\infty }$ is $\beta _{\infty }=0$. This yields 
\begin{equation}
q_{2}>\max \left\{ 1,2\frac{N+b}{N-2}\right\} .  \label{EX(SWW2):q1}
\end{equation}
As to Theorem \ref{THM3}, we observe that, thanks to the choice of $\alpha
_{0}$, the exponents $q_{*}$ and $q_{**}$ are independent of $\beta _{0}$,
so that we can choose $\beta _{0}=0$ in order to get the region $\mathcal{A}%
_{\beta _{0},-a}$ as large as possible (cf. Fig.3 or the third definition in
(\ref{A:=})). Then we get $\alpha _{0}=b_{0}>a=\alpha _{1}$ (recall (\ref
{EX(SWW2):alpha}) and the definition (\ref{alpha_i :=}) of $\alpha _{1}$),
so that $\left( \alpha _{0},q_{1}\right) \in \mathcal{A}_{0,-a}$ if and only
if 
\begin{equation}
\max \left\{ 1,2\frac{N+b_{0}}{N+a}\right\} <q_{1}<2\frac{2N-2+2b_{0}-a}{%
2N-2+a}.  \label{EX(SWW2):q2}
\end{equation}
As a conclusion, we obtain the compact embedding 
\[
H_{V,\mathrm{r}}^{1}\hookrightarrow L_{K}^{q_{1}}+L_{K}^{q_{2}}\qquad \text{%
for every }q_{1},q_{2}\text{\ satisfying (\ref{EX(SWW2):q1}) and (\ref
{EX(SWW2):q2}).}
\]
If furthermore $a,b,b_{0}$ are such that 
\begin{equation}
\frac{N+b}{N-2}<\frac{2N-2+2b_{0}-a}{2N-2+a},  \label{EX1: further}
\end{equation}
then we can take $q_{1}=q_{2}$ and we get the compact embedding 
\begin{equation}
H_{V,\mathrm{r}}^{1}\hookrightarrow L_{K}^{q}\qquad \text{for}\quad \max
\left\{ 1,2\frac{N+b_{0}}{N+a},2\frac{N+b}{N-2}\right\} <q<2\frac{%
2N-2+2b_{0}-a}{2N-2+a}.  \label{EX1: Lq}
\end{equation}
Observe that the potentials $V$ and $K$ behave as a power and thus they fall
into the classes considered in \cite{Su-Wang-Will p,SuTian12,BonMerc11}. In
particular, the results of \cite{Su-Wang-Will p} and \cite{BonMerc11} (which
are the same for $V,K$ as in (\ref{EX1: V}), (\ref{EX1: K})) provide the
compact embedding 
\begin{equation}
H_{V,\mathrm{r}}^{1}\hookrightarrow L_{K}^{q}\qquad \text{for}\quad \max
\left\{ 2,2\frac{N+b}{N-2}\right\} =:\underline{q}<q<\overline{q}:=2\frac{%
2N-2+2b_{0}-a}{2N-2+a}.  \label{EX1: ssw}
\end{equation}
This requires condition (\ref{EX1: further}), which amounts to $\underline{q}%
<\overline{q}$, and no compact embedding is found in \cite{Su-Wang-Will
p,BonMerc11} if (\ref{EX1: further}) fails. Moreover, our result improves (%
\ref{EX1: ssw}) even if (\ref{EX1: further}) holds. Indeed, $b_{0}>a$ and $%
N+a<0$ imply $\frac{N+b_{0}}{N+a}<1$ and thus one has 
\[
\underline{q}=\left\{ 
\begin{array}{ll}
2\frac{N+b}{N-2}=\max \left\{ 1,2\frac{N+b}{N-2},2\frac{N+b_{0}}{N+a}%
\right\} \quad \smallskip  & \text{if }b\geq -2 \\ 
2>\max \left\{ 1,2\frac{N+b}{N-2},2\frac{N+b_{0}}{N+a}\right\}  & \text{if }%
b<-2
\end{array}
\right. ,
\]
so that (\ref{EX1: Lq}) is exactly (\ref{EX1: ssw}) if $b\geq -2$ and it is
better if $b<-2$. This last case actually concerns sub-quadratic exponents,
so it should be also compared with the results of \cite{SuTian12}, where,
setting 
\[
b_{1}:=\frac{a-2-2N}{4},\qquad b_{2}:=\frac{a-2}{2},\qquad b_{3}:=-\frac{N+2%
}{2}
\]
(notice that $-N<b_{1}<b_{2}<b_{3}<-2$ for $a$ as in (\ref{EX1: V})) and 
\begin{equation}
\underline{q}^{\prime }:=\left\{ 
\begin{array}{ll}
2\frac{N+b}{N-2}\quad \quad \smallskip  & \text{if }b\in \left[
b_{3},-2\right)  \\ 
4\frac{N+b}{2N-2+a} & \text{if }b\in \left[ b_{1},b_{2}\right) 
\end{array}
\right. ,\qquad \overline{q}^{\prime }:=\left\{ 
\begin{array}{ll}
2\frac{N+b_{0}}{N-2}\quad \quad \smallskip  & \text{if }b_{0}\in \left(
b_{3},-2\right]  \\ 
4\frac{N+b_{0}}{2N-2+a} & \text{if }b_{0}\in \left( b_{1},b_{2}\right] 
\end{array}
\right. ,  \label{EX1: st-qs}
\end{equation}
the authors find the compact embedding 
\begin{equation}
H_{V,\mathrm{r}}^{1}\hookrightarrow L_{K}^{q}\qquad \text{for}\quad 
\underline{q}^{\prime }<q<\overline{q}^{\prime }.  \label{EX1: st}
\end{equation}
Our result (\ref{EX1: Lq})-(\ref{EX1: further}) extends (\ref{EX1: st}) in
three directions. First, (\ref{EX1: st}) requires that $\underline{q}%
^{\prime }$ and $\overline{q}^{\prime }$ are defined, i.e., $b$ and $b_{0}$
lie in the intervals considered in (\ref{EX1: st-qs}), while (\ref{EX1: Lq})
and (\ref{EX1: further}) do not need such a restriction, also covering cases
of $b\in \left( -\infty ,b_{1}\right) \cup \left[ b_{2},b_{3}\right) $ or $%
b_{0}\in \left( a,b_{1}\right] \cup \left( b_{2},b_{3}\right] $ (take for
instance $b_{0}>a$ arbitrary and $b$ small enough to satisfy (\ref{EX1:
further})). Moreover, (\ref{EX1: st}) asks for the further condition $%
\underline{q}^{\prime }<\overline{q}^{\prime }$, which can be false even if $%
\underline{q}^{\prime }$ and $\overline{q}^{\prime }$ are defined (take for
instance $b_{3}<b=b_{0}<-2$ or $b_{1}<b=b_{0}<b_{2}$, which give $\underline{%
q}^{\prime }=\overline{q}^{\prime }$), while condition (\ref{EX1: further})
does not. Indeed, as soon as $\underline{q}^{\prime }$ and $\overline{q}%
^{\prime }$ are defined, one has $b<-2$ and $b_{0}>-N$, which imply 
\[
\frac{N+b}{N-2}-\frac{2N-2+2b_{0}-a}{2N-2+a}<1+\frac{2+a}{2N-2+a}=2\frac{N+a%
}{2N-2+a}<0.
\]
Finally, setting 
\[
\underline{q}^{\prime \prime }:=\max \left\{ 1,2\frac{N+b_{0}}{N+a},2\frac{%
N+b}{N-2}\right\} 
\]
for brevity, some computations (which we leave to the reader) show that,
whenever $\underline{q}^{\prime }$ and $\overline{q}^{\prime }$ are defined,
one has 
\[
\underline{q}^{\prime }=\left\{ 
\begin{array}{ll}
2\frac{N+b}{N-2}=\underline{q}^{\prime \prime }\medskip  & \text{if }b\in
\left[ b_{3},-2\right)  \\ 
4\frac{N+b}{2N-2+a}=1=\underline{q}^{\prime \prime }~~\medskip  & \text{if }%
b=b_{1} \\ 
4\frac{N+b}{2N-2+a}>1=\underline{q}^{\prime \prime } & \text{if }b\in \left(
b_{1},b_{2}\right) 
\end{array}
\right. \quad \text{and}\quad \overline{q}^{\prime }<2\frac{2N-2+2b_{0}-a}{%
2N-2+a}.
\]
This shows that (\ref{EX1: Lq}) always gives a wider range of exponents $q$
than (\ref{EX1: st}).
\end{exa}

\section{Proof of Theorem \ref{THM(cpt)} \label{SEC:1}}

Assume $N\geq 3$ and let $V$ and $K$ be as in $\left( \mathbf{V}\right) $
and $\left( \mathbf{K}\right) $.

Recall the definitions (\ref{H:=}) and (\ref{Hr:=}) of the Hilbert spaces $%
H_{V}^{1}$ and $H_{V,\mathrm{r}}^{1}$ and fix two constants $S_{N}>0$ and $%
C_{N}>0$, only dependent on $N$, such that 
\begin{equation}
\forall u\in H_{V}^{1}\left( \mathbb{R}^{N}\right) ,\quad \left\| u\right\|
_{L^{2^{*}}(\mathbb{R}^{N})}\leq S_{N}\left\| u\right\|  \label{Sobolev}
\end{equation}
and 
\begin{equation}
\forall u\in H_{V,\mathrm{r}}^{1}\left( \mathbb{R}^{N}\right) ,\quad \left|
u\left( x\right) \right| \leq C_{N}\left\| u\right\| \frac{1}{\left|
x\right| ^{\frac{N-2}{2}}}\quad \text{almost everywhere on }\mathbb{R}^{N}.
\label{PointwiseEstimate}
\end{equation}
The first constant $S_{N}$ exists by Sobolev inequality and the continuous
embedding $H_{V}^{1}\hookrightarrow D^{1,2}(\mathbb{R}^{N})$. The second
constant $C_{N}$ does exist by Ni's inequality \cite{Ni} (see also 
\cite[Lemma 1]{Su-Wang-Will p}) and the continuous embedding $H_{V,\mathrm{r}%
}^{1}\hookrightarrow D_{\mathrm{rad}}^{1,2}(\mathbb{R}^{N})$.\smallskip

Recall from assumption $\left( \mathbf{K}\right) $ that $K\in L_{\mathrm{loc}%
}^{s}\left( \left( 0,+\infty \right) \right) $ for some $s>\frac{2N}{N+2}\ (=%
\frac{2^{*}}{2^{*}-1})$.

\begin{lem}
\label{Lem(corone)}Let $R>r>0$ and $1<q<\infty $. Then there exists $\tilde{C%
}=\tilde{C}\left( N,r,R,q,s\right) >0$ such that $\forall u\in H_{V,\mathrm{r%
}}^{1}$ and $\forall h\in H_{V}^{1}$ one has 
\[
\frac{\int_{B_{R}\setminus B_{r}}K\left( \left| x\right| \right) \left|
u\right| ^{q-1}\left| h\right| dx}{\tilde{C}\left\| K\left( \left| \cdot
\right| \right) \right\| _{L^{s}(B_{R}\setminus B_{r})}}\leq \left\{ 
\begin{array}{ll}
\left( \int_{B_{R}\setminus B_{r}}\left| u\right| ^{2}dx\right) ^{\frac{q-1}{%
2}}\left\| h\right\| \medskip  & \text{if }q\leq \tilde{q} \\ 
\left( \int_{B_{R}\setminus B_{r}}\left| u\right| ^{2}dx\right) ^{\frac{%
\tilde{q}-1}{2}}\left\| u\right\| ^{q-\tilde{q}}\left\| h\right\| \quad
\medskip  & \text{if }q>\tilde{q}
\end{array}
\right. 
\]
where $\tilde{q}:=2\left( 1+\frac{1}{N}-\frac{1}{s}\right) $ (note that $s>%
\frac{2N}{N+2}$ implies $\tilde{q}>1$).
\end{lem}

\proof%
For simplicity, we denote by $\sigma $ the H\"{o}lder-conjugate exponent of $%
2^{*}$, i.e., $\sigma :=2N/\left( N+2\right) $. By H\"{o}lder inequality
(note that $\frac{s}{\sigma }>1$), we have 
\begin{eqnarray*}
&&\int_{B_{R}\setminus B_{r}}K\left( \left| x\right| \right) \left| u\right|
^{q-1}\left| h\right| dx \\
&\leq & \left( \int_{B_{R}\setminus B_{r}}K\left(
\left| x\right| \right) ^{\sigma }\left| u\right| ^{\left( q-1\right) \sigma
}dx\right) ^{\frac{1}{\sigma }}\left( \int_{B_{R}\setminus B_{r}}\left|
h\right| ^{2^{*}}dx\right) ^{\frac{1}{2^{*}}} \\
&\leq & \left( \left( \int_{B_{R}\setminus B_{r}}K\left( \left| x\right|
\right) ^{s}dx\right) ^{\frac{\sigma }{s}}\left( \int_{B_{R}\setminus
B_{r}}\left| u\right| ^{\left( q-1\right) \sigma \left( \frac{s}{\sigma }%
\right) ^{\prime }}dx\right) ^{\frac{1}{\left( \frac{s}{\sigma }\right)
^{\prime }}}\right) ^{\frac{1}{\sigma }}S_{N}\left\| h\right\| \\
&\leq & S_{N}\left\| K\left( \left| \cdot \right| \right) \right\|
_{L^{s}(B_{R}\setminus B_{r})}\left\| h\right\| \left( \int_{B_{R}\setminus
B_{r}}\left| u\right| ^{2\frac{q-1}{\tilde{q}-1}}dx\right) ^{\frac{\tilde{q}%
-1}{2}},
\end{eqnarray*}
where we computed $\sigma \left( \frac{s}{\sigma }\right) ^{\prime }=\frac{%
2Ns}{\left( N+2\right) s-2N}=\frac{2}{\tilde{q}-1}$. If $q\leq \tilde{q}$,
then we get 
\begin{eqnarray*}
&&\int_{B_{R}\setminus B_{r}}K\left( \left| x\right| \right) \left| u\right|
^{q-1}\left| h\right| dx \\
&\leq &S_{N}\left\| K\left( \left| \cdot \right|
\right) \right\| _{L^{s}(B_{R}\setminus B_{r})}\left\| h\right\| \left(
\left| B_{R}\setminus B_{r}\right| ^{1-\frac{q-1}{\tilde{q}-1}}\left(
\int_{B_{R}\setminus B_{r}}\left| u\right| ^{2}dx\right) ^{\frac{q-1}{\tilde{%
q}-1}}\right) ^{\frac{\tilde{q}-1}{2}} \\
&=&S_{N}\left\| K\left( \left| \cdot \right| \right) \right\|
_{L^{s}(B_{R}\setminus B_{r})}\left\| h\right\| \left| B_{R}\setminus
B_{r}\right| ^{\frac{\tilde{q}-q}{2}}\left( \int_{B_{R}\setminus
B_{r}}\left| u\right| ^{2}dx\right) ^{\frac{q-1}{2}}.
\end{eqnarray*}
Otherwise, if $q>\tilde{q}$, then by (\ref{PointwiseEstimate}) we obtain 
\begin{eqnarray*}
&&\int_{B_{R}\setminus B_{r}}K\left( \left| x\right| \right) \left| u\right|
^{q-1}\left| h\right| dx \\
&\leq &S_{N}\left\| K\left( \left| \cdot \right|
\right) \right\| _{L^{s}(B_{R}\setminus B_{r})}\left\| h\right\| \left(
\int_{B_{R}\setminus B_{r}}\left| u\right| ^{2\frac{q-1}{\tilde{q}-1}%
-2}\left| u\right| ^{2}dx\right) ^{\frac{\tilde{q}-1}{2}} \\
&\leq &S_{N}\left\| K\left( \left| \cdot \right| \right) \right\|
_{L^{s}(B_{R}\setminus B_{r})}\left\| h\right\| \left( \left( \frac{%
C_{N}\left\| u\right\| }{r^{\frac{N-2}{2}}}\right) ^{2\frac{q-\tilde{q}}{%
\tilde{q}-1}}\int_{B_{R}\setminus B_{r}}\left| u\right| ^{2}dx\right) ^{%
\frac{\tilde{q}-1}{2}} \\
&=&S_{N}\left\| K\left( \left| \cdot \right| \right) \right\|
_{L^{s}(B_{R}\setminus B_{r})}\left\| h\right\| \left( \frac{C_{N}\left\|
u\right\| }{r^{\frac{N-2}{2}}}\right) ^{q-\tilde{q}}\left(
\int_{B_{R}\setminus B_{r}}\left| u\right| ^{2}dx\right) ^{\frac{\tilde{q}-1%
}{2}}.
\end{eqnarray*}
This concludes the proof.%
\endproof%

For future reference, we recall here some results from \cite{BPR} concerning
the sum space 
\[
L_{K}^{p_{1}}+L_{K}^{p_{2}}:=L_{K}^{p_{1}}\left( \mathbb{R}^{N}\right)
+L_{K}^{p_{2}}\left( \mathbb{R}^{N}\right) :=\left\{ u_{1}+u_{2}:u_{1}\in
L_{K}^{p_{1}}\left( \mathbb{R}^{N}\right) ,\,u_{2}\in L_{K}^{p_{2}}\left( \mathbb{R%
}^{N}\right) \right\} , 
\]
where we assume $1<p_{1}\leq p_{2}<\infty $. Such a space can be
characterized as the set of the measurable mappings $u:\mathbb{R}%
^{N}\rightarrow \mathbb{R}$ for which there exists a measurable set $E\subseteq 
\mathbb{R}^{N}$ such that $u\in L_{K}^{p_{1}}\left( E\right) \cap
L_{K}^{p_{2}}\left( E^{c}\right) $ (of course $L_{K}^{p_{1}}\left( E\right)
:=L^{p_{1}}(E,K\left( \left| x\right| \right) dx)$, and so for $%
L_{K}^{p_{2}}\left( E^{c}\right) $). It is a Banach space with respect to
the norm 
\[
\left\| u\right\| _{L_{K}^{p_{1}}+L_{K}^{p_{2}}}:=\inf_{u_{1}+u_{2}=u}\max
\left\{ \left\| u_{1}\right\| _{L_{K}^{p_{1}}},\left\| u_{2}\right\|
_{L_{K}^{p_{2}}}\right\} 
\]
and the continuous embedding $L_{K}^{p}\hookrightarrow
L_{K}^{p_{1}}+L_{K}^{p_{2}}$ holds for all $p\in \left[ p_{1},p_{2}\right] $.

\begin{prop}[see {\cite[Proposition 2.7]{BPR}}] 
\label{Prop(->0)}
Let $\left\{ u_{n}\right\}
\subseteq L_{K}^{p_{1}}+L_{K}^{p_{2}}$ be a sequence such that $\forall
\varepsilon >0$ there exist $n_{\varepsilon }>0$ and a sequence of
measurable sets $E_{\varepsilon ,n}\subseteq \mathbb{R}^{N}$ satisfying 
\begin{equation}
\forall n>n_{\varepsilon },\quad \int_{E_{\varepsilon ,n}}K\left( \left|
x\right| \right) \left| u_{n}\right| ^{p_{1}}dx+\int_{E_{\varepsilon
,n}^{c}}K\left( \left| x\right| \right) \left| u_{n}\right|
^{p_{2}}dx<\varepsilon .  \label{Prop(->0): cond}
\end{equation}
Then $u_{n}\rightarrow 0$ in $L_{K}^{p_{1}}+L_{K}^{p_{2}}$.
\end{prop}

\begin{prop}[{\cite[Propositions 2.17 and 2.14, Corollary 2.19]{BPR}}]
\label{Prop(L+L)}
Let $E\subseteq \mathbb{R}^{N}$ be a measurable set.

\begin{itemize}
\item[(i)]  If $\int_{E}K\left( \left| x\right| \right) dx<\infty $, then $%
L_{K}^{p_{1}}+L_{K}^{p_{2}}$ is continuously embedded into $%
L_{K}^{p_{1}}\left( E\right) $.

\item[(ii)]  Every $u\in (L_{K}^{p_{1}}+L_{K}^{p_{2}})\cap L^{\infty }\left(
E\right) $ satisfies 
\begin{equation}
\left\| u\right\| _{L_{K}^{p_{2}}\left( E\right) }^{p_{2}/p_{1}}\leq \left(
\left\| u\right\| _{L^{\infty }\left( E\right) }^{p_{2}/p_{1}-1}+\left\|
u\right\| _{L_{K}^{p_{2}}\left( E\right) }^{p_{2}/p_{1}-1}\right) \left\|
u\right\| _{L_{K}^{p_{1}}+L_{K}^{p_{2}}}.  \label{PropLL:1}
\end{equation}
If moreover $\left\| u\right\| _{L^{\infty }\left( E\right) }\leq 1$, then 
\begin{equation}
\left\| u\right\| _{L_{K}^{p_{2}}\left( E\right) }\leq 2\left\| u\right\|
_{L_{K}^{p_{1}}+L_{K}^{p_{2}}}+1.  \label{PropLL:2}
\end{equation}
\end{itemize}
\end{prop}

Recall the definitions (\ref{S_o :=})-(\ref{S_i :=}) of the functions $\mathcal{S}_{0}$ and $%
\mathcal{S}_{\infty }$.

\proof[Proof of Theorem \ref{THM(cpt)}]
We prove each part of the theorem separately.\smallskip

\noindent (i) By the monotonicity of $\mathcal{S}_{0}$ and $\mathcal{S}%
_{\infty }$, it is not restrictive to assume $R_{1}<R_{2}$ in hypothesis $%
\left( \mathcal{S}_{q_{1},q_{2}}^{\prime }\right) $. In order to prove the
continuous embedding, let $u\in H_{V,\mathrm{r}}^{1}$, $u\neq 0$. Then we
have 
\begin{equation}
\int_{B_{R_{1}}}K\left( \left| x\right| \right) \left| u\right|
^{q_{1}}dx=\left\| u\right\| ^{q_{1}}\int_{B_{R_{1}}}K\left( \left| x\right|
\right) \frac{\left| u\right| ^{q_{1}}}{\left\| u\right\| ^{q_{1}}}dx\leq
\left\| u\right\| ^{q_{1}}\mathcal{S}_{0}\left( q_{1},R_{1}\right) 
\label{pf1}
\end{equation}
and, similarly, 
\begin{equation}
\int_{B_{R_{2}}^{c}}K\left( \left| x\right| \right) \left| u\right|
^{q_{2}}dx\leq \left\| u\right\| ^{q_{2}}\mathcal{S}_{\infty }\left(
q_{2},R_{2}\right) .  \label{pf2}
\end{equation}
On the other hand, by Lemma \ref{Lem(corone)} and the continuous embedding $%
D_{\mathrm{rad}}^{1,2}(\mathbb{R}^{N})\hookrightarrow L_{\mathrm{loc}}^{2}(\mathbb{%
R}^{N})$, we deduce that there exists a constant $C_{1}>0$, independent from 
$u$, such that 
\begin{equation}
\int_{B_{R_{2}}\setminus B_{R_{1}}}K\left( \left| x\right| \right) \left|
u\right| ^{q_{1}}dx\leq C_{1}\left\| u\right\| ^{q_{1}}.  \label{pf3}
\end{equation}
Hence $u\in L_{K}^{q_{1}}(B_{R_{2}})\cap L_{K}^{q_{2}}(B_{R_{2}}^{c})$ and
thus $u\in L_{K}^{q_{1}}+L_{K}^{q_{2}}$. Moreover, if $u_{n}\rightarrow 0$
in $H_{V,\mathrm{r}}^{1}$, then, using (\ref{pf1}), (\ref{pf2}) and (\ref
{pf3}), we get 
\[
\int_{B_{R_{2}}}K\left( \left| x\right| \right) \left| u_{n}\right|
^{q_{1}}dx+\int_{B_{R_{2}}^{c}}K\left( \left| x\right| \right) \left|
u_{n}\right| ^{q_{2}}dx=o\left( 1\right) _{n\rightarrow \infty },
\]
which means $u_{n}\rightarrow 0$ in $L_{K}^{q_{1}}+L_{K}^{q_{2}}$ by
Proposition \ref{Prop(->0)}. \emph{\smallskip }

\noindent (ii) Assume hypothesis $\left( \mathcal{S}_{q_{1},q_{2}}^{\prime
\prime }\right) $. Let $\varepsilon >0$ and let $u_{n}\rightharpoonup 0$ in $%
H_{V,\mathrm{r}}^{1}$. Then $\left\{ \left\| u_{n}\right\| \right\} $ is
bounded and, arguing as for (\ref{pf1}) and (\ref{pf2}), we can take $%
r_{\varepsilon }>0$ and $R_{\varepsilon }>r_{\varepsilon }$ such that for
all $n$ one has 
\[
\int_{B_{r_{\varepsilon }}}K\left( \left| x\right| \right) \left|
u_{n}\right| ^{q_{1}}dx\leq \left\| u_{n}\right\| ^{q_{1}}\mathcal{S}%
_{0}\left( q_{1},r_{\varepsilon }\right) \leq \sup_{n}\left\| u_{n}\right\|
^{q_{1}}\mathcal{S}_{0}\left( q_{1},r_{\varepsilon }\right) <\frac{%
\varepsilon }{3}
\]
and 
\[
\int_{B_{R_{\varepsilon }}^{c}}K\left( \left| x\right| \right) \left|
u_{n}\right| ^{q_{2}}dx\leq \sup_{n}\left\| u_{n}\right\| ^{q_{2}}\mathcal{S}%
_{\infty }\left( q_{2},R_{\varepsilon }\right) <\frac{\varepsilon }{3}.
\]
Using Lemma \ref{Lem(corone)} and the boundedness of $\left\{ \left\|
u_{n}\right\| \right\} $ again, we infer that there exist two constants $%
C_{2},l>0$, independent from $n$, such that 
\[
\int_{B_{R_{\varepsilon }}\setminus B_{r_{\varepsilon }}}K\left( \left|
x\right| \right) \left| u_{n}\right| ^{q_{1}}dx\leq C_{2}\left(
\int_{B_{R_{\varepsilon }}\setminus B_{r_{\varepsilon }}}\left| u_{n}\right|
^{2}dx\right) ^{l},
\]
where 
\[
\int_{B_{R_{\varepsilon }}\setminus B_{r_{\varepsilon }}}\left| u_{n}\right|
^{2}dx\rightarrow 0\quad \text{as }n\rightarrow \infty \quad \text{(}%
\varepsilon ~\text{fixed)}
\]
thanks to the compactness of the embedding $D_{\mathrm{rad}}^{1,2}(\mathbb{R}%
^{N})\hookrightarrow L_{\mathrm{loc}}^{2}(\mathbb{R}^{N})$. Therefore we obtain 
\[
\int_{B_{R_{\varepsilon }}}K\left( \left| x\right| \right) \left|
u_{n}\right| ^{q_{1}}dx+\int_{B_{R_{\varepsilon }}^{c}}K\left( \left|
x\right| \right) \left| u_{n}\right| ^{q_{2}}dx<\varepsilon 
\]
for all $n$ sufficiently large, which means $u_{n}\rightarrow 0$ in $%
L_{K}^{q_{1}}+L_{K}^{q_{2}}$ (Proposition \ref{Prop(->0)}). This concludes
the proof of part (ii).\smallskip

\noindent (iii) First we observe that $K\left( \left| \cdot \right| \right)
\in L^{1}(B_{1})$ and assumption $\left( \mathbf{K}\right) $ imply $K\left(
\left| \cdot \right| \right) \in L_{\mathrm{loc}}^{1}(\mathbb{R}^{N})$. Assume $%
H_{V,\mathrm{r}}^{1}\hookrightarrow L_{K}^{q_{1}}+L_{K}^{q_{2}}$ with $%
q_{1}\leq q_{2}$. Fix $R_{1}>0$. Then, by (i) of Proposition \ref{Prop(L+L)}%
, there exist two constants $c_{1},c_{2}>0$ such that $\forall u\in H_{V,%
\mathrm{r}}^{1}$ we have 
\[
\int_{B_{R_{1}}}K\left( \left| x\right| \right) \left| u\right|
^{q_{1}}dx\leq c_{1}\left\| u\right\|
_{L_{K}^{q_{1}}+L_{K}^{q_{2}}}^{q_{1}}\leq c_{2}\left\| u\right\| ^{q_{1}}, 
\]
which implies $\mathcal{S}_{0}\left( q_{1},R_{1}\right) \leq c_{2}$. By (\ref
{PointwiseEstimate}), fix $R_{2}>0$ such that every $u\in H_{V,\mathrm{r}%
}^{1}$ with $\left\| u\right\| =1$ satisfies $\left| u\left( x\right)
\right| \leq 1$ almost everywhere on $B_{R_{2}}^{c}$. Then, by (\ref
{PropLL:2}), we have 
\[
\int_{B_{R_{2}}^{c}}K\left( \left| x\right| \right) \left| u\right|
^{q_{2}}dx\leq \left( 2\left\| u\right\|
_{L_{K}^{q_{1}}+L_{K}^{q_{2}}}+1\right) ^{q_{2}}\leq \left( c_{3}\left\|
u\right\| +1\right) ^{q_{2}}=\left( c_{3}+1\right) ^{q_{2}} 
\]
for some constant $c_{3}>0$. This gives $\mathcal{S}_{\infty }\left(
q_{2},R_{2}\right) <\infty $ and thus $\left( \mathcal{S}_{q_{1},q_{2}}^{%
\prime }\right) $ holds (with $R_{1}>0$ arbitrary and $R_{2}$ large enough).
Now assume that the embedding $H_{V,\mathrm{r}}^{1}\hookrightarrow
L_{K}^{q_{1}}+L_{K}^{q_{2}}$ is compact and, by contradiction, that $%
\lim_{R\rightarrow 0^{+}}\mathcal{S}_{0}\left( q_{1},R\right) >\varepsilon
_{1}>0$ (the limit exists by monotonicity). Then for every $n\in \mathbb{N}%
\setminus \left\{ 0\right\} $ we have $\mathcal{S}_{0}\left(
q_{1},1/n\right) >\varepsilon _{1}$ and thus there exists $u_{n}\in H_{V,%
\mathrm{r}}^{1}$ such that $\left\| u_{n}\right\| =1$ and 
\[
\int_{B_{1/n}}K\left( \left| x\right| \right) \left| u_{n}\right|
^{q_{1}}dx>\varepsilon _{1}. 
\]
Since $\left\{ u_{n}\right\} $ is bounded in $H_{V,\mathrm{r}}^{1}$, by the
compactness assumption together with the continuous embedding $%
L_{K}^{q_{1}}+L_{K}^{q_{2}}\hookrightarrow L_{K}^{q_{1}}(B_{1})$ ((i) of
Proposition \ref{Prop(L+L)}), we get that there exists $u\in H_{V,\mathrm{r}%
}^{1}$ such that, up to a subsequence, $u_{n}\rightarrow u$ in $L_{K}^{q_{1}}(B_{1})$. This
implies 
\[
\int_{B_{1/n}}K\left( \left| x\right| \right) \left| u_{n}\right|
^{q_{1}}dx\rightarrow 0\quad \text{as }n\rightarrow \infty , 
\]
which is a contradiction. Similarly, if $\lim_{R\rightarrow +\infty }%
\mathcal{S}_{\infty }\left( q_{2},R\right) >\varepsilon _{2}>0$, then there
exists a sequence $\left\{ u_{n}\right\} \subset H_{V,\mathrm{r}}^{1}$ such
that $\left\| u_{n}\right\| =1$ and 
\begin{equation}
\int_{B_{n}^{c}}K\left( \left| x\right| \right) \left| u_{n}\right|
^{q_{2}}dx>\varepsilon _{2}.  \label{>eps2}
\end{equation}
Moreover, we can assume that $\exists u\in H_{V,\mathrm{r}}^{1}$ such that $%
u_{n}\rightharpoonup u$ in $H_{V,\mathrm{r}}^{1}$, $u_{n}\rightarrow u$ in $%
L_{K}^{q_{1}}+L_{K}^{q_{2}}$ and 
\begin{equation}
\left\| u_{n}-u\right\| \leq \left\| u_{n}\right\| +\left\| u\right\| \leq
1+\liminf_{n\rightarrow \infty }\left\| u_{n}\right\| =2.  \label{un-u}
\end{equation}
Now, by (\ref{un-u}) and (\ref{PointwiseEstimate}), fix $R_{2}>0$ such that $%
\left| u_{n}\left( x\right) -u\left( x\right) \right| \leq 1$ almost
everywhere on $B_{R_{2}}^{c}$. Then $\left\{ u_{n}-u\right\} $ is bounded in 
$L_{K}^{q_{2}}(B_{R_{2}}^{c})$ by (\ref{PropLL:2}) and therefore (\ref
{PropLL:1}) gives 
\[
\int_{B_{R_{2}}^{c}}K\left( \left| x\right| \right) \left| u_{n}-u\right|
^{q_{2}}dx\leq c_{4}\left( \left\| u_{n}-u\right\|
_{L_{K}^{q_{1}}+L_{K}^{q_{2}}}\right) ^{q_{1}}\rightarrow 0\quad \text{as }%
n\rightarrow \infty 
\]
for some constant $c_{4}>0$. 
Since $u\in L_{K}^{q_{2}}(B_{R_{2}}^{c})$ by (\ref{PointwiseEstimate}) and (\ref{PropLL:1}), 
this implies 
\[
\int_{B_{n}^{c}}K\left( \left| x\right| \right) \left| u_{n}\right|
^{q_{2}}dx\rightarrow 0\quad \text{as }n\rightarrow \infty , 
\]
which contradicts (\ref{>eps2}). Hence we conclude $\lim_{R\rightarrow 0^{+}}%
\mathcal{S}_{0}\left( q_{1},R\right) =\lim_{R\rightarrow +\infty }\mathcal{S}%
_{\infty }\left( q_{2},R\right)$ $=0$, which completes the proof.%
\endproof%

\section{Proof of Theorems \ref{THM0}\thinspace -\thinspace \ref{THM3} \label%
{SEC:2}}

Assume $N\geq 3$ and let $V$ and $K$ be as in $\left( \mathbf{V}\right) $
and $\left( \mathbf{K}\right) $ with $V\left( r\right) <+\infty $.
As in the previous section, we fix a constant $S_{N}>0$ such that (\ref{Sobolev}) holds.

\begin{lem}
\label{Lem(Omega)}Let $\Omega \subseteq \mathbb{R}^{N}$ be a nonempty
measurable set and assume that 
\[
\Lambda :=\esssup_{x\in \Omega }\frac{K\left( \left| x\right|
\right) }{\left| x\right| ^{\alpha }V\left( \left| x\right| \right) ^{\beta }%
}<+\infty \quad \text{for some }0\leq \beta \leq 1\text{~and }\alpha \in 
\mathbb{R}.
\]
Let $u\in H_{V}^{1}$ and assume that there exist $\nu \in \mathbb{R}$ and $m>0$
such that 
\[
\left| u\left( x\right) \right| \leq \frac{m}{\left| x\right| ^{\nu }}\quad 
\text{almost everywhere on }\Omega .
\]
Then $\forall h\in H_{V}^{1}$ and $\forall q>\max \left\{ 1,2\beta \right\} $%
, one has 

$\displaystyle\int_{\Omega }K\left( \left| x\right| \right) \left| u\right| ^{q-1}\left|
h\right| dx$
\[
\leq \left\{ 
\begin{array}{ll}
\Lambda m^{q-1}S_{N}^{1-2\beta }\left( \int_{\Omega }\left| x\right| ^{\frac{%
\alpha -\nu \left( q-1\right) }{N+2\left( 1-2\beta \right) }2N}dx\right) ^{%
\frac{N+2\left( 1-2\beta \right) }{2N}}\left\| h\right\| \quad \medskip  & 
\text{if }0\leq \beta \leq \frac{1}{2} \\ 
\Lambda m^{q-2\beta }\left( \int_{\Omega }\left| x\right| ^{\frac{\alpha
-\nu \left( q-2\beta \right) }{1-\beta }}dx\right) ^{1-\beta }\left\|
u\right\| ^{2\beta -1}\left\| h\right\| \medskip  & \text{if }\frac{1}{2}%
<\beta <1 \\ 
\Lambda m^{q-2}\left( \int_{\Omega }\left| x\right| ^{2\alpha -2\nu \left(
q-2\right) }V\left( \left| x\right| \right) \left| u\right| ^{2}dx\right) ^{%
\frac{1}{2}}\left\| h\right\|  & \text{if }\beta =1.
\end{array}
\right. 
\]
\end{lem}

\begin{proof}
We distinguish several cases, where we will use H\"{o}lder inequality many
times, without explicitly noting it. \smallskip

\noindent \emph{Case }$\beta =0$\emph{. }

\noindent We have 
{\allowdisplaybreaks
\begin{eqnarray*}
\frac{1}{\Lambda }\int_{\Omega }K\left( \left| x\right| \right) \left|
u\right| ^{q-1}\left| h\right| dx 
&\leq & \int_{\Omega }\left| x\right|^{\alpha }\left| u\right| ^{q-1}\left| h\right| dx \\
&\leq & \left( \int_{\Omega
}\left( \left| x\right| ^{\alpha }\left| u\right| ^{q-1}\right) ^{\frac{2N}{%
N+2}}dx\right) ^{\frac{N+2}{2N}}\left( \int_{\Omega }\left| h\right|
^{2^{*}}dx\right) ^{\frac{1}{2^{*}}} \\
&\leq & m^{q-1}S_{N}\left( \int_{\Omega }\left| x\right| ^{\frac{\alpha -\nu
\left( q-1\right) }{N+2}2N}dx\right) ^{\frac{N+2}{2N}}\left\| h\right\| .
\end{eqnarray*}
}

\noindent \emph{Case }$0<\beta <1/2$\emph{.}

\noindent One has $\frac{1}{\beta }>1$ and $\frac{1-\beta }{1-2\beta }%
2^{*}>1 $, with H\"{o}lder conjugate exponents $\left( \frac{1}{\beta }%
\right) ^{\prime }=\frac{1}{1-\beta }$ and $\left( \frac{1-\beta }{1-2\beta }%
2^{*}\right) ^{\prime }=\frac{2N\left( 1-\beta \right) }{N+2\left( 1-2\beta
\right) }$. Then we get 
{\allowdisplaybreaks
\begin{eqnarray*}
&&\frac{1}{\Lambda }\int_{\Omega }K\left( \left| x\right| \right) \left|
u\right| ^{q-1}\left| h\right| dx \\
&\leq &\int_{\Omega }\left| x\right| ^{\alpha }V\left( \left| x\right|
\right) ^{\beta }\left| u\right| ^{q-1}\left| h\right| dx=\int_{\Omega
}\left| x\right| ^{\alpha }\left| u\right| ^{q-1}\left| h\right| ^{1-2\beta
}V\left( \left| x\right| \right) ^{\beta }\left| h\right| ^{2\beta }dx \\
&\leq &\left( \int_{\Omega }\left( \left| x\right| ^{\alpha }\left| u\right|
^{q-1}\left| h\right| ^{1-2\beta }\right) ^{\frac{1}{1-\beta }}dx\right)
^{1-\beta }\left( \int_{\Omega }V\left( \left| x\right| \right) \left|
h\right| ^{2}dx\right) ^{\beta } \\
&\leq &\left( \int_{\Omega }\left( \left| x\right| ^{\alpha }\left| u\right|
^{q-1}\left| h\right| ^{1-2\beta }\right) ^{\frac{1}{1-\beta }}dx\right)
^{1-\beta }\left\| h\right\| ^{2\beta } \\
&\leq &\left( \left( \int_{\Omega }\left( \left| x\right| ^{\frac{\alpha }{%
1-\beta }}\left| u\right| ^{\frac{q-1}{1-\beta }}\right) ^{\left( \frac{%
1-\beta }{1-2\beta }2^{*}\right) ^{\prime }}dx\right) ^{\frac{1}{\left( 
\frac{1-\beta }{1-2\beta }2^{*}\right) ^{\prime }}}\left( \int_{\Omega
}\left| h\right| ^{2^{*}}dx\right) ^{\frac{1-2\beta }{\left( 1-\beta \right)
2^{*}}}\right) ^{1-\beta }\left\| h\right\| ^{2\beta } \\
&\leq &m^{q-1}\left( \left( \int_{\Omega }\left( \left| x\right| ^{\frac{%
\alpha }{1-\beta }-\nu \frac{q-1}{1-\beta }}\right) ^{\left( \frac{1-\beta }{%
1-2\beta }2^{*}\right) ^{\prime }}dx\right) ^{\frac{1}{\left( \frac{1-\beta 
}{1-2\beta }2^{*}\right) ^{\prime }}}S_{N}^{\frac{1-2\beta }{1-\beta }%
}\left\| h\right\| ^{\frac{1-2\beta }{1-\beta }}\right) ^{1-\beta }\left\|
h\right\| ^{2\beta } \\
&=&m^{q-1}\left( \int_{\Omega }\left| x\right| ^{\frac{\alpha -\nu \left(
q-1\right) }{N+2\left( 1-2\beta \right) }2N}dx\right) ^{\frac{N+2\left(
1-2\beta \right) }{2N}}S_{N}^{1-2\beta }\left\| h\right\| .
\end{eqnarray*}
}

\noindent \emph{Case }$\beta =\frac{1}{2}$\emph{.}

\noindent We have 
{\allowdisplaybreaks
\begin{eqnarray*}
\frac{1}{\Lambda }\int_{\Omega }K\left( \left| x\right| \right) \left|
u\right| ^{q-1}\left| h\right| dx &\leq &\int_{\Omega }\left| x\right|
^{\alpha }\left| u\right| ^{q-1}V\left( \left| x\right| \right) ^{\frac{1}{2}%
}\left| h\right| dx \\
&\leq &\left( \int_{\Omega }\left| x\right| ^{2\alpha }\left| u\right|
^{2\left( q-1\right) }dx\right) ^{\frac{1}{2}}\left( \int_{\Omega }V\left(
\left| x\right| \right) \left| h\right| ^{2}dx\right) ^{\frac{1}{2}} \\
&\leq &m^{q-1}\left( \int_{\Omega }\left| x\right| ^{2\alpha -2\nu \left(
q-1\right) }dx\right) ^{\frac{1}{2}}\left\| h\right\| .
\end{eqnarray*}
}

\noindent \emph{Case }$1/2<\beta <1$\emph{.}

\noindent One has $\frac{1}{2\beta -1}>1$, with H\"{o}lder conjugate
exponent $\left( \frac{1}{2\beta -1}\right) ^{\prime }=\frac{1}{2\left(
1-\beta \right) }$. Then 
{\allowdisplaybreaks
\begin{eqnarray*}
&&\frac{1}{\Lambda }\int_{\Omega }K\left( \left| x\right| \right) \left|
u\right| ^{q-1}\left| h\right| dx \\
&\leq &\int_{\Omega }\left| x\right|
^{\alpha }V\left( \left| x\right| \right) ^{\beta }\left| u\right|
^{q-1}\left| h\right| dx=\int_{\Omega }\left| x\right| ^{\alpha }V\left(
\left| x\right| \right) ^{\frac{2\beta -1}{2}}\left| u\right| ^{q-1}V\left(
\left| x\right| \right) ^{\frac{1}{2}}\left| h\right| dx \\
&\leq &\left( \int_{\Omega }\left| x\right| ^{2\alpha }V\left( \left|
x\right| \right) ^{2\beta -1}\left| u\right| ^{2\left( q-1\right) }dx\right)
^{\frac{1}{2}}\left( \int_{\Omega }V\left( \left| x\right| \right) \left|
h\right| ^{2}dx\right) ^{\frac{1}{2}} \\
&\leq &\left( \int_{\Omega }\left| x\right| ^{2\alpha }\left| u\right|
^{2\left( q-2\beta \right) }V\left( \left| x\right| \right) ^{2\beta
-1}\left| u\right| ^{2\left( 2\beta -1\right) }dx\right) ^{\frac{1}{2}%
}\left\| h\right\| \\
&\leq &\left( \left( \int_{\Omega }\left| x\right| ^{\frac{\alpha }{1-\beta }%
}\left| u\right| ^{\frac{q-2\beta }{1-\beta }}dx\right) ^{2\left( 1-\beta
\right) }\left( \int_{\Omega }V\left( \left| x\right| \right) \left|
u\right| ^{2}dx\right) ^{2\beta -1}\right) ^{\frac{1}{2}}\left\| h\right\| \\
&\leq &m^{q-2\beta }\left( \left( \int_{\Omega }\left| x\right| ^{\frac{%
\alpha }{1-\beta }-\nu \frac{q-2\beta }{1-\beta }}dx\right) ^{2\left(
1-\beta \right) }\left( \int_{\Omega }V\left( \left| x\right| \right) \left|
u\right| ^{2}dx\right) ^{2\beta -1}\right) ^{\frac{1}{2}}\left\| h\right\| \\
&=&m^{q-2\beta }\left( \int_{\Omega }\left| x\right| ^{\frac{\alpha -\nu
(q-2\beta )}{1-\beta }}dx\right) ^{1-\beta }\left( \int_{\Omega }V\left(
\left| x\right| \right) \left| u\right| ^{2}dx\right) ^{\frac{2\beta -1}{2}%
}\left\| h\right\| \\
&\leq &m^{q-2\beta }\left( \int_{\Omega }\left| x\right| ^{\frac{\alpha -\nu
(q-2\beta )}{1-\beta }}dx\right) ^{1-\beta }\left\| u\right\| ^{2\beta
-1}\left\| h\right\| .
\end{eqnarray*}
}

\noindent \emph{Case }$\beta =1$\emph{.}

\noindent Assumption $q>\max \left\{ 1,2\beta \right\} $ means $q>2$ and
thus we have 
{\allowdisplaybreaks
\begin{eqnarray*}
\frac{1}{\Lambda }\int_{\Omega }K\left( \left| x\right| \right) \left|
u\right| ^{q-1}\left| h\right| dx 
&\leq &\int_{\Omega }\left| x\right|
^{\alpha }V\left( \left| x\right| \right) \left| u\right| ^{q-1}\left|
h\right| dx\\
&=&\int_{\Omega }\left| x\right| ^{\alpha }V\left( \left| x\right|
\right) ^{\frac{1}{2}}\left| u\right| ^{q-1}V\left( \left| x\right| \right)
^{\frac{1}{2}}\left| h\right| dx \\
&\leq &\left( \int_{\Omega }\left| x\right| ^{2\alpha }V\left( \left|
x\right| \right) \left| u\right| ^{2\left( q-1\right) }dx\right) ^{\frac{1}{2%
}}\left( \int_{\Omega }V\left( \left| x\right| \right) \left| h\right|
^{2}dx\right) ^{\frac{1}{2}} \\
&\leq &\left( \int_{\Omega }\left| x\right| ^{2\alpha }\left| u\right|
^{2\left( q-2\right) }V\left( \left| x\right| \right) \left| u\right|
^{2}dx\right) ^{\frac{1}{2}}\left\| h\right\| \\
&\leq &m^{q-2}\left( \int_{\Omega }\left| x\right| ^{2\alpha -2\nu \left(
q-2\right) }V\left( \left| x\right| \right) \left| u\right| ^{2}dx\right) ^{%
\frac{1}{2}}\left\| h\right\| .
\end{eqnarray*}
}
\end{proof}

As in the previous section, we fix a constant $C_{N}>0$ such that (\ref
{PointwiseEstimate}) holds. Recall the definitions (\ref{N_o})-(\ref{N_i})
of the functions $\mathcal{R}_{0}$ and $\mathcal{R}_{\infty }$.

\begin{proof}[Proof of Theorem \ref{THM0}]
Assume the hypotheses of the theorem and let $u\in H_{V,\mathrm{r}}^{1}$ and $h\in H_{V}^{1}$ be
such that $\left\| u\right\| =\left\| h\right\| =1$. Let $0<R\leq R_{1}$. We
will denote by $C$ any positive constant which does not depend on $u$, $h$
and $R$.

By (\ref{PointwiseEstimate}) and the fact that 
\[
\esssup_{x\in B_{R}}\frac{K\left( \left| x\right| \right) }{\left|
x\right| ^{\alpha _{0}}V\left( \left| x\right| \right) ^{\beta _{0}}}\leq 
\esssup_{r\in \left( 0,R_{1}\right) }\frac{K\left( r\right) }{%
r^{\alpha _{0}}V\left( r\right) ^{\beta _{0}}}<+\infty , 
\]
we can apply Lemma \ref{Lem(Omega)} with $\Omega =B_{R}$, $\alpha =\alpha
_{0}$, $\beta =\beta _{0}$, $m=C_{N}\left\| u\right\| =C_{N}$ and $\nu =%
\frac{N-2}{2}$. If $0\leq \beta _{0}\leq 1/2$ we get 
\begin{eqnarray*}
\int_{B_{R}}K\left( \left| x\right| \right) \left| u\right| ^{q_{1}-1}\left|
h\right| dx &\leq &C\left( \int_{B_{R}}\left| x\right| ^{\frac{\alpha _{0}-%
\frac{N-2}{2}\left( q_{1}-1\right) }{N+2\left( 1-2\beta _{0}\right) }%
2N}dx\right) ^{\frac{N+2\left( 1-2\beta _{0}\right) }{2N}} \\
&\leq &C\left( \int_{0}^{R}r^{\frac{2\alpha _{0}-\left( N-2\right) \left(
q_{1}-1\right) }{N+2\left( 1-2\beta _{0}\right) }N+N-1}dr\right) ^{\frac{%
N+2\left( 1-2\beta _{0}\right) }{2N}} \\
&=&C\left( R^{\frac{2\alpha _{0}-4\beta _{0}+2N-\left( N-2\right) q_{1}}{%
N+2\left( 1-2\beta _{0}\right) }N}\right) ^{\frac{N+2\left( 1-2\beta
_{0}\right) }{2N}},
\end{eqnarray*}
since 
\[
2\alpha _{0}-4\beta _{0}+2N-\left( N-2\right) q_{1}=\left( N-2\right) \left(
q^{*}\left( \alpha _{0},\beta _{0}\right) -q_{1}\right) >0. 
\]
On the other hand, if $1/2<\beta _{0}<1$ we have 
\begin{eqnarray*}
\int_{B_{R}}K\left( \left| x\right| \right) \left| u\right| ^{q_{1}-1}\left|
h\right| dx 
&\leq & C\left( \int_{B_{R}}\left| x\right| ^{\frac{\alpha _{0}-%
\frac{N-2}{2}\left( q_{1}-2\beta _{0}\right) }{1-\beta _{0}}}dx\right)^{1-\beta _{0}}\\
&\leq & C\left( \int_{0}^{R}r^{\frac{\alpha _{0}-\frac{N-2}{2}%
\left( q_{1}-2\beta _{0}\right) }{1-\beta _{0}}+N-1}dr\right) ^{1-\beta _{0}}
\\
&=&C\left( R^{\frac{2\alpha _{0}-\left( N-2\right) \left( q_{1}-2\beta
_{0}\right) }{2\left( 1-\beta _{0}\right) }+N}\right) ^{1-\beta _{0}},
\end{eqnarray*}
since 
\[
\frac{2\alpha _{0}-\left( N-2\right) \left( q_{1}-2\beta _{0}\right) }{%
2\left( 1-\beta _{0}\right) }+N=\frac{N-2}{2\left( 1-\beta _{0}\right) }%
\left( q^{*}\left( \alpha _{0},\beta _{0}\right) -q_{1}\right) >0. 
\]
Finally, if $\beta _{0}=1$, we obtain 
\begin{eqnarray*}
\int_{B_{R}}K\left( \left| x\right| \right) \left| u\right| ^{q_{1}-1}\left|
h\right| dx &\leq &C\left( \int_{B_{R}}\left| x\right| ^{2\alpha _{0}-\left(
N-2\right) \left( q_{1}-2\right) }V\left( \left| x\right| \right) \left|
u\right| ^{2}dx\right) ^{\frac{1}{2}} \\
&\leq &C\left( R^{2\alpha _{0}-\left( N-2\right) \left( q_{1}-2\right)
}\int_{B_{R}}V\left( \left| x\right| \right) \left| u\right| ^{2}dx\right) ^{%
\frac{1}{2}}\\
&\leq & CR^{\frac{2\alpha _{0}-\left( N-2\right) \left(q_{1}-2\right) }{2}},
\end{eqnarray*}
since 
\[
2\alpha _{0}-\left( N-2\right) \left( q_{1}-2\right) =\left( N-2\right)
\left( q^{*}\left( \alpha _{0},1\right) -q_{1}\right) >0. 
\]
So, in any case, we deduce $\mathcal{R}_{0}\left( q_{1},R\right) \leq
CR^{\delta }$ for some $\delta =\delta \left( N,\alpha _{0},\beta
_{0},q_{1}\right) >0$ and this concludes the proof.%
\end{proof}

\proof[Proof of Theorem \ref{THM1}]
Assume the hypotheses of the theorem and let $u\in H_{V,\mathrm{r}}^{1}$ and $h\in H_{V}^{1}$ be
such that $\left\| u\right\| =\left\| h\right\| =1$. Let $R\geq R_{2}$. We
will denote by $C$ any positive constant which does not depend on $u$, $h$
and $R$.

By (\ref{PointwiseEstimate}) and the fact that 
\[
\esssup_{x\in B_{R}^{c}}\frac{K\left( \left| x\right| \right) }{%
\left| x\right| ^{\alpha _{\infty }}V\left( \left| x\right| \right) ^{\beta
_{\infty }}}\leq \esssup_{r>R_{2}}\frac{K\left( r\right) }{%
r^{\alpha _{\infty }}V\left( r\right) ^{\beta _{\infty }}}<+\infty , 
\]
we can apply Lemma \ref{Lem(Omega)} with $\Omega =B_{R}^{c}$, $\alpha
=\alpha _{\infty }$, $\beta =\beta _{\infty }$, $m=C_{N}\left\| u\right\|
=C_{N}$ and $\nu =\frac{N-2}{2}$. If $0\leq \beta _{\infty }\leq 1/2$ we get 
\begin{eqnarray*}
\int_{B_{R}^{c}}K\left( \left| x\right| \right) \left| u\right|
^{q_{2}-1}\left| h\right| dx &\leq &C\left( \int_{B_{R}^{c}}\left| x\right|
^{\frac{\alpha _{\infty }-\frac{N-2}{2}\left( q_{2}-1\right) }{N+2\left(
1-2\beta _{\infty }\right) }2N}dx\right) ^{\frac{N+2\left( 1-2\beta _{\infty
}\right) }{2N}} \\
&\leq &C\left( \int_{R}^{+\infty }r^{\frac{2\alpha _{\infty }-\left(
N-2\right) \left( q_{2}-1\right) }{N+2\left( 1-2\beta _{\infty }\right) }%
N+N-1}dr\right) ^{\frac{N+2\left( 1-2\beta _{\infty }\right) }{2N}} \\
&=&C\left( R^{\frac{2\alpha _{\infty }-4\beta _{\infty }+2N-\left(
N-2\right) q_{2}}{N+2\left( 1-2\beta _{\infty }\right) }N}\right) ^{\frac{%
N+2\left( 1-2\beta _{\infty }\right) }{2N}},
\end{eqnarray*}
since 
\[
2\alpha _{\infty }-4\beta _{\infty }+2N-\left( N-2\right) q_{2}=\left(
N-2\right) \left( q^{*}\left( \alpha _{\infty },\beta _{\infty }\right)
-q_{2}\right) <0. 
\]
On the other hand, if $1/2<\beta _{\infty }<1$ we have 
\begin{eqnarray*}
\int_{B_{R}^{c}}K\left( \left| x\right| \right) \left| u\right|
^{q_{2}-1}\left| h\right| dx &\leq &C\left( \int_{B_{R}^{c}}\left| x\right|
^{\frac{\alpha _{\infty }-\frac{N-2}{2}\left( q_{2}-2\beta _{\infty }\right) 
}{1-\beta _{\infty }}}dx\right) ^{1-\beta _{\infty }} \\
&\leq &C\left( \int_{R}^{+\infty }r^{\frac{\alpha _{\infty }-\frac{N-2}{2}%
\left( q_{2}-2\beta _{\infty }\right) }{1-\beta _{\infty }}+N-1}dr\right)
^{1-\beta _{\infty }} \\
&=&C\left( R^{\frac{2\alpha _{\infty }-\left( N-2\right) \left( q_{2}-2\beta
_{\infty }\right) }{2\left( 1-\beta _{\infty }\right) }+N}\right) ^{1-\beta
_{\infty }},
\end{eqnarray*}
since 
\[
\frac{2\alpha _{\infty }-\left( N-2\right) \left( q_{2}-2\beta _{\infty
}\right) }{2\left( 1-\beta _{\infty }\right) }+N=\frac{N-2}{2\left( 1-\beta
_{_{\infty }}\right) }\left( q^{*}\left( \alpha _{_{\infty }},\beta
_{_{\infty }}\right) -q_{2}\right) <0. 
\]
Finally, if $\beta _{\infty }=1$, we obtain 
\begin{eqnarray*}
\int_{B_{R}^{c}}K\left( \left| x\right| \right) \left| u\right|
^{q_{2}-1}\left| h\right| dx &\leq &C\left( \int_{B_{R}^{c}}\left| x\right|
^{2\alpha _{\infty }-\left( N-2\right) \left( q_{2}-2\right) }V\left( \left|
x\right| \right) \left| u\right| ^{2}dx\right) ^{\frac{1}{2}} \\
&\leq &C\left( R^{2\alpha _{\infty }-\left( N-2\right) \left( q_{2}-2\right)
}\int_{B_{R}^{c}}V\left( \left| x\right| \right) \left| u\right|
^{2}dx\right) ^{\frac{1}{2}}\\
&\leq & CR^{\frac{2\alpha _{\infty }-\left(
N-2\right) \left( q_{2}-2\right) }{2}},
\end{eqnarray*}
since 
\[
2\alpha _{\infty }-\left( N-2\right) \left( q_{2}-2\right) =\left(
N-2\right) \left( q^{*}\left( \alpha _{\infty },1\right) -q_{2}\right) <0. 
\]
This completes the proof, since in any case we get $\mathcal{R}_{\infty }\left(
q_{2},R\right) \leq CR^{\delta }$ for some $\delta =\delta \left( N,\alpha
_{\infty },\beta _{\infty },q_{2}\right) <0$.%
\endproof%

In proving Theorem \ref{THM2}, we will need the following lemma.

\begin{lem}
Assume that there exists $R_{2}>0$ such that 
\[
\lambda _{\infty }:=\essinf_{r>R_{2}}r^{\gamma _{\infty }}V\left(
r\right) >0\quad \text{for some }\gamma _{\infty }\leq 2.
\]
Then there exists a constant $c_{\infty }>0$, only dependent on $N$ and $%
\gamma _{\infty }$, such that 
\begin{equation}
\forall u\in H_{V,\mathrm{r}}^{1}\left( \mathbb{R}^{N}\right) ,\quad \left|
u\left( x\right) \right| \leq c_{\infty }\lambda _{\infty }^{-\frac{1}{4}%
}\left\| u\right\| \left| x\right| ^{-\frac{2(N-1)-\gamma _{\infty }}{4}%
}\quad \text{almost everywhere in }B_{R_{2}}^{c}.  \label{PointwiseInfty}
\end{equation}
\end{lem}

\proof%
The lemma is proved in \cite[Lemma 4]{Su-Wang-Will p} with a global
assumption on $V$, but, checking the proof, we see that the result actually
holds in the form given here.%
\endproof%

Observe that $2(N-1)-\gamma _{\infty }>0$ in (\ref{PointwiseInfty}), since $%
\gamma _{\infty }\leq 2$ and $N\geq 3$.

\proof[Proof of Theorem \ref{THM2}]
Assume the hypotheses of the theorem and denote 
\[
\Lambda _{\infty }:=\esssup_{r>R_{2}}\frac{K\left( r\right) }{%
r^{\alpha _{\infty }}V\left( r\right) ^{\beta _{\infty }}}\quad \text{and}%
\quad \lambda _{\infty }:=\essinf_{r>R_{2}}r^{\gamma _{\infty
}}V\left( r\right) . 
\]
Let $u\in H_{V,\mathrm{r}}^{1}$ and $h\in H_{V}^{1}$ be such that $\left\|
u\right\| =\left\| h\right\| =1$. Let $R\geq R_{2}$ and observe that $%
\forall \xi \geq 0$ one has 
\begin{equation}
\esssup_{r>R}\frac{K\left( r\right) }{r^{\alpha _{\infty }+\xi
\gamma _{\infty }}V\left( r\right) ^{\beta _{\infty }+\xi }}\leq 
\esssup_{r>R_{2}}\frac{K\left( r\right) }{r^{\alpha _{\infty }}V\left(
r\right) ^{\beta _{\infty }}\left( r^{\gamma _{\infty }}V\left( r\right)
\right) ^{\xi }}\leq \frac{\Lambda _{\infty }}{\lambda _{\infty }^{\xi }}%
<+\infty .  \label{stimaETA}
\end{equation}
We will denote by $C$ any positive constant which does not depend on $u$, $h$
or $R$ (such as $\Lambda _{\infty }/\lambda _{\infty }^{\xi }$ if $\xi $
does not depend on $u$, $h$ or $R$).

Denoting $\alpha _{1}=\alpha _{1}\left( \beta _{\infty },\gamma _{\infty
}\right) $, $\alpha _{2}=\alpha _{2}\left( \beta _{\infty }\right) $ and $%
\alpha _{3}=\alpha _{3}\left( \beta _{\infty },\gamma _{\infty }\right) $,
as defined in (\ref{alpha_i :=}), we will distinguish several cases,
according to the description (\ref{descrizioneThm2}). In each of such cases,
we will choose a suitable $\xi \geq 0$ and, thanks to (\ref{stimaETA}) and (%
\ref{PointwiseInfty}), we will apply Lemma \ref{Lem(Omega)} with $\Omega
=B_{R}^{c}$, $\alpha =\alpha _{\infty }+\xi \gamma _{\infty }$, $\beta
=\beta _{\infty }+\xi $ (whence $\Lambda $ will be given by the left hand
side of (\ref{stimaETA})), $m=c_{\infty }\lambda _{\infty }^{-\frac{1}{4}%
}\left\| u\right\| =c_{\infty }\lambda _{\infty }^{-\frac{1}{4}}$ and $\nu =%
\frac{2(N-1)-\gamma _{\infty }}{4}$. We will obtain that 
\[
\int_{B_{R}^{c}}K\left( \left| x\right| \right) \left| u\right|
^{q_{2}-1}\left| h\right| dx\leq CR^{\delta }
\]
for some $\delta <0$, not dependent on $R$, so that the result
follows.\medskip

\noindent \emph{Case }$\alpha _{\infty }\geq \alpha _{1}$.\smallskip

\noindent We take $\xi =1-\beta _{\infty }$ and apply Lemma \ref{Lem(Omega)}
with $\beta =\beta _{\infty }+\xi =1$ and $\alpha =\alpha _{\infty }+\xi
\gamma _{\infty }=\alpha _{\infty }+\left( 1-\beta _{\infty }\right) \gamma
_{\infty }$. We get 
\begin{eqnarray*}
\int_{B_{R}^{c}}K\left( \left| x\right| \right) \left| u\right|
^{q_{2}-1}\left| h\right| dx &\leq &C\left( \int_{B_{R}^{c}}\left| x\right|
^{2\alpha -2\nu \left( q_{2}-2\right) }V\left( \left| x\right| \right)
\left| u\right| ^{2}dx\right) ^{\frac{1}{2}} \\
&\leq &C\left( R^{2\alpha -2\nu \left( q_{2}-2\right)
}\int_{B_{R}^{c}}V\left( \left| x\right| \right) \left| u\right|
^{2}dx\right) ^{\frac{1}{2}}\leq CR^{\alpha -\nu \left( q_{2}-2\right) },
\end{eqnarray*}
since 
\begin{eqnarray*}
\alpha -\nu \left( q_{2}-2\right) &=& \alpha _{\infty }+\left( 1-\beta _{\infty
}\right) \gamma _{\infty }-\frac{2(N-1)-\gamma _{\infty }}{4}\left(
q_{2}-2\right) \\
&=&\frac{2(N-1)-\gamma _{\infty }}{4}\left( q_{**}-q_{2}\right)
<0. 
\end{eqnarray*}

\noindent \emph{Case }$\max \left\{ \alpha _{2},\alpha _{3}\right\} <\alpha
_{\infty }<\alpha _{1}$.\smallskip

\noindent Take $\xi =\frac{\alpha _{\infty }+\left( 1-\beta _{\infty
}\right) N}{N-\gamma _{\infty }}>0$ and apply Lemma \ref{Lem(Omega)} with $%
\beta =\beta _{\infty }+\xi $ and $\alpha =\alpha _{\infty }+\xi \gamma
_{\infty }$. For doing this, observe that $\alpha _{3}<\alpha _{\infty
}<\alpha _{1}$ implies 
\[
\beta =\beta _{\infty }+\xi =\frac{\alpha _{\infty }-\gamma _{\infty }\beta
_{\infty }+N}{N-\gamma _{\infty }}\in \left( \frac{1}{2},1\right) . 
\]
We get 
\[
\int_{B_{R}^{c}}K\left( \left| x\right| \right) \left| u\right|
^{q_{2}-1}\left| h\right| dx\leq C\left( \int_{B_{R}^{c}}\left| x\right| ^{%
\frac{\alpha -\nu \left( q_{2}-2\beta \right) }{1-\beta }}dx\right)
^{1-\beta }\leq C\left( R^{\frac{\alpha -\nu \left( q_{2}-2\beta \right) }{%
1-\beta }+N}\right) ^{1-\beta }, 
\]
since 
\[
\frac{\alpha -\nu \left( q_{2}-2\beta \right) }{1-\beta }+N=\frac{\nu }{%
1-\beta }\left( 2\frac{\alpha _{\infty }-\beta _{\infty }\gamma _{\infty }+N%
}{N-\gamma _{\infty }}-q_{2}\right) =\frac{\nu }{1-\beta }\left(
q_{*}-q_{2}\right) <0. 
\]

\noindent \emph{Case }$\beta _{\infty }=1$\emph{\ and }$\alpha _{\infty
}\leq 0=\alpha _{2}\,\left( =\max \left\{ \alpha _{2},\alpha _{3}\right\}
\right) $.\smallskip

\noindent Take $\xi =0$ and apply Lemma \ref{Lem(Omega)} with $\beta =\beta
_{\infty }+\xi =1$ and $\alpha =\alpha _{\infty }+\xi \gamma _{\infty
}=\alpha _{\infty }$. We get 
\[
\int_{B_{R}^{c}}K\left( \left| x\right| \right) \left| u\right|
^{q_{2}-1}\left| h\right| dx\leq C\left( \int_{B_{R}^{c}}\left| x\right|
^{2\alpha _{\infty }-2\nu \left( q_{2}-2\right) }V\left( \left| x\right|
\right) \left| u\right| ^{2}dx\right) ^{\frac{1}{2}}\leq CR^{\alpha _{\infty
}-\nu \left( q_{2}-2\right) }, 
\]
since $\alpha _{\infty }-\nu \left( q_{2}-2\right) \leq -\nu \left(
q_{2}-2\right) <0.$\medskip

\noindent \emph{Case }$\frac{1}{2}<\beta _{\infty }<1$\emph{\ and }$\alpha
_{\infty }\leq \alpha _{2}\,\left( =\max \left\{ \alpha _{2},\alpha
_{3}\right\} \right) $.\smallskip

\noindent Take $\xi =0$ again and apply Lemma \ref{Lem(Omega)} with $\beta
=\beta _{\infty }\in \left( \frac{1}{2},1\right) $ and $\alpha =\alpha
_{\infty }$. We get 
\begin{eqnarray*}
\int_{B_{R}^{c}}K\left( \left| x\right| \right) \left| u\right|^{q_{2}-1}\left| h\right| dx 
&\leq & C\left( \int_{B_{R}^{c}}\left| x\right| ^{\frac{\alpha _{\infty }-\nu \left( q_{2}-2\beta _{\infty }\right) }{1-\beta
_{\infty }}}dx\right) ^{1-\beta _{\infty }}\\
&\leq & C\left( R^{\frac{\alpha_{\infty }-\nu \left( q_{2}-2\beta _{\infty }\right) }{1-\beta _{\infty }}%
+N}\right) ^{1-\beta _{\infty }}, 
\end{eqnarray*}
since 
\[
\frac{\alpha _{\infty }-\nu \left( q_{2}-2\beta _{\infty }\right) }{1-\beta
_{\infty }}+N=\frac{\alpha _{\infty }-\alpha _{2}-\nu \left( q_{2}-2\beta
_{\infty }\right) }{1-\beta _{\infty }}<0 
\]

\noindent \emph{Case }$\beta _{\infty }\leq \frac{1}{2}$\emph{\ and }$\alpha
_{\infty }\leq \alpha _{3}\,\left( =\max \left\{ \alpha _{2},\alpha
_{3}\right\} \right) $.\smallskip

\noindent Take $\xi =\frac{1-2\beta _{\infty }}{2}\geq 0$, we can apply
Lemma \ref{Lem(Omega)} with $\beta =\beta _{\infty }+\xi =\frac{1}{2}$ and $%
\alpha =\alpha _{\infty }+\xi \gamma _{\infty }$. We get 
\[
\int_{B_{R}^{c}}K\left( \left| x\right| \right) \left| u\right|
^{q_{2}-1}\left| h\right| dx\leq C\left( \int_{B_{R}^{c}}\left| x\right|
^{2\alpha -2\nu \left( q_{2}-1\right) }dx\right) ^{\frac{1}{2}}\leq
CR^{\alpha -\nu \left( q_{2}-1\right) +\frac{N}{2}}, 
\]
since 
\[
\alpha -\nu \left( q_{2}-1\right) +\frac{N}{2}=\alpha _{\infty }+\frac{%
1-2\beta _{\infty }}{2}\gamma _{\infty }+\frac{N}{2}-\nu \left(
q_{2}-1\right) =\alpha _{\infty }-\alpha _{3}-\nu \left( q_{2}-1\right) <0. 
\]
\endproof%

The proof of Theorem \ref{THM3} will be achieved by several lemmas.

\begin{lem}
\label{LEM(pointwise0)}Assume that there exists $R>0$ such that 
\[
\lambda \left( R\right) :=\essinf_{r\in \left( 0,R\right)
}r^{\gamma _{0}}V\left( r\right) >0\quad \text{for some }\gamma _{0}\geq 2.
\]
Then there exists a constant $c_{0}>0$, only dependent on $N$ and $\gamma
_{0}$, such that $\forall u\in H_{V,\mathrm{r}}^{1}\cap D_{0}^{1,2}\left(
B_{R}\right) $ one has 
\begin{equation}
\left| u\left( x\right) \right| \leq c_{0}\left( \frac{1}{\sqrt{\lambda
\left( R\right) }}+\frac{R^{\frac{\gamma _{0}-2}{2}}}{\lambda \left(
R\right) }\right) ^{\frac{1}{2}}\left\| u\right\| \left| x\right| ^{-\frac{%
2N-2-\gamma _{0}}{4}}\quad \text{almost everywhere in }B_{R}.
\label{Pointwise0}
\end{equation}
\end{lem}

\proof%
The lemma is proved in \cite[Lemma 5]{Su-Wang-Will p} with a global
assumption on $V$, but, checking the proof, we see that the result actually
holds in the form given here.%
\endproof%

\begin{lem}
\label{Lem(Omega0)}Assume that there exists $R>0$ such that 
\begin{equation}
\Lambda _{\alpha ,\beta }\left( R\right) :=\esssup_{r\in \left(
0,R\right) }\frac{K\left( r\right) }{r^{\alpha }V\left( r\right) ^{\beta }}%
<+\infty \quad \text{for some }\frac{1}{2}\leq \beta \leq 1\text{~and }%
\alpha \in \mathbb{R}  \label{Lem(Omega0): hp1}
\end{equation}
and 
\[
\lambda \left( R\right) :=\essinf_{r\in \left( 0,R\right)
}r^{\gamma _{0}}V\left( r\right) >0\quad \text{for some }\gamma _{0}>2.
\]
Assume also that there exists $q>2\beta $ such that 
\[
\left( 2N-2-\gamma _{0}\right) q<4\alpha +4N-2\left( \gamma _{0}+2\right)
\beta .
\]
Then $\forall u\in H_{V,\mathrm{r}}^{1}\cap D_{0}^{1,2}\left( B_{R}\right) $
and $\forall h\in H_{V}^{1}$ one has 
\[
\int_{B_{R}}K\left( \left| x\right| \right) \left| u\right| ^{q-1}\left|
h\right| dx\leq c_{0}^{q-2\beta }a\left( R\right) R^{\frac{4\alpha
+4N-2\left( \gamma _{0}+2\right) \beta -\left( 2N-2-\gamma _{0}\right) q}{4}%
}\left\| u\right\| ^{q-1}\left\| h\right\| ,
\]
where $c_{0}$ is the constant of Lemma \ref{LEM(pointwise0)} and $a\left(
R\right) :=\Lambda _{\alpha ,\beta }\left( R\right) \left( \frac{1}{\sqrt{%
\lambda \left( R\right) }}+\frac{R^{\frac{\gamma _{0}-2}{2}}}{\lambda \left(
R\right) }\right) ^{\frac{q-2\beta }{2}}$.
\end{lem}

\proof%
Let $u\in H_{V,\mathrm{r}}^{1}\cap D_{0}^{1,2}\left( B_{R}\right) $ and $%
h\in H_{V}^{1}$. By assumption (\ref{Lem(Omega0): hp1}) and Lemma \ref
{LEM(pointwise0)}, we can apply Lemma \ref{Lem(Omega)} with $\Omega =B_{R}$, 
$\Lambda =\Lambda _{\alpha ,\beta }\left( R\right) $, $\nu =\frac{%
2N-2-\gamma _{0}}{4}$ and 
\[
m=c_{0}\left( \frac{1}{\sqrt{\lambda \left( R\right) }}+\frac{R^{\frac{%
\gamma _{0}-2}{2}}}{\lambda \left( R\right) }\right) ^{\frac{1}{2}}\left\|
u\right\| . 
\]
If $\frac{1}{2}\leq \beta <1$, we get 
\begin{eqnarray*}
\int_{B_{R}}K\left( \left| x\right| \right) \left| u\right| ^{q-1}\left|
h\right| dx &\leq &\Lambda m^{q-2\beta }\left( \int_{\Omega }\left| x\right|
^{\frac{\alpha -\nu \left( q-2\beta \right) }{1-\beta }}dx\right) ^{1-\beta
}\left\| u\right\| ^{2\beta -1}\left\| h\right\| \\
&=&c_{0}^{q-2\beta }a\left( R\right) \left( \int_{B_{R}}\left| x\right| ^{%
\frac{4\alpha -\left( 2N-2-\gamma _{0}\right) \left( q-2\beta \right) }{%
4\left( 1-\beta \right) }}dx\right) ^{1-\beta }\left\| u\right\|
^{q-1}\left\| h\right\| \\
&\leq &c_{0}^{q-2\beta }a\left( R\right) \left( R^{\frac{4\alpha -\left(
2N-2-\gamma _{0}\right) \left( q-2\beta \right) }{4\left( 1-\beta \right) }%
+N}\right) ^{1-\beta }\left\| u\right\| ^{q-1}\left\| h\right\| ,
\end{eqnarray*}
since 
\[
\frac{4\alpha -\left( 2N-2-\gamma _{0}\right) \left( q-2\beta \right) }{%
4\left( 1-\beta \right) }+N=\frac{4\alpha +4N-2\left( \gamma _{0}+2\right)
\beta -\left( 2N-2-\gamma _{0}\right) q}{4\left( 1-\beta \right) }>0. 
\]
If instead we have $\beta =1$, we get 
\begin{eqnarray*}
&&\int_{B_{R}}K\left( \left| x\right| \right) \left| u\right| ^{q-1}\left|
h\right| dx \\
&\leq &\Lambda m^{q-2}\left( \int_{\Omega }\left| x\right|
^{2\alpha -2\nu \left( q-2\right) }V\left( \left| x\right| \right) \left|
u\right| ^{2}dx\right) ^{\frac{1}{2}}\left\| h\right\| \\
&=&c_{0}^{q-2}a\left( R\right) \left( \int_{B_{R}}\left| x\right| ^{\frac{%
4\alpha -\left( 2N-2-\gamma _{0}\right) \left( q-2\right) }{2}}V\left(
\left| x\right| \right) \left| u\right| ^{2}dx\right) ^{\frac{1}{2}}\left\|
u\right\| ^{q-2}\left\| h\right\| \\
&\leq &c_{0}^{q-2}a\left( R\right) \left( R^{\frac{4\alpha -\left(
2N-2-\gamma _{0}\right) \left( q-2\right) }{2}}\int_{B_{R}}V\left( \left|
x\right| \right) \left| u\right| ^{2}dx\right) ^{\frac{1}{2}}\left\|
u\right\| ^{q-2}\left\| h\right\| \\
&\leq &c_{0}^{q-2}a\left( R\right) R^{\frac{4\alpha -\left( 2N-2-\gamma
_{0}\right) \left( q-2\right) }{4}}\left\| u\right\| ^{q-1}\left\| h\right\|
,
\end{eqnarray*}
since 
\[
4\alpha -\left( 2N-2-\gamma _{0}\right) \left( q-2\right) =4\alpha
+4N-2\left( \gamma _{0}+2\right) -\left( 2N-2-\gamma _{0}\right) q>0. 
\]
\endproof%

\begin{lem}
\label{Lem(zero0)}Assume that there exists $R_{1}>0$ such that (\ref{hp in 0}%
) and (\ref{stima in 0}) hold with $\gamma _{0}>2$ and let $q_{1}\in \mathbb{R}$
be such that $\left( \alpha _{0},q_{1}\right) \in \mathcal{A}_{\beta
_{0},\gamma _{0}}$. Then for every $0<R\leq R_{1}$ there exists $b\left(
R\right) >0$ such that $b\left( R\right) \rightarrow 0$ as $R\rightarrow
0^{+}$ and 
\[
\int_{B_{R}}K\left( \left| x\right| \right) \left| u\right| ^{q_{1}-1}\left|
h\right| dx\leq b\left( R\right) \left\| u\right\| ^{q_{1}-1}\left\|
h\right\|,
\quad 
\forall u\in H_{V,\mathrm{r}}^{1}\cap D_{0}^{1,2}\left( B_{R}\right),\ \forall h\in H_{V}^{1}.
\]
\end{lem}

\proof%
Denote 
\[
\Lambda _{0}:=\esssup_{r\in \left( 0,R_{1}\right) }\frac{K\left(
r\right) }{r^{\alpha _{0}}V\left( r\right) ^{\beta _{0}}}\quad \text{and}%
\quad \lambda _{0}:=\essinf_{r\in \left( 0,R_{1}\right) }r^{\gamma
_{0}}V\left( r\right) 
\]
and let $0<R\leq R_{1}$. Then 
\begin{equation}
\lambda \left( R\right) :=\essinf_{r\in \left( 0,R\right)
}r^{\gamma _{0}}V\left( r\right) \geq \lambda _{0}>0  \label{stimaETA1}
\end{equation}
and for every $\xi \geq 0$ we have 
\begin{eqnarray}
\Lambda _{\alpha _{0}+\xi \gamma _{0},\beta _{0}+\xi }\left( R\right) &:=&
\esssup_{r\in \left( 0,R\right) }\frac{K\left( r\right) }{r^{\alpha
_{0}+\xi \gamma _{0}}V\left( r\right) ^{\beta _{0}+\xi }}\leq 
\esssup_{r\in \left( 0,R_{1}\right) }\frac{K\left( r\right) }{r^{\alpha
_{0}}V\left( r\right) ^{\beta _{0}}\left( r^{\gamma _{0}}V\left( r\right)
\right) ^{\xi }}\nonumber \\
&\leq &\frac{\Lambda _{0}}{\lambda _{0}^{\xi }}<+\infty .
\label{stimaETA0}
\end{eqnarray}
Denoting $\alpha _{1}=\alpha _{1}\left( \beta _{0},\gamma _{0}\right) $, $%
\alpha _{2}=\alpha _{2}\left( \beta _{0}\right) $ and $\alpha _{3}=\alpha
_{3}\left( \beta _{0},\gamma _{0}\right) $, as defined in (\ref{alpha_i :=}%
), we will now distinguish five cases, which reflect the five definitions (%
\ref{A:=}) of the set $\mathcal{A}_{\beta _{0},\gamma _{0}}$. For the sake
of clarity, some computations will be displaced in the Appendix.\medskip

\noindent \emph{Case }$2<\gamma _{0}<N$\emph{.}\smallskip

\noindent In this case, $\left( \alpha _{0},q_{1}\right) \in \mathcal{A}_{\beta _{0},\gamma _{0}}$ means 
\[
\begin{tabular}{l}
$\alpha _{0}>\max \left\{ \alpha _{2},\alpha _{3}\right\}$\quad and \smallskip
\\
$\max \left\{ 1,2\beta _{0}\right\} <q_{1}<\displaystyle\min \left\{ 2\frac{\alpha
_{0}-\beta _{0}\gamma _{0}+N}{N-\gamma _{0}},2\frac{2\alpha _{0}+\left(
1-2\beta _{0}\right) \gamma _{0}+2N-2}{2N-2-\gamma _{0}}\right\}$\nonumber
\end{tabular}
\]
and these conditions ensure that we can fix $\xi \geq 0$, independent of $R$
(and $u$ and $h$), in such a way that $\alpha =\alpha _{0}+\xi \gamma _{0}$
and $\beta =\beta _{0}+\xi $ satisfy 
\begin{equation}
\frac{1}{2}\leq \beta \leq 1\quad \text{and}\quad 2\beta <q_{1}<\frac{%
4\alpha +4N-2\left( \gamma _{0}+2\right) \beta }{2N-2-\gamma _{0}}
\label{Lem(zero): cond1}
\end{equation}
(see Appendix). Hence, by (\ref{stimaETA0}) and (\ref{stimaETA1}), we can
apply Lemma \ref{Lem(Omega0)} (with $q=q_{1}$), so that $\forall u\in H_{V,%
\mathrm{r}}^{1}\cap D_{0}^{1,2}\left( B_{R}\right) $ and $\forall h\in
H_{V}^{1}$ we get 
\[
\int_{B_{R}}K\left( \left| x\right| \right) \left| u\right| ^{q_{1}-1}\left|
h\right| dx\leq c_{0}^{q_{1}-2\beta }a\left( R\right) R^{\frac{4\alpha
+4N-2\left( \gamma _{0}+2\right) \beta -\left( 2N-2-\gamma _{0}\right) q_{1}%
}{4}}\left\| u\right\| ^{q_{1}-1}\left\| h\right\| . 
\]
This gives the result, since $R^{4\alpha +4N-2\left( \gamma _{0}+2\right)
\beta -\left( 2N-2-\gamma _{0}\right) q_{1}}\rightarrow 0$ as $R\rightarrow
0^{+}$ and 
\[
a\left( R\right) =\Lambda _{\alpha _{0}+\xi \gamma _{0},\beta _{0}+\xi
}\left( R\right) \left( \frac{1}{\sqrt{\lambda \left( R\right) }}+\frac{R^{%
\frac{\gamma _{0}-2}{2}}}{\lambda \left( R\right) }\right) ^{\frac{%
q_{1}-2\beta }{2}}\leq \frac{\Lambda _{0}}{\lambda _{0}^{\xi }}\left( \frac{1%
}{\sqrt{\lambda _{0}}}+\frac{R_{1}^{\frac{\gamma _{0}-2}{2}}}{\lambda _{0}}%
\right) ^{\frac{q_{1}-2\beta }{2}}. 
\]

\noindent \emph{Case }$\gamma _{0}=N$. \smallskip

\noindent In this case, $\left( \alpha _{0},q_{1}\right) \in \mathcal{A}%
_{\beta _{0},\gamma _{0}}$ means 
\[
\alpha _{0}>\alpha _{1}\,\left( =\alpha _{2}=\alpha _{3}\right) \quad \text{%
and}\quad \max \left\{ 1,2\beta _{0}\right\} <q_{1}<2\frac{2\alpha
_{0}+\left( 1-2\beta _{0}\right) \gamma _{0}+2N-2}{2N-2-\gamma _{0}} 
\]
and these conditions still ensure that we can fix $\xi \geq 0$ in such a way
that $\alpha =\alpha _{0}+\xi \gamma _{0}$ and $\beta =\beta _{0}+\xi $
satisfy (\ref{Lem(zero): cond1}) (see Appendix), so that the result ensues
again by Lemma \ref{Lem(Omega0)}.\medskip

\noindent \emph{Case }$N<\gamma _{0}<2N-2$.\smallskip

\noindent In this case, $\left( \alpha _{0},q_{1}\right) \in \mathcal{A}%
_{\beta _{0},\gamma _{0}}$ means 
\[
\begin{tabular}{l}
$\alpha _{0}>\alpha _{1}$\quad and\smallskip\\
$\displaystyle\max \left\{ 1,2\beta _{0},2\frac{\alpha _{0}-\beta _{0}\gamma _{0}+N}{N-\gamma _{0}}\right\} <q_{1}
<\displaystyle 2\frac{2\alpha _{0}+\left( 1-2\beta _{0}\right) \gamma _{0}+2N-2}{2N-2-\gamma_{0}}$
\end{tabular}
\]
and the conclusion then follows as in the former cases (see
Appendix).\medskip

\noindent \emph{Case }$\gamma _{0}=2N-2$.\smallskip

\noindent In this case, $\left( \alpha _{0},q_{1}\right) \in \mathcal{A}%
_{\beta _{0},\gamma _{0}}$ means 
\[
\alpha _{0}>\alpha _{1}\quad \text{and}\quad \max \left\{ 1,2\beta _{0},2%
\frac{\alpha _{0}-\beta _{0}\gamma _{0}+N}{N-\gamma _{0}}\right\} <q_{1} 
\]
and these conditions ensure that we can fix $\xi \geq 0$ in such a way that $%
\alpha =\alpha _{0}+\xi \gamma _{0}$ and $\beta =\beta _{0}+\xi $ satisfy 
\[
\frac{1}{2}\leq \beta \leq 1,\quad q_{1}>2\beta \quad \text{and}\quad
0<2\alpha +2N-\left( \gamma _{0}+2\right) \beta 
\]
(see Appendix). The result then follows again from Lemma \ref{Lem(Omega0)}.\medskip

\noindent \emph{Case }$\gamma _{0}>2N-2$.\smallskip

\noindent In this case, $\left( \alpha _{0},q_{1}\right) \in \mathcal{A}%
_{\beta _{0},\gamma _{0}}$ means 
\[
q_{1}>\max \left\{ 1,2\beta _{0},2\frac{\alpha _{0}-\beta _{0}\gamma _{0}+N}{%
N-\gamma _{0}},2\frac{2\alpha _{0}+\left( 1-2\beta _{0}\right) \gamma
_{0}+2N-2}{2N-2-\gamma _{0}}\right\} 
\]
and this condition ensures that we can fix $\xi \geq 0$ in such a way that $%
\alpha =\alpha _{0}+\xi \gamma _{0}$ and $\beta =\beta _{0}+\xi $ satisfy 
\[
\frac{1}{2}\leq \beta \leq 1\quad \text{and}\quad q_{1}>\max \left\{ 2\beta
,2\frac{2\alpha +2N-\left( \gamma _{0}+2\right) \beta }{2N-2-\gamma _{0}}%
\right\} 
\]
(see Appendix). The result still follows from Lemma \ref{Lem(Omega0)}.%
\endproof%

\proof[Proof of Theorem \ref{THM3}]
Assume the hypotheses of the theorem and denote 
\[
\lambda _{0}:=\essinf_{r\in (0,R_{1})}r^{\gamma _{0}}V\left(
r\right) . 
\]
If $\gamma _{0}=2$ the thesis of the theorem is true by Theorem \ref{THM0}
(see Remark \ref{RMK: Hardy 2}.\ref{RMK: Hardy 2-improve}), whence we can
assume $\gamma _{0}>2$ without restriction. Let $u\in H_{V,\mathrm{r}}^{1}$
and $h\in H_{V}^{1}$ be such that $\left\| u\right\| =\left\| h\right\| =1$.
Let $R\leq R_{1}$ and observe that $\forall x\in B_{R}$ we have 
\begin{equation}
1\leq \frac{1}{\lambda _{0}}\left| x\right| ^{\gamma _{0}}V\left( \left|
x\right| \right) \leq \frac{R^{\gamma _{0}}}{\lambda _{0}}V\left( \left|
x\right| \right) .  \label{THM3-pf: magg}
\end{equation}
Take a function $\varphi \in C^{\infty }\left( \mathbb{R}\right) $ such that $%
0\leq \varphi \leq 1$ on $\mathbb{R}$, $\varphi \left( r\right) =1$ for $r\leq
1 $ and $\varphi \left( r\right) =0$ for $r\geq 2$. Define $M:=\max_{\mathbb{R}%
}\left| \varphi ^{\prime }\right| $ and $\phi _{R}\left( x\right) :=\varphi
\left( 2\left| x\right| /R\right) $, in such a way that $\left| \nabla \phi
_{R}\right| \leq 2M/R$ on $\mathbb{R}^{N}$. Then $\phi _{R}u\in H_{V,\mathrm{r}%
}^{1}\cap D_{0}^{1,2}\left( B_{R}\right) $ and 
{\allowdisplaybreaks
\begin{eqnarray*}
\left\| \phi _{R}u\right\| ^{2} &=&\int_{B_{R}}\left( \left| \nabla \left(
\phi _{R}u\right) \right| ^{2}+V\left( \left| x\right| \right) \left| \phi
_{R}u\right| ^{2}\right) dx\\
&\leq & \int_{B_{R}}\left( 2\left| \nabla u\right|
^{2}+2\left| \nabla \phi _{R}\right| ^{2}\left| u\right| ^{2}+V\left( \left|
x\right| \right) \left| u\right| ^{2}\right) dx \\
&\leq &2\int_{B_{R}}\left( \left| \nabla u\right| ^{2}+\frac{4M^{2}}{R^{2}}%
\left| u\right| ^{2}+V\left( \left| x\right| \right) \left| u\right|
^{2}\right) dx \\
&\leq &2\int_{B_{R}}\left( \left| \nabla u\right| ^{2}+4M^{2}\frac{R^{\gamma
_{0}-2}}{\lambda _{0}}V\left( \left| x\right| \right) \left| u\right|
^{2}+V\left( \left| x\right| \right) \left| u\right| ^{2}\right) dx\\
&\leq & 2\left( 1+4M^{2}\frac{R_{1}^{\gamma _{0}-2}}{\lambda _{0}}\right) ,
\end{eqnarray*}
}
where (\ref{THM3-pf: magg}) has been used. Hence, by Lemma \ref{Lem(zero0)},
we get 
\begin{eqnarray*}
\int_{B_{R/2}}K\left( \left| x\right| \right) \left| u\right|
^{q_{1}-1}\left| h\right| dx &\leq& \int_{B_{R}}K\left( \left| x\right| \right)
\left| \phi _{R}u\right| ^{q_{1}-1}\left| h\right| dx\leq b\left( R\right)
\left\| \phi _{R}u\right\| ^{q_{1}-1}\\
&\leq & b\left( R\right) \left( 2+8M^{2}\frac{R_{1}^{\gamma _{0}-2}}{\lambda _{0}}\right) 
\end{eqnarray*}
where $b\left( R\right) \rightarrow 0$ as $R\rightarrow 0^{+}$. We thus
conclude that $\displaystyle \lim_{R\rightarrow 0^{+}}\mathcal{R}_{0}\left(
q_{1},\frac{R}{2}\right) =0$, which is equivalent to the thesis of the
theorem.%
\endproof%

\section{Appendix}

This Appendix is devoted to complete the computations of the proof of Lemma 
\ref{Lem(zero0)}. We still distinguish the same cases considered there.\medskip

\noindent \emph{Case }$2<\gamma _{0}<N$.\smallskip

\noindent In this case, $\left( \alpha _{0},q_{1}\right) \in \mathcal{A}%
_{\beta _{0},\gamma _{0}}$ means 
\[
\begin{tabular}{l}
$\alpha _{0}>\max \left\{ \alpha _{2},\alpha _{3}\right\}$ \quad and \smallskip\\
$\max \left\{ 1,2\beta _{0}\right\} <q_{1}<\displaystyle\min \left\{ 2\frac{\alpha
_{0}-\beta _{0}\gamma _{0}+N}{N-\gamma _{0}},2\frac{2\alpha _{0}+\left(
1-2\beta _{0}\right) \gamma _{0}+2N-2}{2N-2-\gamma _{0}}\right\}.$
\end{tabular}
\]
This ensures that we can find $\xi \geq 0$ such that 
\[
\frac{1}{2}\leq \beta _{0}+\xi \leq 1\quad \text{and}\quad 2\left( \beta
_{0}+\xi \right) <q_{1}<\frac{4\left( \alpha _{0}+\xi \gamma _{0}\right)
+4N-2\left( \gamma _{0}+2\right) \left( \beta _{0}+\xi \right) }{2N-2-\gamma
_{0}}, 
\]
i.e., 
\[
\frac{1}{2}-\beta _{0}\leq \xi \leq 1-\beta _{0}\quad \text{and}\quad 2\beta
_{0}+2\xi <q_{1}<2\frac{\gamma _{0}-2}{2N-2-\gamma _{0}}\xi +\frac{4\alpha
_{0}+4N-2\left( \gamma _{0}+2\right) \beta _{0}}{2N-2-\gamma _{0}}. 
\]
Indeed, this amounts to find $\xi $ such that 
\[
\left\{ 
\begin{array}{l}
\max \left\{ 0,\frac{1-2\beta _{0}}{2}\right\} \leq \xi \leq 1-\beta
_{0}\medskip \\ 
\xi <\frac{q_{1}-2\beta _{0}}{2}\medskip \\ 
q_{1}-\frac{4\alpha _{0}+4N-2\left( \gamma _{0}+2\right) \beta _{0}}{%
2N-2-\gamma _{0}}<2\frac{\gamma _{0}-2}{2N-2-\gamma _{0}}\xi ,
\end{array}
\right. 
\]
which, since $\frac{\gamma _{0}-2}{2N-2-\gamma _{0}}>0$, is equivalent to 
\[
\left\{ 
\begin{array}{l}
\max \left\{ 0,\frac{1-2\beta _{0}}{2}\right\} \leq \xi \leq 1-\beta
_{0}\medskip \\ 
q_{1}\frac{2N-2-\gamma _{0}}{2\left( \gamma _{0}-2\right) }-\frac{2\alpha
_{0}+2N-\left( \gamma _{0}+2\right) \beta _{0}}{\gamma _{0}-2}<\xi <\frac{%
q_{1}-2\beta _{0}}{2}.
\end{array}
\right. 
\]
Since $\frac{1}{2}-\beta _{0}\leq 1-\beta _{0}$ is obvious and $1-\beta
_{0}\geq 0$ holds by assumption, such a system has a solution $\xi $ if and
only if 
\[
\left\{ 
\begin{array}{l}
\frac{1-2\beta _{0}}{2}<\frac{q_{1}-2\beta _{0}}{2}\medskip \\ 
q_{1}\frac{2N-2-\gamma _{0}}{2\left( \gamma _{0}-2\right) }-\frac{2\alpha
_{0}+2N-\left( \gamma _{0}+2\right) \beta _{0}}{\gamma _{0}-2}<1-\beta
_{0}\medskip \\ 
q_{1}\frac{2N-2-\gamma _{0}}{2\left( \gamma _{0}-2\right) }-\frac{2\alpha
_{0}+2N-\left( \gamma _{0}+2\right) \beta _{0}}{\gamma _{0}-2}<\frac{%
q_{1}-2\beta _{0}}{2}\medskip \\ 
\frac{q_{1}-2\beta _{0}}{2}>0,
\end{array}
\right. 
\]
which is equivalent to 
\[
\left\{ 
\begin{array}{l}
1<q_{1}\medskip \\ 
q_{1}<2\frac{2\alpha _{0}+2N+\left( 1-2\beta _{0}\right) \gamma _{0}-2}{%
2N-2-\gamma _{0}}\medskip \\ 
\frac{q_{1}}{2}<\frac{\alpha _{0}+N-\gamma _{0}\beta _{0}}{N-\gamma _{0}}%
\medskip \\ 
q_{1}>2\beta _{0}.
\end{array}
\right. 
\]

\noindent \emph{Case }$\gamma _{0}=N$.\smallskip

\noindent In this case, $\left( \alpha _{0},q_{1}\right) \in \mathcal{A}%
_{\beta _{0},\gamma _{0}}$ means 
\[
\begin{tabular}{l}
$\alpha _{0}>\alpha _{1}\,\left( =\alpha _{2}=\alpha _{3}\right)$\quad and \smallskip\\
$\max \left\{ 1,2\beta _{0}\right\} <q_{1}<\displaystyle
2\frac{2\alpha_{0}+\left( 1-2\beta _{0}\right) \gamma _{0}+2N-2}{2N-2-\gamma _{0}}=2\frac{%
2\alpha _{0}+3N-2\beta _{0}N-2}{N-2}$ 
\end{tabular}
\]
and this ensures that we can find $\xi \geq 0$ such that 
\[
\frac{1}{2}\leq \beta _{0}+\xi \leq 1\quad \text{and}\quad 2\left( \beta
_{0}+\xi \right) <q_{1}<\frac{4\left( \alpha _{0}+\xi \gamma _{0}\right)
+4N-2\left( \gamma _{0}+2\right) \left( \beta _{0}+\xi \right) }{2N-2-\gamma
_{0}}, 
\]
i.e., 
\[
\frac{1}{2}-\beta _{0}\leq \xi \leq 1-\beta _{0}\quad \text{and}\quad 2\beta
_{0}+2\xi <q_{1}<2\xi +\frac{4\alpha _{0}+4N-2\left( N+2\right) \beta _{0}}{%
N-2}. 
\]
Indeed, this amounts to find $\xi $ such that 
\[
\left\{ 
\begin{array}{l}
\max \left\{ 0,\frac{1-2\beta _{0}}{2}\right\} \leq \xi \leq 1-\beta
_{0}\medskip \\ 
\frac{q_{1}}{2}-\frac{2\alpha _{0}+2N-\left( N+2\right) \beta _{0}}{N-2}<\xi
<\frac{q_{1}-2\beta _{0}}{2},
\end{array}
\right. 
\]
which has a solution $\xi $ if and only if 
\[
\left\{ 
\begin{array}{l}
\frac{1-2\beta _{0}}{2}<\frac{q_{1}-2\beta _{0}}{2}\medskip \\ 
\frac{q_{1}}{2}-\frac{2\alpha _{0}+2N-\left( N+2\right) \beta _{0}}{N-2}%
<1-\beta _{0}\medskip \\ 
\frac{q_{1}}{2}-\frac{2\alpha _{0}+2N-\left( N+2\right) \beta _{0}}{N-2}<%
\frac{q_{1}-2\beta _{0}}{2}\medskip \\ 
0<\frac{q_{1}-2\beta _{0}}{2}.
\end{array}
\right. 
\]
These conditions are equivalent to 
\[
\left\{ 
\begin{array}{l}
1<q_{1}\medskip \\ 
\frac{q_{1}}{2}<\frac{2\alpha _{0}+3N-2-2\beta _{0}N}{N-2}\medskip \\ 
0<\frac{2\alpha _{0}+2N-\left( N+2\right) \beta _{0}}{N-2}-\beta _{0}=2\frac{%
\alpha _{0}+N\left( 1-\beta _{0}\right) }{N-2}=2\frac{\alpha _{0}-\alpha _{1}%
}{N-2}\medskip \\ 
2\beta _{0}<q_{1}.
\end{array}
\right. 
\]

\noindent \emph{Case }$N<\gamma _{0}<2N-2$.\smallskip 

\noindent In this case, $\left( \alpha _{0},q_{1}\right) \in \mathcal{A}%
_{\beta _{0},\gamma _{0}}$ means 
\[
\begin{tabular}{l}
$\alpha _{0}>\alpha _{1}$\quad and \smallskip\\
$\displaystyle\max \left\{ 1,2\beta _{0},2\frac{\alpha _{0}-\beta _{0}\gamma _{0}+N}{N-\gamma _{0}}\right\} <q_{1}<2%
\frac{2\alpha _{0}+\left( 1-2\beta _{0}\right) \gamma _{0}+2N-2}{2N-2-\gamma_{0}}$
\end{tabular}
\]
and these conditions ensure that we can find $\xi \geq 0$ such that 
\[
\frac{1}{2}\leq \beta _{0}+\xi \leq 1\quad \text{and}\quad 2\left( \beta
_{0}+\xi \right) <q_{1}<\frac{4\left( \alpha _{0}+\xi \gamma _{0}\right)
+4N-2\left( \gamma _{0}+2\right) \left( \beta _{0}+\xi \right) }{2N-2-\gamma
_{0}}, 
\]
i.e., 
\[
\frac{1}{2}-\beta _{0}\leq \xi \leq 1-\beta _{0}\quad \text{and}\quad 2\beta
_{0}+2\xi <q_{1}<2\frac{\gamma _{0}-2}{2N-2-\gamma _{0}}\xi +\frac{4\alpha
_{0}+4N-2\left( \gamma _{0}+2\right) \beta _{0}}{2N-2-\gamma _{0}}. 
\]
Indeed, this is equivalent to find $\xi $ such that 
\[
\left\{ 
\begin{array}{l}
\max \left\{ 0,\frac{1-2\beta _{0}}{2}\right\} \leq \xi \leq 1-\beta
_{0}\medskip \\ 
\xi <\frac{q_{1}-2\beta _{0}}{2}\medskip \\ 
q_{1}-\frac{4\alpha _{0}+4N-2\left( \gamma _{0}+2\right) \beta _{0}}{%
2N-2-\gamma _{0}}<2\frac{\gamma _{0}-2}{2N-2-\gamma _{0}}\xi ,
\end{array}
\right. 
\]
which, since $\frac{\gamma _{0}-2}{2N-2-\gamma _{0}}>0$, amounts to 
\[
\left\{ 
\begin{array}{l}
\max \left\{ 0,\frac{1-2\beta _{0}}{2}\right\} \leq \xi \leq 1-\beta
_{0}\medskip \\ 
\xi <\frac{q_{1}-2\beta _{0}}{2}\medskip \\ 
q_{1}\frac{2N-2-\gamma _{0}}{2\left( \gamma _{0}-2\right) }-\frac{2\alpha
_{0}+2N-\left( \gamma _{0}+2\right) \beta _{0}}{\gamma _{0}-2}<\xi .
\end{array}
\right. 
\]
Such a system has a solution $\xi $ if and only if 
\[
\left\{ 
\begin{array}{l}
0<\frac{q_{1}-2\beta _{0}}{2}\medskip \\ 
\frac{1-2\beta _{0}}{2}<\frac{q_{1}-2\beta _{0}}{2}\medskip \\ 
q_{1}\frac{2N-2-\gamma _{0}}{2\left( \gamma _{0}-2\right) }-\frac{2\alpha
_{0}+2N-\left( \gamma _{0}+2\right) \beta _{0}}{\gamma _{0}-2}<1-\beta
_{0}\medskip \\ 
q_{1}\frac{2N-2-\gamma _{0}}{2\left( \gamma _{0}-2\right) }-\frac{2\alpha
_{0}+2N-\left( \gamma _{0}+2\right) \beta _{0}}{\gamma _{0}-2}<\frac{%
q_{1}-2\beta _{0}}{2},
\end{array}
\right. 
\]
which is equivalent to 
\[
\left\{ 
\begin{array}{l}
q_{1}>2\beta _{0}\medskip \\ 
q_{1}>1\medskip \\ 
q_{1}<2\frac{2\alpha _{0}+2N-2\beta _{0}\gamma _{0}+\gamma _{0}-2}{%
2N-2-\gamma _{0}}\medskip \\ 
-2\frac{\alpha _{0}+N-\beta _{0}\gamma _{0}}{\gamma _{0}-2}<q_{1}\frac{%
\gamma _{0}-N}{\gamma _{0}-2}.
\end{array}
\right. 
\]

\noindent \emph{Case }$\gamma _{0}=2N-2$\emph{.\smallskip }

\noindent In this case, $\left( \alpha _{0},q_{1}\right) \in \mathcal{A}%
_{\beta _{0},\gamma _{0}}$ means 
\[
\alpha _{0}>\alpha _{1}\quad \text{and}\quad q_{1}>\max \left\{ 1,2\beta
_{0},-2\frac{\alpha _{0}-2\left( N-1\right) \beta _{0}+N}{N-2}\right\} . 
\]
This ensures that we can find $\xi \geq 0$ such that 
\[
\frac{1}{2}\leq \beta _{0}+\xi \leq 1,\quad q_{1}>2\left( \beta _{0}+\xi
\right) \quad \text{and}\quad 2\left( \alpha _{0}+\xi \gamma _{0}\right)
+2N-\left( \gamma _{0}+2\right) \left( \beta _{0}+\xi \right) >0, 
\]
i.e., 
\[
\frac{1}{2}-\beta _{0}\leq \xi \leq 1-\beta _{0},\quad q_{1}>2\beta
_{0}+2\xi \quad \text{and}\quad \alpha _{0}+N\left( 1-\beta _{0}\right)
+\left( N-2\right) \xi >0. 
\]
Indeed, this amounts to find $\xi $ such that 
\[
\left\{ 
\begin{array}{l}
\max \left\{ 0,\frac{1-2\beta _{0}}{2}\right\} \leq \xi \leq 1-\beta
_{0}\medskip \\ 
-\frac{\alpha _{0}+N\left( 1-\beta _{0}\right) }{N-2}<\xi <\frac{%
q_{1}-2\beta _{0}}{2}
\end{array}
\right. 
\]
and such a system has a solution $\xi $ if and only if 
\[
\left\{ 
\begin{array}{l}
\frac{1-2\beta _{0}}{2}<\frac{q_{1}-2\beta _{0}}{2}\medskip \\ 
-\frac{\alpha _{0}+N\left( 1-\beta _{0}\right) }{N-2}<1-\beta _{0}\medskip
\\ 
-\frac{\alpha _{0}+N\left( 1-\beta _{0}\right) }{N-2}<\frac{q_{1}-2\beta _{0}%
}{2}\medskip \\ 
0<\frac{q_{1}-2\beta _{0}}{2},
\end{array}
\right. 
\]
which means 
\[
\left\{ 
\begin{array}{l}
1<q_{1}\medskip \\ 
\alpha _{0}>-\left( 2N-2\right) \left( 1-\beta _{0}\right) =\alpha
_{1}\medskip \\ 
\frac{q_{1}}{2}>\beta _{0}-\frac{\alpha _{0}+N\left( 1-\beta _{0}\right) }{%
N-2}=-\frac{\alpha _{0}+N-2\left( N-1\right) \beta _{0}}{N-2}\medskip \\ 
2\beta _{0}<q_{1}.
\end{array}
\right. 
\]

\noindent \emph{Case }$\gamma _{0}>2N-2$\emph{.\smallskip }

\noindent In this case, $\left( \alpha _{0},q_{1}\right) \in \mathcal{A}%
_{\beta _{0},\gamma _{0}}$ means 
\[
q_{1}>\max \left\{ 1,2\beta _{0},2\frac{\alpha _{0}-\beta _{0}\gamma _{0}+N}{%
N-\gamma _{0}},2\frac{2\alpha _{0}+\left( 1-2\beta _{0}\right) \gamma
_{0}+2N-2}{2N-2-\gamma _{0}}\right\} 
\]
and this condition ensures that we can find $\xi \geq 0$ such that 
\[
\frac{1}{2}\leq \beta _{0}+\xi \leq 1\quad \text{and}\quad q_{1}>2\max
\left\{ \beta _{0}+\xi ,\frac{2\left( \alpha _{0}+\xi \gamma _{0}\right)
+2N-\left( \gamma _{0}+2\right) \left( \beta _{0}+\xi \right) }{2N-2-\gamma
_{0}}\right\} , 
\]
which amounts to find $\xi $ such that 
\begin{equation}
\left\{ 
\begin{array}{l}
\max \left\{ 0,\frac{1-2\beta _{0}}{2}\right\} \leq \xi \leq 1-\beta
_{0}\medskip \\ 
\frac{q_{1}}{2}>\max \left\{ \beta _{0}+\xi ,\frac{\gamma _{0}-2}{%
2N-2-\gamma _{0}}\xi +\frac{2\alpha _{0}+2N-\left( \gamma _{0}+2\right)
\beta _{0}}{2N-2-\gamma _{0}}\right\} .
\end{array}
\right.  \label{ineq}
\end{equation}
In order to check this, we take into account that $\gamma _{0}>2N-2$ implies 
$\gamma _{0}>N$, and observe that 
\[
\beta _{0}+\xi =\frac{\gamma _{0}-2}{2N-2-\gamma _{0}}\xi +\frac{2\alpha
_{0}+2N-\left( \gamma _{0}+2\right) \beta _{0}}{2N-2-\gamma _{0}}%
\Longleftrightarrow \xi =\frac{\alpha _{0}+\left( 1-\beta _{0}\right) N}{%
N-\gamma _{0}}. 
\]
Accordingly, we distinguish three subcases:$\medskip $

\noindent (I) $\frac{\alpha _{0}+\left( 1-\beta _{0}\right) N}{N-\gamma _{0}}%
\geq 1-\beta _{0}$, i.e., $\alpha _{0}\leq -\gamma _{0}\left( 1-\beta
_{0}\right) =\alpha _{1}$;$\medskip $

\noindent (II) $\frac{\alpha _{0}+\left( 1-\beta _{0}\right) N}{N-\gamma _{0}%
}\leq \max \left\{ 0,\frac{1-2\beta _{0}}{2}\right\} $, i.e., 
\[
\alpha _{0}+\left( 1-\beta _{0}\right) N\geq \left( N-\gamma _{0}\right)
\max \left\{ 0,\frac{1-2\beta _{0}}{2}\right\} =\min \left\{ 0,\left(
N-\gamma _{0}\right) \frac{1-2\beta _{0}}{2}\right\} , 
\]
i.e., 
\[
\alpha _{0}\geq \min \left\{ 0,\left( N-\gamma _{0}\right) \frac{1-2\beta
_{0}}{2}\right\} -\left( 1-\beta _{0}\right) N=\min \left\{ \alpha
_{2},\alpha _{3}\right\} ; 
\]

\noindent (III) $\max \left\{ 0,\frac{1-2\beta _{0}}{2}\right\} <\frac{%
\alpha _{0}+\left( 1-\beta _{0}\right) N}{N-\gamma _{0}}<1-\beta _{0}$,
i.e., $\alpha _{1}<\alpha _{0}<\min \left\{ \alpha _{2},\alpha _{3}\right\}
.\medskip $

\noindent \emph{Subcase (I).\smallskip }

\noindent Since $\xi \leq 1-\beta _{0}$ implies 
$
\max \left\{ \beta _{0}+\xi ,\frac{\gamma _{0}-2}{2N-2-\gamma _{0}}\xi +%
\frac{2\alpha _{0}+2N-\left( \gamma _{0}+2\right) \beta _{0}}{2N-2-\gamma
_{0}}\right\}=\frac{\gamma _{0}-2}{2N-2-\gamma _{0}}\xi +\frac{2\alpha
_{0}+2N-\left( \gamma _{0}+2\right) \beta _{0}}{2N-2-\gamma _{0}}, 
$
the inequalities (\ref{ineq}) become 
\[
\left\{ 
\begin{array}{l}
\max \left\{ 0,\frac{1-2\beta _{0}}{2}\right\} \leq \xi \leq 1-\beta
_{0}\medskip \\ 
\frac{q_{1}}{2}>\frac{\gamma _{0}-2}{2N-2-\gamma _{0}}\xi +\frac{2\alpha
_{0}+2N-\left( \gamma _{0}+2\right) \beta _{0}}{2N-2-\gamma _{0}},
\end{array}
\right. 
\]
i.e., 
\[
\left\{ 
\begin{array}{l}
\max \left\{ 0,\frac{1-2\beta _{0}}{2}\right\} \leq \xi \leq 1-\beta
_{0}\medskip \\ 
q_{1}\frac{2N-2-\gamma _{0}}{2\left( \gamma _{0}-2\right) }-\frac{4\alpha
_{0}+4N-2\left( \gamma _{0}+2\right) \beta _{0}}{2\left( \gamma
_{0}-2\right) }<\xi ,
\end{array}
\right. 
\]
which, since $\max \left\{ 0,\frac{1-2\beta _{0}}{2}\right\} \leq 1-\beta
_{0}$ is clearly true, has a solution $\xi $ if and only if 
\[
q_{1}\frac{2N-2-\gamma _{0}}{2\left( \gamma _{0}-2\right) }-\frac{4\alpha
_{0}+4N-2\left( \gamma _{0}+2\right) \beta _{0}}{2\left( \gamma
_{0}-2\right) }<1-\beta _{0}, 
\]
i.e., 
\[
q_{1}>\frac{4\alpha _{0}+4N-2\left( \gamma _{0}+2\right) \beta _{0}+2\left(
\gamma _{0}-2\right) \left( 1-\beta _{0}\right) }{2N-2-\gamma _{0}}=2\frac{%
2\alpha _{0}+2N-2+\gamma _{0}\left( 1-2\beta _{0}\right) }{2N-2-\gamma _{0}}%
. 
\]

\noindent \emph{Subcase (II).\smallskip }

\noindent Since $\xi \geq \max \left\{ 0,\frac{1-2\beta _{0}}{2}\right\} $
implies $\max \left\{ \beta _{0}+\xi ,\frac{\gamma _{0}-2}{2N-2-\gamma _{0}}%
\xi +\frac{2\alpha _{0}+2N-\left( \gamma _{0}+2\right) \beta _{0}}{%
2N-2-\gamma _{0}}\right\} =\beta _{0}+\xi $, the inequalities (\ref{ineq})
become 
\[
\left\{ 
\begin{array}{l}
\max \left\{ 0,\frac{1-2\beta _{0}}{2}\right\} \leq \xi \leq 1-\beta
_{0}\medskip \\ 
\xi <\frac{q_{1}}{2}-\beta _{0},
\end{array}
\right. 
\]
which has a solution $\xi $ if and only if $\max \left\{ 0,\frac{1-2\beta
_{0}}{2}\right\} \leq \frac{q_{1}}{2}-\beta _{0}$, i.e., $q_{1}>\max \left\{
1,2\beta _{0}\right\} $.$\medskip $

\noindent \emph{Subcase (III).\smallskip }

\noindent We take $\xi =\frac{\alpha _{0}+\left( 1-\beta _{0}\right) N}{%
N-\gamma _{0}}$ and thus the inequalities (\ref{ineq}) are equivalent just
to 
\[
\frac{q_{1}}{2}>\beta _{0}+\frac{\alpha _{0}+\left( 1-\beta _{0}\right) N}{%
N-\gamma _{0}}=\frac{\alpha _{0}+N-\gamma _{0}\beta _{0}}{N-\gamma _{0}}. 
\]


\begin{thebibliography}{SK}

\normalsize
\baselineskip=17pt

\bibitem{Alves-Souto-13}
Alves, C.O., Souto, M.A.S.:
Existence of solutions for a class of nonlinear Schr\"{o}dinger equations
with potential vanishing at infinity.
J. Differential Equations \textbf{254}, 1977-1991 (2013)

\bibitem{Ambr-Fel-Malch}
Ambrosetti, A., Felli, V., Malchiodi, A.:
Ground states of nonlinear Schr\"{o}dinger equations
with potentials vanishing at infinity.
J. Eur. Math. Soc. \textbf{7}, 117-144 (2005)

\bibitem{Ambr-Rab}
Ambrosetti, A., Rabinowitz, P.H.:
Dual variational methods in critical point theory and applications.
J. Funct. Anal. \textbf{14}, 349-381 (1973)

\bibitem{BBR1}
Badiale, M., Benci, V., Rolando, S.:
Solitary waves: physical aspects and mathematical results.
Rend. Sem. Math. Univ. Pol. Torino \textbf{62}, 107-154 (2004)

\bibitem{BGR}
Badiale, M., Guida, M., Rolando, S.:
Elliptic equations with decaying cylindrical potentials and power-type nonlinearities
Adv. Differential Equations \textbf{12}, 1321-1362 (2007)

\bibitem{BGRnonex}
Badiale, M., Guida, M., Rolando, S.:
A nonexistence result for a nonlinear elliptic equation with singular and decaying potential.
Commun. Contemp. Math., to appear

\bibitem{BGRnext}
Badiale, M., Guida, M., Rolando, S.:
Compactness and existence results in weighted Sobolev spaces of
radial functions, Part II: Existence, in preparation

\bibitem{BPR}
Badiale, M., Pisani, L., Rolando, S.:
Sum of weighted Lebesgue spaces and nonlinear elliptic equations.
NoDEA, Nonlinear Differ. Equ. Appl. \textbf{18}, 369-405 (2011)

\bibitem{BR}
Badiale, M., Rolando, S.:
Elliptic problems with singular potentials and double-power nonlinearity.
Mediterr. J. Math. \textbf{2}, 417-436 (2005)

\bibitem{BRpow}
Badiale, M., Rolando, S.:
A note on nonlinear elliptic problems with singular potentials.
Rend. Lincei Mat. Appl. \textbf{17}, 1-13 (2006)

\bibitem{Be09}
Benci, V.:
Hylomorphic solitons.
Milan J. Math. \textbf{77}, 271-322 (2009)

\bibitem{Be-F.2}
Benci, V., Fortunato, D.:
Towards a unified field theory for classical electrodynamics.
Arch. Rational Mech. Anal. \textbf{173}, 379-414 (2004)

\bibitem{Be-Gr-Mic}
Benci, V., Grisanti, C.R., Micheletti, A.M.:
Existence and non existence of the ground state solution for the nonlinear Schr\"{o}dinger equation
with $V\left( \infty\right) =0$.
Topol. Methods Nonlinear Anal. \textbf{26}, 203-220 (2005)

\bibitem{Be-Gr-Mic 2}
Benci, V., Grisanti, C.R., Micheletti, A.M.:
Existence of solutions for the nonlinear Schr\"{o}dinger equation with $V\left( \infty \right) =0$.
Contributions to nonlinear analysis, Progr. Nonlinear Differential Equations Appl., vol. 66,
Birkh\"{a}user, Basel (2006)

\bibitem{Beres-Lions}
Berestycki, H., Lions, P.L.:
Nonlinear Scalar Field Equations, I - II.
Arch. Rational Mech. Anal. \textbf{82}, 313-379 (1983)

\bibitem{BonMerc11}
Bonheure, D., Mercuri, C.:
Embedding theorems and existence results for nonlinear Schr\"{o}dinger-Poisson systems
with unbounded and vanishing potentials.
J. Differential Equations \textbf{251}, 1056-1085 (2011)

\bibitem{Bon-VanSchaft-10}
Bonheure, D., Van Schaftingen, J.:
Groundstates for the nonlinear Schr\"{o}dinger equation with potential vanishing at infinity.
Ann. Mat. Pura Appl. \textbf{189}, 273-301 (2010)

\bibitem{Floer-Wein}
Floer, A., Weinstein, A.:
Nonspreading wave packets for the cubic Schr\"{o}dinger equation with a bounded potential.
J.\ Funct. Anal. \textbf{69}, 397-408 (1986)

\bibitem{GR-boundedPS}
Guida, M., Rolando, S.:
On the existence of bounded Palais-Smale sequences and applications to nonlinear equations
without superlinearity assumptions, in preparation

\bibitem{Ni}
Ni, W.-M.:
A nonlinear Dirichlet problem on the unit ball and its applications.
Indiana Univ. Math. J. \textbf{31}, 801-807 (1982)

\bibitem{Palais}
Palais, R.S.:
The principle of symmetric criticality.
Commun. Math. Phys. \textbf{69}, 19-30 (1979)

\bibitem{Rabi92}
Rabinowitz, P.H.:
On a class of nonlinear Schr\"{o}dinger equations.
Z. Angew. Math. Phys. \textbf{43}, 270-291 (1992)

\bibitem{Strauss}
Strauss, W.A.:
Existence of solitary waves in higher dimensions.
Comm. Math. Phys. \textbf{55}, 149-172 (1977)

\bibitem{SuTian12}
Su, J., Tian, R.:
Weighted Sobolev type embeddings and coercive quasilinear elliptic equations on $\mathbb{R}^{N}$.
Proc. Amer. Math. Soc. \textbf{140}, 891-903 (2012)

\bibitem{Su-Wang-Will 2}
Su, J., Wang, Z.Q., Willem, M.:
Nonlinear Schr\"{o}dinger equations with unbounded and decaying potentials.
Commun. Contemp. Math. \textbf{9}, 571-583 (2007)

\bibitem{Su-Wang-Will p}
Su, J., Wang, Z.-Q., Willem, M.:
Weighted Sobolev embedding with unbounded and decaying radial potentials.
J. Differential Equations \textbf{238}, 201-219 (2007)

\bibitem{YangY}
Yang, Y.:
Solitons in field theory and nonlinear analysis.
Springer Monographs in Mathematics, Springer-Verlag, New York (2001)

\end{thebibliography}
\end{document}